\def\todaysdate{3\textsuperscript{rd} March 2026}
\documentclass[reqno,a4paper]{article}
\usepackage{etex}
\usepackage[T1]{fontenc}
\usepackage[utf8]{inputenc}
\usepackage{lmodern}

\usepackage[style=alphabetic,giveninits=true,backref,backend=biber,isbn=false,url=false,maxbibnames=99,maxalphanames=4]{biblatex}

\renewbibmacro{in:}{}
\newbibmacro{string+doi}[1]{\iffieldundef{doi}{\iffieldundef{url}{#1}{\href{\thefield{url}}{#1}}}{\href{http://dx.doi.org/\thefield{doi}}{#1}}}
\DeclareFieldFormat*{title}{\usebibmacro{string+doi}{\emph{#1}}}
\DefineBibliographyStrings{english}{%
  backrefpage = {$\uparrow$},
  backrefpages = {$\uparrow$},
}

\usepackage{amssymb,amsmath,amsthm}
\usepackage{thmtools,xcolor}
\usepackage{units}
\usepackage{graphicx}
\usepackage[all]{xy}
\usepackage{tikz}
\usetikzlibrary{arrows,calc,positioning,decorations.pathreplacing}
\usepackage{tikz-cd}
\usepackage{relsize}
\usepackage{mhsetup}
\usepackage{mathtools}
\usepackage{stmaryrd}
\usepackage{tocloft}
\usepackage{paralist} 
\usepackage[left=3cm,right=3cm,top=3cm,bottom=3cm]{geometry} 
\usepackage[bottom]{footmisc}
\usepackage[colorlinks,citecolor=blue,linkcolor=blue,urlcolor=blue,filecolor=blue,breaklinks,hypertexnames=false]{hyperref}
\usepackage{footnotebackref}
\usepackage[format=plain,indention=1em,labelsep=quad,labelfont={up},textfont={footnotesize},margin=2em]{caption}
\usepackage[mathscr]{euscript}
\usepackage{accents}
\usepackage{xparse}
\usepackage{array}
\usepackage{multirow}
\usepackage{enumitem}

\definecolor{lightblue}{rgb}{0.8,0.8,1}

\numberwithin{equation}{section}
\numberwithin{figure}{section}

\definecolor{vdarkred}{rgb}{0.7,0,0}
\declaretheoremstyle[
  spaceabove=\topsep,
  spacebelow=\topsep,
  headpunct=,
  numbered=no,
  postheadspace=1ex,
  headfont=\color{vdarkred}\normalfont\bfseries,
  bodyfont=\normalfont\itshape,
]{colored}
\declaretheoremstyle[
  spaceabove=\topsep,
  spacebelow=\topsep,
  headpunct=,
  numbered=no,
  postheadspace=1ex,
  headfont=\normalfont\bfseries,
  bodyfont=\normalfont\itshape,
]{italic}
\declaretheoremstyle[
  spaceabove=\topsep,
  spacebelow=\topsep,
  headpunct=,
  numbered=no,
  postheadspace=1ex,
  headfont=\normalfont\bfseries,
  bodyfont=\normalfont\upshape,
]{upright}

\declaretheorem[style=italic,name=Theorem,numbered=yes,numberwithin=section]{thm}
\declaretheorem[style=italic,name=Lemma,numbered=yes,numberlike=thm]{lem}
\declaretheorem[style=italic,name=Proposition,numbered=yes,numberlike=thm]{prop}
\declaretheorem[style=italic,name=Corollary,numbered=yes,numberlike=thm]{coro}

\declaretheorem[style=italic,name=Theorem,numbered=yes,numberwithin=section]{athm}

\declaretheorem[style=upright,name=Definition,numbered=yes,numberlike=thm]{defn}
\declaretheorem[style=upright,name=Remark,numbered=yes,numberlike=thm]{rmk}
\declaretheorem[style=upright,name=Example,numbered=yes,numberlike=thm]{eg}

\declaretheorem[style=upright,name=Notation,numbered=yes,numberlike=thm]{notation}
\declaretheorem[style=upright,name=Convention,numbered=yes,numberlike=thm]{convention}

\makeatletter
\renewcommand*{\@seccntformat}[1]{\upshape\csname the#1\endcsname.\hspace{1ex}}
\renewcommand*{\section}{\@startsection{section}{1}{\z@}%
	{2.5ex \@plus 1ex \@minus 0.2ex}%
	{1.5ex \@plus 0.2ex}%
	{\normalfont\Large\bfseries}}
\renewcommand*{\subsection}{\@startsection{subsection}{2}{\z@}%
	{2.5ex \@plus 1ex \@minus 0.2ex}%
	{1.5ex \@plus 0.2ex}%
	{\normalfont\large\bfseries}}
\renewcommand*{\subsubsection}{\@startsection{subsubsection}{3}{\z@}%
	{2.5ex \@plus 1ex \@minus 0.2ex}%
	{1.5ex \@plus 0.2ex}%
	{\normalfont\normalsize\bfseries}}
\newcommand*{\subsubsubsection}{\@startsection{paragraph}{4}{\z@}%
	{2.5ex \@plus 1ex \@minus 0.2ex}%
	{1.5ex \@plus 0.2ex}%
	{\normalfont\normalsize\bfseries}}
\makeatother

\newcommand{\Diff}{\mathrm{Diff}}

\newcommand{\Aut}{\mathrm{Aut}}

\newcommand{\Map}{\ensuremath{\mathrm{Map}}}

\newcommand{\incl}[3][right]%
{%
\draw[<-,>=#1 hook] #2 to ($ #2!0.5!#3 $);
\draw[->] ($ #2!0.5!#3 $) to #3;%
}
\newcommand{\inclusion}[5][right]%
{%
\draw[<-,>=#1 hook] #4 to ($ #4!0.5!#5 $) node[#2,font=\small]{#3};
\draw[->] ($ #4!0.5!#5 $) to #5;%
}

\newsavebox{\pullback}
\sbox\pullback{%
\begin{tikzpicture}%
\draw (0,0) -- (1ex,0ex);%
\draw (1ex,0ex) -- (1ex,1ex);%
\end{tikzpicture}}

\newenvironment{itemizeb}%
{\begin{compactitem}

}%
{\end{compactitem}}
{\begin{compactitem}[#1]

}%
{\end{compactitem}}
{\begin{compactdesc}

}%
{\end{compactdesc}}

\newcommand{\cB}{\mathcal{B}}
\newcommand{\cC}{\mathcal{C}}
\newcommand{\cD}{\mathcal{D}}

\newcommand{\cG}{\mathcal{G}}
\newcommand{\cH}{\mathcal{H}}

\newcommand{\cM}{\mathcal{M}}
\newcommand{\cN}{\mathcal{N}}

\newcommand{\cP}{\mathcal{P}}

\newcommand{\cV}{\mathcal{V}}

\newcommand{\bD}{\mathbb{D}}

\newcommand{\bN}{\mathbb{N}}

\newcommand{\bQ}{\mathbb{Q}}
\newcommand{\bR}{\mathbb{R}}
\newcommand{\bS}{\mathbb{S}}

\newcommand{\bZ}{\mathbb{Z}}

\newcommand{\ab}{\mathbf{ab}}
\newcommand{\Fct}{\mathbf{Fct}}

\newcommand{\B}{\mathrm{B}}

\newcommand{\fF}{\mathfrak{F}}
\newcommand{\tti}{\mathtt{i}}

\newcommand{\fC}{\mathfrak{C}}
\newcommand{\st}{\mathfrak{s}}
\newcommand{\sH}{\boldsymbol{\mathscr{H}}}

\renewcommand{\footnoterule}{%
  \kern -3pt
  \hrule width \textwidth height 0.4pt
  \kern 2.6pt
}

\newcommand{\obj}{\ensuremath{\mathrm{ob}}}

\newcommand{\Ab}{\mathbf{Ab}}

\newcommand{\Grp}{\mathbf{Grp}}

\newcommand{\Top}{\mathbf{Top}}
\newcommand{\hTop}{\mathbf{hTop}}
\newcommand{\hMan}{\mathbf{hMan}}

\newcommand{\Cok}{\mathrm{Coker}}
\newcommand{\Ker}{\mathrm{Ker}}
\newcommand{\Image}{\mathrm{Im}}

\newcommand{\id}{\mathrm{id}}

\newcommand{\Bun}{\mathrm{Bun}}
\newcommand{\Hom}{\mathrm{Hom}}

\newcommand{\MCG}{\Gamma}

\newcommand{\Orth}{\mathrm{O}}

\DeclareMathOperator*{\colim}{colim}

\definecolor{Arthur}{rgb}{0.3,0,1}

\definecolor{dgreen}{rgb}{0,0.5,0}

\definecolor{later}{rgb}{0,0,0.5}

\definecolor{dOrange}{rgb}{0.9,0.4,0}

\tikzset{
  symbol/.style={
    draw=none,
    every to/.append style={
      edge node={node [sloped, allow upside down, auto=false]{$#1$}}}
  }
}

\begin{document}
\title{\LARGE\bfseries Twisted homological stability for handlebody mapping class groups}
\author{\normalsize Erik Lindell and Arthur Souli{\'e}}
\date{\normalsize\todaysdate}
\maketitle
{\makeatletter
\renewcommand*{\BHFN@OldMakefntext}{}
\makeatother
\footnotetext{2020 \textit{Mathematics Subject Classification}:
Primary: 20J05, 20J06, 55R40, 57K20; Secondary: 18M05, 55U10, 57M07.}
\footnotetext{\textit{Key words and phrases}: twisted homological stability, handlebody groups, moduli spaces of 3-manifolds, tangential structures.}

\begin{abstract}
We prove twisted homological stability for handlebody mapping class groups. Using the categorical framework developed by Randal-Williams and Wahl, we establish that the homology of the handlebody groups stabilises with respect to both genus and the number of marked boundary discs, for all coefficient systems of finite degree. Our first main theorem refines and extends the twisted stability result for handlebodies outlined by Randal-Williams and Wahl, allowing any number of marked discs and boundary points. We then introduce the notion of coefficient bisystem to treat stability under variation of boundary markings. As an application, we deduce homological stability for moduli spaces of 3-dimensional handlebodies equipped with tangential structures.
\end{abstract}

\section{Introduction}

Homological stability is a prevalent phenomenon in topology, asserting that the homology of a sequence of groups (or spaces) $G_{0} \to G_{1} \to G_{2} \to \cdots $ becomes independent of the index in a range increasing with homological degree. Beyond the classical case with constant coefficients, its twisted form considers functorial coefficient systems, allowing one to capture interactions between geometric and representation-theoretic information. In particular, twisted homological stability is thus a fundamental tool for computing group homology via the stable value, and for constructing universal characteristic classes of bundles of families of manifolds. In this paper, we study twisted homological stability for \emph{handlebody groups}. Namely, for $g\ge 0$, let $V_{g}$ denote the $3$--dimensional handlebody of genus $g$, i.e.~the iterated boundary connected sum $V_{g}:=\flat_{g} S^1\times\bD^{2}$ for some integer $g\ge 0$, where $\bD^{2}$ denotes the $2$--dimensional unit disc. Denoting by $\cD$ the image of an embedded copy of $\bD^{2}$ in the boundary of $V_{g}$, we consider the \emph{handlebody mapping class group} (a.k.a.~\emph{handlebody group})
\begin{equation}\label{eq:def-handlebody-mcg} 
\cH_{g,1}:=\pi_{0}\Diff(V_{g};\cD),
\end{equation}
that is the group of isotopy classes of diffeomorphisms of $V_{g}$ that pointwise fix the disc $\cD$.

By taking the boundary connected sum $V_{g}\flat V_{1}$ along some marked disc in $\partial V_{g}$ as in Figure~\ref{fig:Stabilisation}, we can extend any diffeomorphism of $V_{g}$ that fixes $\cD$ to one of $V_{g+1}$, by defining it to be the identity on $V_{1}$. By picking any embedded 2-disc in the boundary of $V_{g+1}\setminus V_{g}$, this defines a \emph{stabilisation homomorphism}
\[
\sigma_{g,1}\colon \cH_{g,1}\to\cH_{g+1,1}.
\]
\textcite[Cor.~1.9]{HatcherWahl} proved that the induced sequence of homomorphisms satisfies integral \emph{homological stability}, i.e.~that in degrees sufficiently small compared to $g$, $\sigma_{g,1}$ induces an isomorphism in integral homology. A proof generalising this result to hold with various twisted coefficient was outlined by \textcite[\S5.7]{RWW}, recovering in particular a previous stability result for the first homology with a specific twisted coefficient system found by Ishida and Sato in \cite[Th.~1.2]{IshidaSato} (see Example~\ref{eg:HS-examples-recover}). In this paper, we generalise these results further, in several directions. The first main result fills in the details of the proof of twisted homological stability for handlebody mapping class groups of \cite[Th.~5.31]{RWW}, and generalises it to hold with any number of marked discs and points in the boundary of $V_{g}$; see Theorem~\ref{thm:mainA}.

\begin{figure}[h]
\centering
\includegraphics[scale=0.5]{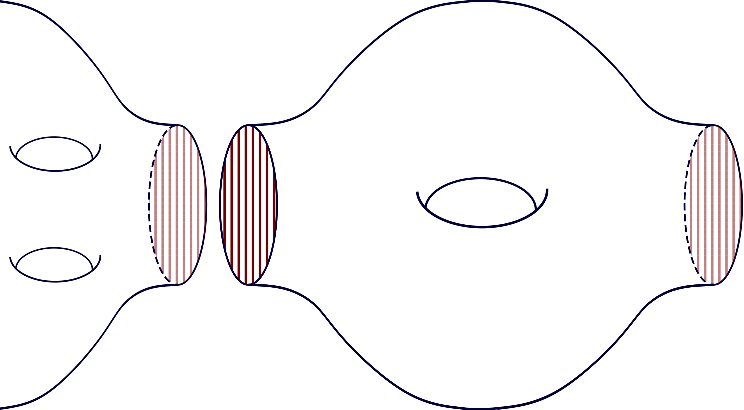}
\caption{The handlebody $V_{g+1}$ is obtained from $V_{g}$ by gluing on a genus $1$ handlebody along the marked disc and the stabilisation map $\sigma_{g,1}$ is defined by extending diffeomorphisms to the new part by the identity.}
\label{fig:Stabilisation}
\end{figure}

\paragraph*{Additional marked discs and points.} A more general flavour of the above handlebody mapping class group \eqref{eq:def-handlebody-mcg} is given by replacing the marked disc $\cD$ in the boundary of $V_{g}$ with any number $r\ge 1$ of disjointly embedded discs and $s\ge 0$ distinct points in the boundary. We denote by $\cH^{s}_{g,r}$ the group of the isotopy classes of diffeomorphisms of $V_{g}$ that pointwise fix all marked discs and points. If $s$ is zero, we drop it from the notation. Also, we write
\[
\sigma^{s}_{g,r}\colon \cH^{s}_{g,r}\to\cH^{s}_{g+1,r}
\]
for the analogously defined stabilisation homomorphism; see also \S\ref{sss:hom-stdness-properties} for the precise definition of this map.

\paragraph*{Twisted coefficients.} For integers $r\ge 1$ and $s\ge 0$, let $\{(F^{s}_{g,r},c^{s}_{g,r})\}_{g\ge 0}$ be a sequence of pairs where $F^{s}_{g,r}$ is a $\bZ[\cH^{s}_{g,r}]$-module and
\[
c^{s}_{g,r}:F^{s}_{g,r}\to F^{s}_{g+1,r}
\]
is a $\bZ[\cH^{s}_{g,r}]$-linear map, where the target is considered a $\bZ[\cH^{s}_{g,r}]$-module via the stabilisation map $\sigma^{s}_{g,r}$. We call such a sequence a \emph{coefficient system} for the sequence $\{(\cH^{s}_{g,r},\sigma^{s}_{g,r})\}_{g\ge 0}$.

\paragraph*{Stability with respect to handles.}
A coefficient system $\{(F^{s}_{g,r},c^{s}_{g,r})\}_{g\ge 0}$ for the sequence $\{(\cH^{s}_{g,r},\sigma^{s}_{g,r})\}_{g\ge 0}$ induces a map in homology
\begin{equation}\label{eq:intro-stab-maps-in-homology}
    (\sigma^{s}_{g,r},c^{s}_{g,r})_{*}\colon H_{*}(\cH^{s}_{g,r};F^{s}_{g,r})\to H_{*}(\cH^{s}_{g+1,r};F^{s}_{g+1,r})
\end{equation}
for each $g\ge 0$. Using the general framework for homological stability developed by \cite{RWW}, we prove that if the coefficient system is of \emph{finite degree}, which means that it satisfies a certain polynomiality condition (see \S\ref{ss:finite-degree} for details), the sequence of the maps \eqref{eq:intro-stab-maps-in-homology} satisfies homological stability:
\begin{athm}[Theorem~\ref{thm:HS_stabilistation-1}]\label{thm:mainA}
If $\{(F^{s}_{g,r},c^{s}_{g,r})\}_{g\ge 0}$ for integers $r\ge 1$ and $s\ge 0$, is a coefficient system of finite degree $d$, with respect to the sequence $\{(\cH^{s}_{g,r},\sigma^{s}_{g,r})\}_{g\ge 0}$, the map
$$(\sigma^{s}_{g,r},c^{s}_{g,r})_{*}\colon H_{*}(\cH^{s}_{g,r};F^{s}_{g,r})\to H_{*}(\cH^{s}_{g+1,r};F^{s}_{g+1,r})$$
is an isomorphism in degrees $*\le \frac{g-1}{2}-d-1$. 
\end{athm}

\begin{rmk}
As detailed in the more technical Theorem~\ref{thm:HS_stabilistation-1} below, the stability range can be improved if the coefficient system is \emph{split} (see \S\ref{ss:finite-degree} for the definition), an additional assumption which roughly boils down to the assumption that $c^{s}_{g,r}:F^{s}_{g,r}\to F^{s}_{g+1,r}$ is split injective for each $g\ge 0$; see \S\ref{ss:finite-degree}.
\end{rmk}

Some examples of finite degree coefficients $\{(F^{s}_{g,r},c^{s}_{g,r})\}_{g\ge 0}$ can be obtained from the homology of $V_{g}$ and of its boundary, the surface $\partial V_{g}$. More specifically, if we let $\cD_{r}$ denote the union of the marked discs and $\cP_{s}$ the union of the marked points in the boundary of $V_{g}$, the relative homology groups $H_{1}(V_{g},\cD_{r}\sqcup\cP_{s};\bZ)$ and $H_{1}(\partial V_{g}\setminus (\mathring{\cD}_{r}\sqcup\cP_{s});\bZ )$ are the components of coefficient systems of degree $1$; see Example~\ref{eg:HS-examples-recover} for more details. This implies that taking the $r$-th tensor powers of these coefficient systems give us coefficient systems of degree $r$ (see \cite[Cor.~1.5]{AnnexLS1}).

\paragraph*{Stability with respect to marked discs.}
Hatcher and Wahl \cite{HatcherWahl} also proved that when $g$ is large compared to the homological degree, the integral homology of $\cH_{g,r}$ is independent of $r$, the number of marked discs in the boundary of $V_{g}$. More specifically, there is a group homomorphism
\[
\mu^{s}_{g,r}\colon \cH^{s}_{g,r}\to \cH^{s}_{g,r+1},
\]
defined by gluing a sphere with three marked discs to the first marked disc of $V_{g}$ (see Figure~\ref{fig:Stabilisation_mu}) and extending diffeomorphisms by the identity, and the result of Hatcher and Wahl says (in the case $s=0$) that for $g$ large enough to the homological degree, the induced map in homology is an isomorphism.

Our second main theorem generalises this result to homology with twisted coefficients. Our assumption on the twisted coefficients is somewhat technical, so let us leave out the details in the introduction. However, it essentially says that the coefficients need to be of finite degree in two different ways: we call this being a ``coefficient bisystem''; see Definition~\ref{def:double-coeff-system}.
\begin{figure}[h]
\centering
\includegraphics[scale=0.7]{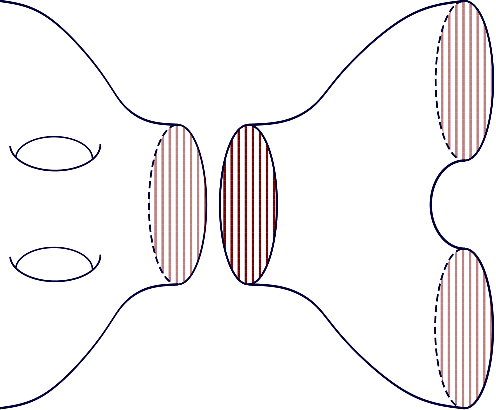}
\caption{By gluing a ``solid pair of pants'' to the first marked disc of our handlebody, we obtain a handlebody with an additional marked disc.}
\label{fig:Stabilisation_mu}
\end{figure}

\begin{athm}[Corollary~\ref{cor:independence-of-marked-discs}]\label{thm:mainB}
For integers $r\ge 1$ and $s\ge 0$, let $m^{s}_{g,r}\colon  F^{s}_{g,r}\to F^{s}_{g,r+1}$ be a $\bZ[\cH^{s}_{g,r}]$-linear map, where the target is consider a $\bZ[\cH^{s}_{g,r}]$-module via $\mu^{s}_{g,r}$. If $m^{s}_{g,r}$ is induced by a coefficient bisystem of degree $d$ (see Definition~\ref{def:double-coeff-system} below), the map
$$(\mu^{s}_{g,r}, m^{s}_{g,r})_{*}\colon H_{*}(\cH^{s}_{g,r}; F^{s}_{g,r})\to H_{*}(\cH^{s}_{g,r+1}; F^{s}_{g,r+1}),$$
is an isomorphism in degree $*\le\frac{g-1}{2}-d-1$.
\end{athm}

\begin{rmk}
As in Theorem~\ref{thm:mainA}, the stability range can be improved if the coefficient bisystem is further assumed to be split.
\end{rmk}

Some examples of finite degree coefficient bisystems can again be obtained from the $\cH^{s}_{g,r}$-representations $H_{1}(V_{g},\cD_{r}\sqcup\cP_{s};\bZ)$ and $H_{1}(\partial V_{g}\setminus\mathring{\cD}_{r},\cP_{s};\bZ)$, considered above; see Example~\ref{eg:HW-recover} for details.

\paragraph*{Moduli spaces of handlebodies with tangential structures.}
Our homological stability results for handlebody mapping class groups with twisted coefficients also admit a reinterpretation in terms of moduli spaces, and lead to stability results for handlebodies equipped with tangential structures. Let $R^{s}_{g,r}\subset \partial V_{g}$ denote the union of $r\ge 1$ marked discs and $s\ge 0$ marked points in $\partial V_{g}$. By Lemma~\ref{lem:H-injects-into-Gamma}, the diffeomorphism group $\Diff(V_{g};R^{s}_{g,r})$ has contractible path components, and is therefore homotopy equivalent to the discrete group $\cH^{s}_{g,r}$. In particular, there is an homotopy equivalence $\B\Diff(V_{g};R^{s}_{g,r}) \simeq \B\cH^{s}_{g,r}$, so the group cohomology of $\cH^{s}_{g,r}$ agrees with the singular cohomology of $\B\Diff(V_{g};R^{s}_{g,r})$. Thus homological stability for $\cH^{s}_{g,r}$ may equivalently be formulated as homological stability for the spaces $\B\Diff(V_{g};R^{s}_{g,r})$. The space $\B\Diff(V_{g};R^{s}_{g,r})$ may be viewed as a \emph{moduli space} of genus $g$ handlebodies with marked boundary data. This viewpoint extends naturally to moduli spaces of handlebodies endowed with additional \emph{tangential structures}. In higher dimensions, homological stability for such moduli spaces has been studied by \textcite{Perlmutter}, and their stable homology by \textcite{Botvinnik-Perlmutter}. Related stability and stable cohomology results for the even-dimensional manifolds $W_{g} := \#_{g}(S^{n}\times S^{n})$, which arise as boundaries of $(2n+1)$--dimensional handlebodies, were previously established by \textcite{GalatiusRW-HomStab,GalatiusRW-StableCoh}.

In our setting, a tangential structure is a fibration $\theta\colon B \to \B\Orth(3)$, where $\Orth(3)$ is the group of $3\times3$ orthogonal matrices, and a $\theta$-structure on $V_{g}$ is a map of vector bundles (i.e.~a fibrewise map whose restriction to each fibre is a linear isomorphism) $\ell\colon TV_{g}\to \theta^{*}\gamma_{3}$, where $TV_{g}$ is the tangent bundle of $V_{g}$, $\gamma_{3}$ is the tautological rank-3 vector bundle over $\B\Orth(3)$ and $\theta^{*}\gamma_{3}$ denotes its pullback along $\theta$. We denote by $\Bun^{\theta}(V_{g}) \subset \Map(TV_{g},\theta^{*}\gamma_{3})$ the space of $\theta$-structures on $V_{g}$. The corresponding \emph{moduli space of genus $g$ handlebodies with a $\theta$-structure} is defined as the homotopy orbit space
\[
\B\Diff^{\theta}(V_{g})
:=
\Bun^{\theta}(V_{g})_{h\Diff(V_{g})},
\]
where $\Diff(V_{g})$ acts by precomposition. In \S\ref{def:tangential-structures}, we define these notions in more detail, and refine this construction by imposing boundary conditions along $R^{s}_{g,r}$, leading to moduli spaces $\B\Diff^{\theta}(V_{g};\ell^{s}_{g,r})$.

The boundary condition $\ell^{s}_{g,r}$ allows us to define stabilisation maps
\begin{equation}\label{eq:intro-tangential-stabilisation-phi}
\breve{\sigma}_{g,r}^{\theta,s}\colon
\B\Diff^{\theta}(V_{g};\ell^{s}_{g,r})
\longrightarrow
\B\Diff^{\theta}(V_{g+1};\ell^{s}_{g+1,r}),
\end{equation}
which increase the genus by extending the $\theta$-structure by a fixed one on the added handle, and similarly
\begin{equation}\label{eq:intro-tangential-stabilisation-mu}
\breve{\mu}_{g,r}^{\theta,s}\colon
\B\Diff^{\theta}(V_{g};\ell^{s}_{g,r})
\longrightarrow
\B\Diff^{\theta}(V_{g};\ell^{s}_{g,r+1}),
\end{equation}
which increase the number of marked discs. We may now state our final main result.

\begin{athm}[Theorem~\ref{thmC-detailed}]\label{thm:mainC}
Let $\theta\colon B\to \B\Orth(3)$ be a simply connected tangential structure. Then the induced maps
\[
(\breve{\sigma}_{g,r}^{\theta,s})_{*}\colon
H_{*}(\B\Diff^{\theta}(V_{g};\ell^{s}_{g,r});\bQ)
\longrightarrow
H_{*}(\B\Diff^{\theta}(V_{g+1};\ell^{s}_{g+1,r});\bQ)
\]
and
\[
(\breve{\mu}_{g,r}^{\theta,s})_{*}\colon
H_{*}(\B\Diff^{\theta}(V_{g};\ell^{s}_{g,r});\bQ)
\longrightarrow
H_{*}(\B\Diff^{\theta}(V_{g};\ell^{s}_{g,r+1});\bQ)
\]
are isomorphisms in degrees $*\le \frac{g-3}{2}$.
\end{athm}

\begin{rmk}
Although Theorem~\ref{thm:mainC} is deduced from Theorems~\ref{thm:mainA} and~\ref{thm:mainB}, which both hold integrally, it is stated with rational coefficients. This restriction stems from a technical issue concerning finite-degree coefficient systems; see Remark~\ref{rmk:thmC-over-Q}.
\end{rmk}

\paragraph*{Applications.}
In \cite{LS2}, we apply the results of this paper to compute the stable cohomology of $\cH_{g,1}$ with coefficients in tensor products of the representations $H_{1}(V_{g};\bQ)$ and $H_{1}(\partial V_{g};\bQ)$. Let us briefly outline below how these results are leveraged to carry out these computations. The starting point is Theorem~\ref{thm:mainC}, which allows us to invoke forthcoming results of \textcite{Barkan-Steinebrunner} to determine the stable rational cohomology of $\B\Diff^{\theta}(V_{g};\ell_{g,1})$ for suitable tangential structures $\theta\colon B\to \B\Orth(3)$. Given a rational vector space $W$, we consider the tangential structure $\theta_{W}\colon \B\Orth(3)\times K(W^{\vee},2)\to \B\Orth(3)$. In this case one has $\B\Diff^{\theta_{W}}(V_{g};\ell_{g,1}\simeq \Map((V_{g},\mathcal{D}), (K(W^{\vee},2),*))$.
A spectral sequence argument, combined with representation-theoretic input for general linear and symmetric groups, then allows us to compute the stable cohomology of $\cH_{g,1}$ with tensor powers of $H_{1}(V_{g};\bQ)$ as coefficients. To incorporate coefficients involving $H_{1}(\partial V_{g};\bQ)$, we use the short exact sequences relating the groups $\cH_{g,r}$ as the numbers of marked discs and points vary. The associated spectral sequences yield an inductive procedure for computing stable cohomology groups of the form $H^{*}(\cH_{g,1}; H_{1}(\partial V_{g};\bQ)\otimes M)$, assuming the stable cohomology with coefficients in $M$ is already known. The independence of stable twisted cohomology from the number of marked discs, established in Theorem~\ref{thm:mainB}, is a crucial input in this argument.

\paragraph*{Conventions and notation.}
We denote by $\Ab$ the category of abelian groups. We write $\otimes$ for $\otimes_{R}$ when the ground ring $R$ is clear from the context. Non-specified tensor products are taken over $\bZ$. For each $d\ge 1$, the $d$-fold tensor power of abelian groups functor is denoted by $T^{d}\colon \Ab \to \Ab$. We denote by $I$ the unit interval $[0,1]$.

\paragraph*{Acknowledgements} The authors thank Nathalie Wahl for several useful discussions. During the project, the first author was first supported by the project ANR-22-CE40-0008 SHoCoS, as a postdoc of Najib Idrissi, and then by the Knut and Alice Wallenberg Foundation through grant no. 2022.0278. He is also grateful to the Copenhagen Centre for Geometry and Topology for their hospitality during the writing. The second author was supported by PEPS JCJC 2024 (Projets Exploratoires-Premier Soutien, Jeunes chercheurs et jeunes chercheuses) of the INSMI. This collaboration was supported by the IRN MaDeF (International Research Networks Mathematics in Denmark and France).

\tableofcontents

\section{Background}\label{s:background}

In this section, we start by recollecting in \S\ref{ss:handlebody_MCG_recollections} some basics about handlebody groups and their relation to mapping class groups of surfaces. After this, we construct in \S\ref{sss:categorical_framework} a number of braided monoidal groupoids of handlebodies and modules over these, whose automorphism groups are precisely the handlebody groups, and we recall the Quillen bracket construction. As we study homological stability with twisted coefficients, we review in \S\ref{ss:finite-degree} the designed notion of finite degree coefficient system, of which we also introduce interesting examples.
We finally recall in \S\ref{ss:framework-HS} the general framework of \cite{RWW,Krannich} to prove twisted homological stability, and we verify that the categories we have introduced satisfy some basic properties that make them fit into it.

\subsection{Recollections on handlebody groups}\label{ss:handlebody_MCG_recollections}

We recall that a {$3$--dimensional compact handlebody} is a smooth $3$-manifold obtained by attaching a finite number of handles to the $3$-ball. Equivalently, it is diffeomorphic to the boundary connected sum $\flat_{g} (\bS^{1}\times \bD^{2})$ for some integer $g\ge 0$, which is denoted by $V_{g}$ and its boundary by $\Sigma_{g}:=\partial V_{g}$. For integers $r,s\ge 0$, let $\cD_{r}$ be the image of the disjoint union $\sqcup_{r}\bD^{2}$ of $r\ge 0$ embedded copies of the $2$-disc $\bD^{2}$ in the boundary of $V_{g}$ and $\cP_{s}$ be the image of an embedded copy of the set of $s\ge 0$ distinct points $\{p_{1},\ldots,p_{s}\}$ in $\Sigma_{g}\setminus \cD_{r}$. We also denote by $\Sigma^{s}_{g,r}$ the surface $\Sigma_{g}\setminus (\mathring{\cD}_{r}\sqcup\cP_{s})$, where $\mathring{\cD}_{r}$ is the disjoint union of the interiors of the marked discs of $\cD_{r}$.

We consider the topological group $\Diff(V_{g};\cD_{r}\sqcup\cP_{s})$ of orientation-preserving diffeomorphisms of $V_{g}$, which pointwise fix the sets $\cD_{r}$ and $\cP_{s}$. Its group of path components is denoted by $\cH^{s}_{g,r}$ and called the \emph{handlebody (mapping class) group}.
We also deal with the topological group $\Diff(\Sigma_{g};\cD_{r}\sqcup\cP_{s})$ of orientation-preserving diffeomorphisms of the boundary of $V_{g}$ fixing both $\cD_{r}$ and $\cP_{s}$ pointwise. We write $\MCG^{s}_{g,r}$ for its group of path components, which is called the \emph{mapping class group} of $\sigma^{s}_{g,r}$.
If $r=0$ (respectively $s=0$), we may omit $r$ and $\cD_{r}$ (respectively $s$ and $\cP_{s}$) from the notation.

A key property of handlebody groups is that they canonically embed into the mapping class groups of their boundaries. This result is probably known to the experts, but we give a short proof here for the convenience of the reader; see Lemma~\ref{lem:H-injects-into-Gamma}.
Since diffeomorphisms of smooth manifolds are boundary-preserving (see \cite[Th.~2.18]{Lee} for instance), the restriction map to the boundary $V_{g}\to \Sigma_{g}$ induces a fibration
\begin{equation}
\label{eq:fibration_restriction_diffs}
\Diff(V_{g};\cD_{r}\sqcup \cP_{s})\to \Diff(\Sigma_{g};\cD_{r}\sqcup \cP_{s}),
\end{equation}
with fibre $\Diff(V_{g};\partial V_{g})$.
\begin{lem}\label{lem:H-injects-into-Gamma}
For each $g,r,s\ge 0$, the map \eqref{eq:fibration_restriction_diffs} induces an injective homomorphism
\begin{equation}\label{eq:H-injects-into-Gamma}
\tti^{s}_{g,r}\colon\cH^{s}_{g,r}\hookrightarrow\MCG^{s}_{g,r}.
\end{equation}
Furthermore, if $g\ge 2$, then the homotopy groups $\pi_{i}(\Diff(V_{g};\cD_{r}\sqcup \cP_{s}))$ are trivial for all $i\ge 1$.
\end{lem}
\begin{proof}
First, the morphism $\tti^{s}_{g,r}$ is defined as the right-most map of the homotopy long exact sequence $\cdots \to \pi_{0}\Diff(V_{g};\partial V_{g}) \to \cH^{s}_{g,r}\to \MCG^{s}_{g,r}$ of the fibration \eqref{eq:fibration_restriction_diffs}.
We note that the handlebody $V_{g}$ is a compact and irreducible $3$-manifold, homeomorphic to $\bD^{2}_{g}\times I$ where $\bD^{2}_{g}:= \bD^{2}\setminus \sqcup_{1\le i\le g}\mathring{\bD}^{2}_{i}$ (where $\mathring{\bD}^{2}_{i}$ are in the interior of $\bD^{2}$).
Then we know from \cite[Proof of Th.~1, Step~4]{Hatcherincompressible} that $\pi_{0}\Diff(V_{g};\partial V_{g})=0$, and so we deduce from the above homotopy long exact sequence that $\tti^{s}_{g,r}$ is injective.

Finally, for $g\ge 2$, the path components of $\Diff(\Sigma_{g};\cD_{r}\sqcup \cP_{s})$ are contractible by \cite{EarleEells1,EarleEells2,EarleSchatz}, as well as those of $\Diff(V_{g};\partial V_{g})$ by \cite[Th.~2]{Ivanov2}, \cite{Hatcher3manif} or \cite[Th.~2]{Hatcherincompressible}. The vanishing of the higher homotopy groups follows from the homotopy long exact sequence of the fibration \eqref{eq:fibration_restriction_diffs}.
\end{proof}
\begin{rmk}\label{rmk:classifying-space-handlebody}
Since the path component of the identity in $\Diff(V_{g};\cD_{r}\sqcup \cP_{s})$ is trivial by Lemma~\ref{lem:H-injects-into-Gamma}, the cohomology of the classifying space $\B\Diff(V_{g};\cD_{r}\sqcup\cP_{s})$ agrees with the group cohomology of the handlebody group $\cH^{s}_{g,r}$.
\end{rmk}

\subsection{Categorical framework}\label{sss:categorical_framework}

In this section, we review the necessary specific categories to apply the framework of \cite{RWW} and \cite{Krannich} to prove our twisted homological stability results for the handlebody groups.

\paragraph*{Preliminaries on categorical tools.}
We refer to \cite[\S VII]{MacLane1} for a complete introduction to the notions of monoidal categories and modules over them. We generically denote a monoidal category by $(\cC,\odot,0)$, where $\cC$ is a category, $\odot$ is the monoidal product, and $0$ is the monoidal unit. If it is braided, then its braiding is generically denoted by $b_{A,B}^{\cC} \colon A\odot B\overset{\sim}{\to}B\odot A$ for all objects $A$ and $B$ of $\cC$. A right-module $(\cM,\triangleright)$ (resp.~left-module $(\cM,\triangleleft)$) over a monoidal category $(\cC,\odot,0)$ is a category $\cM$ with a functor $\triangleright\colon\cM\times\cC\to\cM$ (resp.~$\triangleleft\colon\cC\times\cM\to\cM$) that is unital and associative. For example, a monoidal category $(\cC,\odot,0)$ is equipped with a left or right-module structure over itself, induced by its monoidal product. Each right-module structure $\triangleright$ in this paper is defined from some underlying monoidal structure $\odot$ (see \S\ref{sss:groupoids-handlebodies}), so we abuse notation by using the same symbol $\odot$ for $\triangleright$ or $\triangleleft$.

\subsubsection{Groupoids associated to compact handlebodies}\label{sss:groupoids-handlebodies}

We start by defining some braided monoidal groupoids, where the objects are handlebodies and the automorphism groups are the groups $\cH^{s}_{g,r}$.

\begin{defn}\label{def:groupoid-handlebody}
We define a category $\sH$ as follows.
\begin{itemizeb}
    \item The objects are 5-tuples $(V,r,s,i,j)$. Here $V$ is a $3$--dimensional compact handlebody (i.e.~it is diffeomorphic to $V_{g}=\flat_{g} (\bS^{1}\times \bD^{2})$ for some integer $g\ge 0$) or a disjoint union of a finite number of $3$-balls; $r\ge 1$ and $s\ge 0$ are integers, $i\colon \cD_{r}\hookrightarrow \partial V$ is a smooth embedding and $j\colon \cP_{s}\hookrightarrow \partial V\setminus\Image(i)$ is an injection. In particular, we remember the data of the embedding $j$ and not just its image, so that we have well-defined notions of \emph{left}-half and \emph{right}-half of the image of each disc. The set $\cD_{r}:=\sqcup_{r}\bD^{2}$ has a canonical ordering induced by that of the set $\{1,\ldots,r\}$, and we denote by $\cD^{k}_{V}\in \cD_{r}$ the image of the $k$-th embedded disc via the embedding $i$. 
    \item  A morphism $(V_{1},r_{1},s_{1},i_{1},j_{1})\to (V_{2},r_{2},s_{2},i_{2},j_{2})$ is given by the isotopy class of a diffeomorphism $\phi\colon V_{1}\to V_{2}$ such that $\phi\circ i_{1}=i_{2}$ and $\phi\circ j_{1}=j_{2}$. Note that in particular, this implies that $r_{1}=r_{2}$ and $s_{1}=s_{2}$.
\end{itemizeb}
For $r\ge 1$, let us denote by $\sH^{\ge r}$ the full subcategory with objects of the form $(V,k,s,i,j)$ for $k\ge r$. In particular, we have $\sH^{\ge 1}=\sH$.
\end{defn}
We fix the following conventions for the compact handlebodies of $\sH$.
\begin{notation}\label{nota:V-g-r-s-object-G-H}
For integers $r\ge 1$ and $s\ge 0$, we denote by $V^{s}_{g,r}$ the genus $g\ge 0$ compact handlebody $V_{g}=\flat_{g} (\bS^{1}\times \bD^{2})$, along with $r$ marked discs $\cD_{r}$ and $s$ marked points $\cP_{s}$ in the boundary, as well as implicitly associated smooth embedding $i\colon \cD_{r}\hookrightarrow \partial V$ and injection $j\colon \cP_{s}\hookrightarrow \partial V\setminus\Image(i)$, to make it an object of $\sH$; see Definition~\ref{def:groupoid-handlebody}. When $s=0$, we omit it from the notation. In particular, we have $\Aut_{\sH}(V^{s}_{g,r})=\cH^{s}_{g,r}$.
\end{notation}

\paragraph*{Monoidal structures.} We continue by defining a monoidal structure on the category $\sH^{\ge r}$ for each $r\ge 1$, each corresponding to a different family of stabilisation maps between handlebody groups. 
We only consider the cases $r=1,2$ later in \S\ref{s:background} and \S\ref{s:homological_stability}, but it is more practical to introduce all $r\ge 1$ at once.

\begin{figure}[h]
\centering
\includegraphics[scale=0.3]{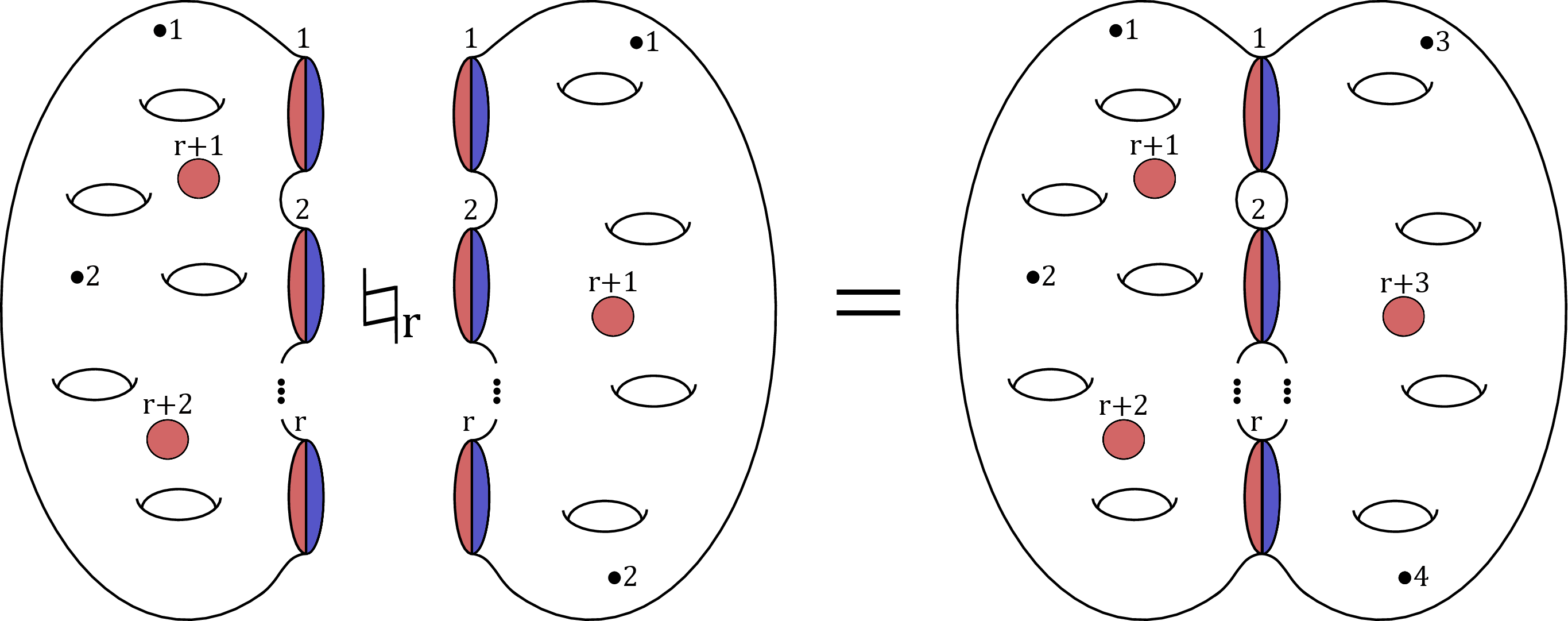}
\caption{The monoidal product $V_{1}\natural_{r} V_{2}$ of two handlebodies in $\cG_{\cG}^{\ge r}$.
}
\label{fig:natural_r}
\end{figure}

\begin{defn}\label{def:monoidal-structure-cG}
For $r\ge 1$ and for $(V_{1},r_{1},s_{1},i_{1},j_{1})$ and for objects $(V_{2},r_{2}, s_{2},i_{2},j_{2})$ of $\sH^{\ge r}$, we define the object $V_{1}\natural_{r}V_{2}$, which will serve as a monoidal product, as follows. We denote by $\frac{1}{2}\bD^{2}$ a half-disc. Let $f_{1}\colon \sqcup_{r}\frac{1}{2}\bD^{2}\hookrightarrow \Image(i_{1})$ denote the canonical inclusion into the right-half of the first $r$ marked discs $\sqcup_{1\leq i\leq r}\cD_{V_{1}}^{i}$ of $V_{1}$. Similarly, we denote by $f_{2}\colon \sqcup_{r}\frac{1}{2}\bD^{2}\hookrightarrow \Image(i_{2})$ the canonical inclusion into the left-half of the $r$ first marked discs $\sqcup_{1\leq i\leq r}\cD_{V_{2}}^{i}$ of $V_{2}$. The monoidal product $V_{1}\natural_{r} V_{2}$ is then given by the 5-tuple
\[(V_{1}\cup_{f_{1}\sim f_{2}} V_{2},r_{1}+r_{2}-r,s_{1}+s_{2},i_{1}\natural_{r}i_{2},j_{1}\sqcup j_{2}).\]
Namely:
\begin{itemizeb}
    \item The adjunction space $V_{1}\cup_{f_{1}\sim f_{2}} V_{2}$ is the quotient space $V_{1}\sqcup V_{2}/(f_{1}\sim f_{2})$, i.e.~the pushout of the diagram $(V_{1}\overset{f_{1}}\hookleftarrow \sqcup_{r} \frac{1}{2}\bD^{2} \overset{f_{2}}{\hookrightarrow}V_{2})$. In other words, it is the boundary connected sum along the right-half of each of the first $r$ marked discs (following the order given by $i_{1}$) in $V_{1}$ with the left-half of each corresponding marked disc in $V_{2}$.
    \item Identifying $\bD^{2}$ with the unit disc in $\bR^{2}$, we define $i_{1}\natural_{r}i_{2}\colon \sqcup_{r_{1}+r_{2}-r}\bD^{2}\to V_{1}\cup_{f_{1}\sim f_{2}} V_{2}$ by 
    \begin{align*}
    i_{1}\natural_{r}i_{2}(x,y)=\begin{cases}
    i_{1}(x,y)&\text{ if }x\le 0,\\
    i_{2}(x,y)&\text{ if }x\ge 0,
    \end{cases}
    \end{align*}
    on the first $r$ discs and as before on the remaining discs (ordered according to the order of the monoidal product).
\end{itemizeb}
Now, for morphisms $\phi_{1} \in \Hom_{\sH^{\ge r}}(V_{1}, V'_{1})$ and $\phi_{2} \in \Hom_{\sH^{\ge r}}(V_{2},V'_{2})$, we define the morphism $\phi_{1} \natural_{r}\phi_{2}$ by extending (up to isotopy) $\phi_{1}$ and $\phi_{2}$ by the identity on a collar neighbourhood of each disc in $\Image(i_{1})$ and $\Image(i_{2})$ respectively where the gluing is done.
\end{defn}
\begin{prop}\label{prop:monoidal-structure-G-cH}
The operation $-\natural_{r}-\colon \sH^{\ge r} \times\sH^{\ge r} \to \sH^{\ge r}$ makes $\sH^{\ge r}$ into a monoidal groupoid. The monoidal unit is the disjoint union of $r$ 3-balls $I_{r}$, each with one embedded disc and no marked points, which may be assumed to be \emph{strict} for simplicity.
\end{prop}
\begin{proof}
As $-\natural_{r} -$ is defined as a pushout on objects and by extending by the identity in the glued region for morphisms, the assignments of $-\natural_{r} -\colon \sH\times\sH\to\sH$ on objects and morphisms are defined in a unique way. The identity and composition axioms for morphisms are then straightforwardly checked from Definition~\ref{def:monoidal-structure-cG}, and the operation $-\natural_{r}-\colon \sH^{\ge r} \times\sH^{\ge r} \to \sH^{\ge r}$ is thus a well-defined bifunctor. The associator is defined by the natural isomorphism $V_{1}\natural_{r} (V_{2} \natural_{r} V_{3}) \cong (V_{1}\natural_{r} V_{2}) \natural_{r} V_{3}$ provided by the universal properties of $\cup_{f_{1}\sim f_{2}}$ and $\cup_{f_{2}\sim f_{3}}$ as pushouts, for which one easily checks the pentagon diagram axiom. The left and right unitors are given by the natural isomorphisms $V \natural_{r} I_{r} \cong V$ and $I_{r} \natural_{r} V \cong V$ induced by an evident deformation retraction of the 3-balls in $I_{r}$ onto their boundary marked discs, for which the triangle diagram axiom is obviously satisfied. Finally, we may turn $I_{r}$ into a \emph{strict} monoidal unit by slightly modifying the definition of the bifunctor $\natural_{r}$ through identification $V\natural_{r} I_{r}=V=I_{r}\natural_{r} V$, since the isomorphisms $V \natural_{r} I_{r} \cong V$ and $I_{r} \natural_{r} V \cong V$ then induce an obvious natural equivalence between $- \natural_{r} I_{r}$ and $I_{r} \natural_{r} -$.
\end{proof}

Although $\sH^{\ge 1}=\sH$ by Definition~\ref{def:groupoid-handlebody}, we implicitly mean that this groupoid is considered along with its monoidal structure $(\natural_{1},I_{1})$ when using the notation $\sH^{\ge 1}$.
Furthermore, for all integers $g\ge 0$, $r\ge 1$ and $s\ge 0$, it straightforwardly follows from the construction of the monoidal structures in Definition~\ref{def:monoidal-structure-cG} that there are evident natural isomorphisms in the groupoid $\sH$
\begin{equation}\label{eq:natural-iso-decomposition}
V^{s}_{g,r}\cong V^{s}_{0,r}\natural_{1} V_{1,1}^{\natural_{1} g} \cong V_{1,1}^{\natural_{1} g} \natural_{1} V^{s}_{0,r} \,\,\,\,\text{ and }\,\,\,\,V^{s}_{g,r}\cong V^{s}_{0,r}\natural_{2} V_{0,2}^{\natural_{2} g}\cong V_{0,2}^{\natural_{2} g} \natural_{2} V^{s}_{0,r},\end{equation}
with $r\ge 2$ for the isomorphisms involving $\natural_{2}$.

Let us now check some first elementary properties of these monoidal structures.
We recall that a monoidal category $(\cC, \odot, 0)$ is said to have \emph{no zero-divisors} if, for all objects $A$ and $B$ of $\cC$, $A\odot B \cong 0$ if and only if $A \cong B \cong 0$.
\begin{prop}\label{prop:no-0-divisors-no-automorphisms-of-0}
The monoidal groupoid $(\sH^{\ge r},\natural_{r},I_{r})$ has no zero-divisors and the monoidal unit $I_{r}$ has no non-trivial automorphisms, i.e.~$\Aut_{\sH^{\ge r}}(I_{r})=\{\id_{I_{r}}\}$.
\end{prop}
\begin{proof}
First, we note that the condition $A\natural_{r} B\cong I_{r}$ for objects in $\sH^{\ge r}$ implies that each component of $A$ and $B$ is a $3$-ball, each with one marked disc and no marked points. Since $A$ and $B$ must have $r$ components, we deduce that $A\cong I_{r}\cong B$, thus proving the first statement.
Now, since we require that the automorphisms of $I_{r}$ fix the marked disc in each component individually, we have group isomorphisms
\[
\Aut_{\sH^{\ge r}}(I_{r})\cong \Aut_{\sH}(I_{1})^{\times r}\cong \cH_{0,1}^{\times r}.
\]
Recall that $\MCG_{0,1}\cong 0$ by Alexander's trick (see \cite[Lem.~2.1]{farbmargalit} for instance), so $\cH_{0,1}$ is also trivial by Lemma~\ref{lem:H-injects-into-Gamma}, which then implies the second statement.
\end{proof}

\begin{figure}[h]
    \centering
    \includegraphics[scale=0.3]{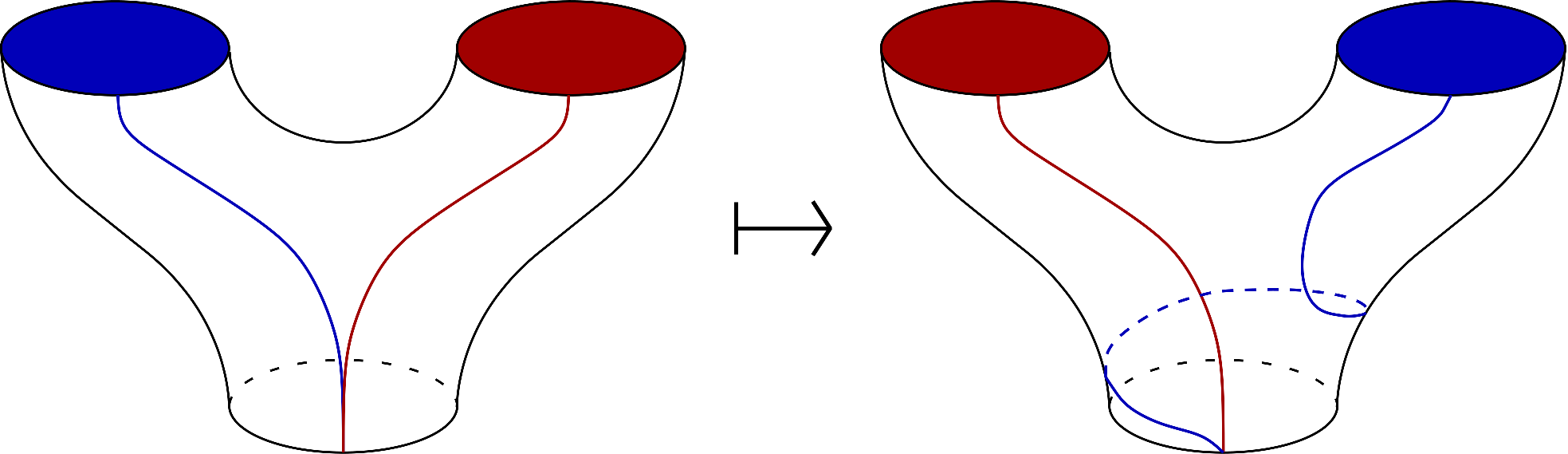}
    \caption{Illustration of the braiding in $(\sH^{\ge r},\natural_{r},I_{r})$ in a neighbourhood of one of the glued disc of $V_{1}\natural_{r} V_{2}$.}
    \label{fig:braiding}
\end{figure}

\paragraph*{Braidings.} Finally, we define a braiding on each $\sH^{\ge r}$.

Let us pick a closed neighbourhood $\cN$ in the handlebody $V_{1}\natural_{r} V_{2}$ of the subsurface induced by the gluing discs, i.e.~defined by the union of the glued half-discs $(f_{1}\sim f_{2})$ along with the new $r$ marked discs defined by $i_{1}\natural_{r}i_{2}$. Note that $\cN$ is a disjoint union of pair of pants neighbourhoods of the glued discs as shown in Figure~\ref{fig:braiding}: namely, it is diffeomorphic to $r$ disjoint $3$-balls, each with one marked disc defined by $i_{1}\natural_{r}i_{2}$, as well as two marked discs in each component of $\cN$ defined by the intersection 
\[
(V_{1}\natural_{r} V_{2}\setminus \mathring{\cN})\cap \cN.
\]
There is a diffeomorphism $b\colon \cN\to \cN$, unique up to isotopy, which fixes $\Image(i_{1}\natural_{r}i_{2})$ pointwise and switches the remaining two discs of each component, as described in Figure~\ref{fig:braiding}. Extending $b$ by the identity on $V_{1}\natural_{r} V_{2}\setminus \cN$, this defines an isomorphism:
\[
\beta_{V_{1},V_{2}}^{r}\colon V_{1}\natural_{r} V_{2}\to V_{2}\natural_{r} V_{1}.
\]
Note that this construction is clearly independent of the choice of the closed neighbourhood $\cN$, so $\beta_{V_{1},V_{2}}$ is a well-defined morphism in $\sH$.
\begin{prop}\label{prop:beta(1)-braided-monoidal}
The natural isomorphism $\beta^{r}$ makes $(\sH^{\ge r},\natural_{r},I_{r})$ into a braided monoidal category.
\end{prop}
\begin{proof}
Note that a diffeomorphism of $V_{i}$ for $i\in \{1,2\}$ is isotopic to the identity in $V_{i} \cap \cN \subset V_{1}\natural_{r} V_{2}$. Hence it follows from a disjoint support argument and extension by the identity for these diffeomorphisms that each isomorphism $\beta_{V_{1},V_{2}}^{r}$ is natural with respect to morphisms of $\sH^{\ge r}$ in both variables $V_{1}$ and $V_{2}$. Now, note that the braiding is supported on a neighbourhood of the glued discs for any monoidal product $V_{1}\natural_{r} V_{2}$, which can be chosen as a disjoint union of $3$-balls on which the braiding looks as in Figure~\ref{fig:braiding}. Hence it suffices to show that the braid relation (i.e.~the hexagon identities) hold in $\cH_{0,1,(3)}$, by which we mean the mapping class group of diffeomorphisms of the $3$-ball which fix one marked disc pointwise and fix three other marked discs \emph{setwise}, i.e.~up to permutation. The analogous mapping class group of diffeomorphisms of the $2$-sphere fixing one marked disc pointwise and three other marked discs \emph{setwise} is isomorphic to the braid group on three strands $\mathbf{B}_{3}$; see \cite[\S9.1.4]{farbmargalit} for instance.
Since $\Diff(V_{0};\partial V_{0})$ is contractible by the Smale conjecture (see \cite[Appendix]{HatcherSmale}), the fibration \eqref{eq:fibration_restriction_diffs} defined by the restriction map to the boundary $V_{0}\to \Sigma_{0}$ induces an injection $\cH_{0,1,(3)} \hookrightarrow \mathbf{B}_{3}$. The claim thus follows from the fact that the braid relation holds in $\mathbf{B}_{3}$; see for example \cite[\S5.6.1]{RWW}.
\end{proof}

The monoidal products $\natural_{r}$ are used below to define stabilisation maps $\cH^{s}_{g,r}\to\cH^{s}_{g+1,r}$ for $r=1,2$; see \eqref{eq:canonical_stabilisation-natural} and \eqref{eq:canonical_stabilisation-sharp}. 
Going forward, we only consider these cases, and we thus introduce the following notation:
\begin{notation}\label{nota:monoidal-structure}
We write $\natural:=\natural_{1}$ and $\#:=\natural_{2}$.
\end{notation}

\subsubsection{The Quillen bracket construction}\label{sss:Quillen-bracket-construction}

We now recall a categorical construction necessary for the twisted homological stability framework in \S\ref{ss:framework-HS} that we follow in \S\ref{s:homological_stability}. This notion was originally introduced in \cite[p.3]{graysonQuillen}, and later developed in \cite[\S1.1]{RWW} and \cite[\S4.1]{Krannich}, to which we refer the reader for further details.

Let $(\cG,\odot,0)$ be a monoidal groupoid. We consider a left $(\cG,\odot,0)$-module $(\cM,\odot)$ and a right $(\cG,\odot,0)$-module $(\cM',\odot)$. (For instance, we may take $\cM=\cM'=\cG$ with the left or right-module structure induced by the monoidal product $\odot$.)
The \emph{Quillen bracket constructions} $\langle \cM,\cG ]$ and $[ \cM',\cG \rangle$ on $(\cM,\odot)$ and $(\cM',\odot)$ over the groupoid $(\cG,\odot,0)$ are the categories whose objects are the same as those of $\cM$ and whose morphisms are given by
\[
\Hom_{\langle \cG,\cM ]}(X,Y)=\colim_{\cG}\left(\Hom_{\cM}(- \odot X,Y)\right) \text{ and }
\Hom_{[ \cM',\cG \rangle}(X,Y)=\colim_{\cG}\left(\Hom_{\cM}(X\odot -,Y)\right).
\]
Hence, a morphism from $X$ to $Y$ in $\langle \cM,\cG ]$ (resp.~$[ \cM',\cG \rangle$) is an equivalence class of pairs $(A,\phi)$ (resp.~$(\phi',A')$), denoted by $[A,\phi]$ (resp.~$[\phi',A']$), where $A\in\obj(\cG)$ (resp.~$A'\in\obj(\cG)$) and $\phi \in \Hom_{\cG}(A\odot X , Y)$ (resp.~$\phi'\in \Hom_{\cG}(X\odot A', Y)$). Namely, two pairs $(A_{1},\phi_{1})$ and $(A_{2},\phi_{2})$ (resp.~$(\phi'_{1},A'_{1})$ and $(\phi'_{2},A'_{2})$) are equivalent if there exists an isomorphism $\chi \in \Hom_{\cG}(A_{1},A_{2})$ (resp.~$\chi' \in \Hom_{\cG}(A'_{1},A'_{2})$) such that $\phi_{1}=\phi_{2}\circ (\chi\odot\id_{X})$ (resp.~$\phi'_{1}=\phi'_{2}\circ (\id_{X}\odot\chi')$).
For two morphisms $[A,\phi] \colon X\to Y$ and $[B,\psi] \colon Y\to Z$ in $\langle \cG , \cM' ] \rangle$ (resp.~$[\phi',A'] \colon X\to Y$ and $[\psi',B'] \colon Y\to Z$ in $[ \cM',\cG \rangle$), the composition is defined by $[B,\psi]\circ[A,\phi]=[B\odot A,\psi\circ(\id_{B}\odot \phi)]$ (resp.~$[\psi',B']\circ[\phi',A']=[\psi'\circ(\phi'\odot\id_{B'}),B'\odot A']$).

There are canonical functors $\cM\to\langle \cG,\cM ]$ and $\cM'\to [ \cM',\cG \rangle $ defined as the identity on objects, while sending $f\in\Hom_{\cM}(X,Y)$ and $f'\in\Hom_{\cM'}(X',Y')$ to $[0,f]$ and $[f',0]$ respectively.
We collect here some elementary properties about the Quillen bracket constructions that we will later use in \S\ref{ss:finite-degree}--\S\ref{ss:framework-HS}.
\begin{prop}\label{prop:general-properties-Quillen}
We assume that the modules $\cM$ and $\cM'$ are groupoids. If the monoidal groupoid $(\cG,\odot,0)$ is braided (with braiding $\beta^{\cG}\colon \cG\times\cG\to\cG\times\cG$), has no zero-divisors and is such that $\mathrm{Aut}_{\cG}(0)=\{\id_{0}\}$.
Then:
\begin{compactenum}[(i)]
    \item\label{item:maximal-subgroupoid} The groupoid $\cM$ is the maximal subgroupoid of $\langle \cG,\cM ]$ and $[ \cM,\cG \rangle$.
    \item\label{item:extension-monoidal-structure} For $\cM=\cG$, the monoidal unit $0$ is an initial object in the category $\langle \cG,\cG ]$ and the monoidal structure $(\cG,\odot,0)$ extends to $\langle \cG,\cG ]$ by the same assignments for objects and defined as follows on morphisms
\[
[A,\phi]\odot[B,\psi]:=[A\odot B,(\phi\odot\psi)\circ(\id_{A}\odot (b^{\cG}_{X_{1},B})^{-1}\odot\id_{X_{2}})]\in\Hom_{\langle \cG,\cG ]}(X_{1}\odot X_{2},Y_{1}\odot Y_{2});
\]
   
    \item \label{item:extension-module-structure} The right-module structure $(\cM',\odot)$ induces a right $(\cG,\odot,0)$-module structure on $[ \cM',\cG \rangle$, with the same assignments for objects and defined as follows on morphisms
\[
[\phi',A']\odot \psi':=[(\varphi'\odot\psi')\circ(\id_{X_{1}}\odot (b^{\cG}_{A',X_{2}})^{-1}),A']\in\Hom_{[ \cM',\cG \rangle}(X_{1}\odot X_{2},Y_{1}\odot X_{2}).
\]
\end{compactenum}
\end{prop}
\begin{proof}
For the category $\langle \cG,\cM ]$, the proof of \eqref{item:maximal-subgroupoid} repeats verbatim that of \cite[Prop.~1.7]{RWW}, while that of \eqref{item:extension-monoidal-structure} is same as that of those of \cite[Prop.~1.8(i)-(ii)]{RWW}. Also, the proof of \cite[Prop.~1.8(ii)]{RWW} adapts mutatis mutandis to show the analogous property \eqref{item:extension-module-structure} for the category $[ \cM',\cG \rangle$; see \cite[Rem.~4.12]{Krannich}.
\end{proof}
\begin{notation}
In the situation of Proposition~\ref{prop:general-properties-Quillen}, we abuse notation and write $\phi$ for $[0,\phi]\in\Hom_{\langle \cG,\cM]}(X,Y)$ or $[\phi',0]\in\Hom_{[ \cM',\cG \rangle}(X,Y)$ for all $X,Y\in\obj(\cG)$, $\phi\in\Hom_{\cM}(X,Y)$ and $\phi'\in\Hom_{\cM'}(X,Y)$.
\end{notation}

\begin{rmk}\label{rmk:improve-results-Quillen}
Monoidal and module structure properties analogous to \eqref{item:extension-monoidal-structure} and \eqref{item:extension-module-structure} of Proposition~\ref{prop:general-properties-Quillen} also hold for the categories $[ \cG,\cG \rangle$ and $\langle \cG,\cM ]$ respectively. We only stated the results we need for our work, in particular used to introduce the notion of finite degree coefficient systems in our context (see \S\ref{ss:finite-degree}) or to interpret some objects we introduce (see Remarks~\ref{rmk:colim-as-Hom+interpretation-W-hom-Quillen} and \ref{rmk:Theta-alternative-formula}).
\end{rmk}

Finally, we deduce the following result from Propositions~\ref{prop:no-0-divisors-no-automorphisms-of-0}, \ref{prop:beta(1)-braided-monoidal} and \ref{prop:general-properties-Quillen}.
\begin{coro}\label{coro:maximal-subgroupoid-of-UGH}
For any $r\ge 1$, the groupoid $\sH^{\ge r}$ is the maximal subgroupoid of the categories $\langle \sH^{\ge r},\sH^{\ge r} ]$ and $[ \sH^{\ge r},\sH^{\ge r}\rangle$. Also, the category $\langle \sH^{\ge r},\sH^{\ge r} ]$ (resp.~$[ \sH^{\ge r},\sH^{\ge r}\rangle$) is equipped with a monoidal (resp.~right $\sH^{\ge r}$-module) structure $(\natural_{r},I_{r})$.
\end{coro}

\subsubsection{Handlebodies of infinite rank}\label{sss:infinite-handlebodies}

Because of a technical issue in \S\ref{ss:handle-stabilisation-sharp}, we need to consider handlebodies which are not compact at some point. We therefore make the following definition:

\begin{defn}\label{def:groupoids-handlebody-infinity}
For each $3$--dimensional compact handlebody $V$ of $\sH$, let $V_{\infty,1}\natural V$ be the directed colimit in the category $\langle \sH^{\ge 1},\sH^{\ge 1}]$ of the sequential diagram
\begin{equation}\label{eq:def-infinite-in-barGH}
\begin{tikzcd}
    V\arrow[r,"\iota_{1}(V)"]& V_{1,1}\natural V\arrow[r,"\iota_{2}(V)"]& V_{1,1}^{\natural 2}\natural V\arrow[r]& \cdots \arrow[r,"\iota_{i}(V)"]& V_{1,1}^{\natural i}\natural V \arrow[r,"\iota_{i+1}(V)"]& \cdots.
\end{tikzcd}
\end{equation}
where $\iota_{i}(V)$ denotes the morphism $[V_{1,1},\id_{V_{1,1}^{\natural i}\natural V}]$ of $\Hom_{\langle \sH^{\ge 1},\sH^{\ge 1}]}(V_{1,1}^{\natural i-1}\natural V,V_{1,1}^{\natural i}\natural V)$. It is called the \emph{infinite object of $V$}. We construct the groupoid $\overline{\sH}$ as follows:
\begin{itemizeb}
\item The objects of $\overline{\sH}$ are those of $\sH$ along with any infinite object $V_{\infty,1}\natural V$ for $V$ a compact handlebody with marked discs and marked points of $\sH$.
\item The morphisms in $\overline{\sH}$ for the objects of $\sH$ are those Definition~\ref{def:groupoid-handlebody}. Those involving infinite objects are of the form $\phi\colon V_{\infty,1}\natural V\to V_{\infty,1}\natural V'$, i.e.~$\phi$ is defined by the universal property of the colimit $V_{\infty,1}\natural V$ if there exist integers $j, j'\ge 0$ and a set of morphisms $\{\phi_{i}\in \Hom_{\sH}(V^{\natural i+j}_{1,1}\natural V,V^{\natural i+j'}_{1,1}\natural V')\}_{i\ge j}$, such that the following diagram is commutative:
\begin{equation}\label{eq:morphism-infty-handlebody}
\begin{tikzcd}[column sep=.7em]
    V^{\natural j}_{1,1}\natural V\arrow{d}{\phi_{j}}\arrow[rrrrrr,"\iota_{1+j}(V)"] &&&&&&
    V^{\natural 1+j}_{1,1}\natural V\arrow{d}{\phi_{1+j}}\arrow[rrrrrr,"\iota_{2+j}(V)"] &&&&&&
    \cdots \arrow[rrrrrr,"\iota_{i+j}(V)"] &&&&&&
    V_{1,1}^{\natural i+j}\natural V\arrow{d}{\phi_{i+j}} \arrow[rrrrrr,"\iota_{i+1+j}(V)"] &&&&&&
    \cdots \\
    V^{\natural j'}\natural V' \arrow[rrrrrr,"\iota_{1+j'}(V')"] &&&&&&
    V^{\natural 1+j'}_{1,1}\natural V' \arrow[rrrrrr,"\iota_{2+j'}(V')"] &&&&&&
    \cdots \arrow[rrrrrr,"\iota_{i+j'}(V')"] &&&&&&
    V^{\natural i+j'}_{1,1}\natural V'\arrow[rrrrrr,"\iota_{i+1+j'}(V')"]  &&&&&&
    \cdots.
\end{tikzcd}
\end{equation}
\end{itemizeb}
We also denote by $\overline{\sH}^{\ge 2}$ the full subcategory of $\overline{\sH}$ with objects those of $\sH^{\ge 2}$ along with the infinite objects $V_{\infty,1}\natural V$ with $V\in\obj(\sH^{\ge 2})$.
\end{defn}
\begin{rmk}\label{rmk:diagrammatic-object-well-defined}
It is a classical fact (see \cite[Prop.~2.13.3]{Borceux1} for instance) that a sequential diagram \eqref{eq:def-infinite-in-barGH} in the category $\langle \sH^{\ge 1},\sH^{\ge 1}]$ corresponds to a quotient of a infinite disjoint union of the form $\coprod_{i\ge 0}V_{1,1}^{\natural i}\natural(V\#V^{\# n}_{0,2})/\sim$, while their morphisms are infinite disjoint unions of diffeomorphisms up to isotopy satisfying the condition that the diagram \eqref{eq:morphism-infty-handlebody} is commutative. Such infinite disjoint unions do exist (for instance in the homotopy category of topological spaces), and so the infinite objects of Definition~\ref{def:groupoids-handlebody-infinity} and their morphisms are well-defined.
\end{rmk}
\begin{notation}\label{nota:V-g-r-s-infinity}
Following Notation~\ref{nota:V-g-r-s-object-G-H}, we write $V^{s}_{\infty,r}$ for the infinite object $V_{\infty,1}\natural V^{s}_{0,r}$ in $\overline{\sH}$ for some fixed integers $r\ge 1$ and $s\ge 0$.
\end{notation}

\begin{convention}
Assuming that manifolds are second-countable (thus implying that their cardinality is no larger than $\lvert \bR \rvert$), the categories of Definitions~\ref{def:groupoid-handlebody} and \ref{def:groupoids-handlebody-infinity}, as well as those built from them with the Quillen bracket construction of \S\ref{sss:Quillen-bracket-construction}, are essentially small following the standard general arguments of \cite[Rem.~1.2.7]{Galatius}. Namely, the objects of these categories are the disjoint unions of a finite number of $3$-balls, $3$--dimensional compact handlebodies and their sequential diagrams of the form \eqref{eq:def-infinite-in-barGH}, plus some additional data (such as marked discs, marked points, etc.).
Then, fixing a sufficiently large set $\Omega$ (i.e.~its cardinality is at least $\lvert \bR \rvert$), we consider the full subcategory $\cC$ whose objects are only those where the underlying set of a $3$-manifold is a subset of $\Omega$. This category $\cC$ is small and its inclusion into the whole (large) category is essentially surjective on objects, and so it is an equivalence. We shall thus consider this type of corresponding small category whenever we need our categories to be small, for instance when considering functor categories in \S\ref{ss:finite-degree}.
Moreover, assuming the axiom of choice, all the categories associated to handlebodies we consider have a skeleton; see \cite[Prop.~2.6.4]{Richter} for instance. Hence, up to equivalence, we may always work with these categories as if they were skeletal. 
In particular, similarly to \cite{RWW,Krannich}, we use the notation $\Aut_{\cC}(X)$ of the automorphisms of an object $X$ to also denote its class of isomorphisms in $\cC$ (which coincide when $\cC$ is skeletal).
We will not dwell further on these considerations, as these minor category theory details and mild identifications do not impact our methods and results, nor add value to our work.
\end{convention}

\paragraph*{Module structure.}
We now define a right $(\sH^{\ge 2},\#,I_{2})$-module structure over $\overline{\sH}^{\ge 2}$ as follow. We define $-\# -\colon \overline{\sH}^{\ge 2}\times\sH^{\ge 2} \to \overline{\sH}^{\ge 2}$ by the assignments of Definition~\ref{def:monoidal-structure-cG} for the objects of $\sH^{\ge 2}$, and by assigning $(V_{\infty,1}\natural V)\# W$ to be the directed colimit in the category $\langle \sH^{\ge 1},\sH^{\ge 1}]$ of the diagram
\begin{equation}\label{eq:def-sharp-infinite-object}
\begin{tikzcd}
    V\# W\arrow[rr,"\iota_{1}(V\# W)"]&& V_{1,1}\natural (V\# W)\arrow[rr,"\iota_{2}(V\# W)"]&& \cdots \arrow[rr,"\iota_{i}(V\# W)"]&& V_{1,1}^{\natural i}\natural (V\# W) \arrow[rr,"\iota_{i+1}(V\# W)"]&& \cdots
\end{tikzcd}
\end{equation}
for any infinite object $V_{\infty,1}\natural V$ and $W$ an object of $\sH^{\ge 2}$. We then set $-\# -\colon \overline{\sH}^{\ge 2}\times\sH^{\ge 2} \to \overline{\sH}^{\ge 2}$ on morphisms in the evident way, i.e.~as in Definition~\ref{def:monoidal-structure-cG} for morphisms $\sH^{\ge 2}$ and thanks to the universal property of a colimit for a diagram the form \eqref{eq:morphism-infty-handlebody} applied to \eqref{eq:def-sharp-infinite-object} for morphisms involving infinite objects in the first variable.
\begin{lem}\label{lem:right-module-infinite}
The groupoid $\overline{\sH}^{\ge 2}$ has a well-defined right $(\sH^{\ge 2},\#,I_{2})$-module structure with the above assignments.
\end{lem}
\begin{proof}
By restriction to the full subcategory $\sH^{\ge 2}\subset \overline{\sH}^{\ge 2}$, it follows from Proposition~\ref{prop:monoidal-structure-G-cH} that the operation $-\# -$ gives $\sH^{\ge 2}$ a well-defined right $(\sH^{\ge 2},\#,I_{2})$-module structure. Hence, it suffices to check that this structure holds for the infinite objects $V_{\infty,1}\natural V$ of $\overline{\sH}^{\ge 2}$. Recall that connected objects $V$ and $W$ of $\sH^{\ge 2}$ are of the form $V^{s}_{h,r}$ and $V^{s'}_{h',r'}$ for fixed $h,h',s,s'\ge 0$ and $r,r'\ge 2$, so there is a natural isomorphism $\psi_{i,W}\colon (V^{\natural i}_{1,1}\natural V)\# W \cong V^{\natural i}_{1,1}\natural (V\# W)$ obtained by composing isomorphisms of the form \eqref{eq:natural-iso-decomposition}. Then, by applying these isomorphisms $\psi_{i,W}$ for all $i\ge 0$ to the diagram \eqref{eq:def-sharp-infinite-object}, the functoriality of $-\# -$ with respect to diagrams of the form \eqref{eq:morphism-infty-handlebody} for the first summand and the morphisms of $\sH^{\ge 2}$ in the second summand, as well as the associativity and unitality axioms for a module structure, are straightforward consequences of Proposition~\ref{prop:monoidal-structure-G-cH} by some obvious diagram chasings along with the universal properties of the colimits involved in the definitions.
\end{proof}

Finally, we will make use of the following property in $\overline{\sH}^{\ge 2}$ at several points. 
\begin{lem}\label{lem:stabilising-handlebody-of-inf-rank}
For any $V\in\obj(\sH^{\ge 2})$, there exists an isomorphism in $\overline{\sH}^{\ge 2}$
\begin{equation}\label{eq:canonical-iso-dash-natural}
(V_{\infty,1}\natural V)\# V_{0,2}\cong V_{\infty,1}\natural V.
\end{equation}
\end{lem}
\begin{proof}
By Definition~\ref{def:groupoids-handlebody-infinity}, the object $V_{\infty,1}\natural V$ is the diagram \eqref{eq:def-infinite-in-barGH} with $V$ an object of $\sH^{\ge 2}$ of the form $V^{s}_{h,r}$ for some fixed $r\ge 2$, $h,s\ge 0$. Note that, for any such $V\in\obj(\sH^{\ge 2})$, there exists an isomorphism $\phi_{V}\colon V\# V_{0,2}\cong V_{1,1}\natural V$ in $\sH^{\ge 2}$ obtained by composing isomorphisms of the form \eqref{eq:natural-iso-decomposition}. Hence, we have a sequence of isomorphisms
\[
\begin{tikzcd}[column sep=.7em]
    V\# V_{0,2}\arrow{d}{\phi_{V}}[swap]{\cong}\arrow[rrr] &&& V_{1,1}\natural (V\# V_{0,2})\arrow{d}{\id_{V_{1,1}}\natural\phi_{V}}[swap]{\cong}\arrow[rrr] &&&\cdots\\
    V_{1,1}\natural V\arrow[rrr] &&&V_{1,1}^{\natural 2}\natural V\arrow[rrr] &&& \cdots
\end{tikzcd}
\]
which induces the isomorphism \eqref{eq:canonical-iso-dash-natural} by the universal properties of the involved colimits.
\end{proof}

\subsection{Finite degree coefficient systems}\label{ss:finite-degree}

We review here the notion of \emph{finite degree coefficient system}, designed for twisted coefficients in order to deal with twisted homological stability; see \S\ref{sss:HS-meta-theorem}. This is a classical concept appearing in the literature on homological stability, in particular in \cite[\S4.1--\S4.4]{RWW} and \cite[\S4.1]{Krannich} to which we refer the reader for further details.
We consider a braided monoidal groupoid $(\cG,\odot,0)$ with no zero-divisors and such that $\mathrm{Aut}_{\cG}(0)=\{\id_{0}\}$, and a right $(\cG,\odot,0)$-module $(\cM,\odot)$ such that $\cM$ is a groupoid.
In order to introduce the notion of coefficient system of finite degree, we need to work over the following groupoids:
\begin{notation}\label{nota:M-AX-G-X}
For objects $A \in\obj(\cM)$ and $X\in\obj(\cG)$, we write $A_{n}:=A\odot X^{\odot n}$ for each integer $n\ge 0$. We also denote by $\cG_{X}$ and $\cM_{A,X}$ the full subgroupoids of $\cG$ and $\cM$ on the objects $\{X^{\odot n}\}_{n\in \bN}$ and $\{A_{n}\}_{n\in \bN}$ respectively.
\end{notation}
\begin{lem}\label{lem:G-X_M-X-Quillen}
The subgroupoid $\cG_{X}$ inherits the braided monoidal structure $(\odot,0)$ of $\cG$, and the groupoid $\cM_{A,X}$ is a compatible $(\cG_{X},\odot)$-submodule of $\cM$. Furthermore, the category $[ \cM_{A,X},\cG_{X} \rangle$ is isomorphic to the full subcategory of $ [ \cM,\cG \rangle$ on the objects $\{A_{n}\}_{n\in \bN}$ denoted by $ [ \cM,\cG \rangle_{A,X}$.
\end{lem}
\begin{proof}
The first statements are straightforwardly verified from the fact that $\cG_{X}$ and $\cM_{A,X}$ are full subgroupoids of $\cG$ and $\cM$ respectively, generated by the monoidal structure $\odot$ with the object $X$ iterated over $0$ and $A$ respectively.
Then, the categories $[ \cM_{A,X},\cG_{X} \rangle$ and $[ \cM,\cG \rangle_{A,X}$ have the same objects, while $\colim_{\cG_{X}}\left(\Hom_{\cM_{X}}(A_{n}\odot -,A_{n+m})\right)$ and $\colim_{\cG_{X}}\left(\Hom_{\cM_{A,X}}(A_{n}\odot -,A_{n+m})\right)$ are obviously isomorphic for all integers $n,m\ge 0$ because $\cG_{X}$ and $\cM_{A,X}$ are full subgroupoids of $\cG$ and $\cM$ respectively. This provides a canonical isomorphism between these categories by the definition of the Quillen bracket construction in \S\ref{sss:Quillen-bracket-construction}.
\end{proof}
We fix two objects $A \in\obj(\cM)$ and $X\in\obj(\cG)$ for the remainder of \S\ref{ss:finite-degree}. A \emph{coefficient system} over the category $[ \cM_{A,X},\cG_{X} \rangle$ defined in \S\ref{sss:Quillen-bracket-construction} is a functor $ [ \cM_{A,X},\cG_{X} \rangle \to \Ab$.
For a fixed coefficient system $F$ over $[ \cM_{A,X},\cG_{X} \rangle$, it follows from \eqref{item:extension-module-structure} of Proposition~\ref{prop:general-properties-Quillen} that we have a well-defined functor $\tau_{X}(F):=F(-\odot \id_{X})\colon[ \cM_{A,X},\cG_{X} \rangle \to \Ab$ (called the \emph{translation functor} or \emph{upper suspension} of $F$) defined by assigning $A_{n}\mapsto F(A_{n} \odot X)$ on objects and $[\phi,X^{\odot m}]\mapsto [\phi,X^{\odot m}] \odot \id_{X}$ on morphisms. There is also a well-defined natural transformation $i_{X}(F)\colon F \to \tau_{X}(F)$ in $\Fct([ \cM_{A,X},\cG_{X} \rangle,\Ab)$ induced by the canonical morphisms $\{[\id_{A_{n+1}},X]\colon A_{n} \to A_{n+1}\}_{n\in \bN}$ of $[ \cM_{A,X} ,\cG_{X} \rangle$ (see \S\ref{sss:Quillen-bracket-construction}), i.e.~$i_{X}(F)_{n}:=F([\id_{A_{n+1}},X])$ for each integer $n\ge 0$. Since the functor category $\Fct([ \cM_{A,X},\cG_{X} \rangle,\Ab)$ is abelian (see \cite[\S VIII.3]{Gabriel}), the cokernel $\Cok(i_{X}(F))$ and the kernel $\Ker(i_{X}(F))$ of the natural transformation $i_{X}(F)$ are well-defined coefficient systems over $[ \cM_{A,X},\cG_{X} \rangle$, called the \emph{difference} functor and the \emph{evanescence} functor of $F$ respectively, and denoted by $\delta_{X}(F)$ and $\kappa_{X}(F)$ respectively.

\begin{defn}\label{def:finite-deg-coeff-system}
We now recursively define the notion of \emph{finite degree} for a coeﬃcient system $F\colon [ \cM_{A,X},\cG_{X} \rangle \to \Ab$ as follows. 
\begin{itemizeb}
\item A coefficient system $F$ is of finite degree $-1$ if $F=0$.
\item For an integer $d\ge 0$, a coefficient system $F$ is of finite degree less than or equal to $d$ if the natural transformation $i_{X}(F)\colon F\to F(-\odot X)$ is injective in $\Fct([ \cM_{A,X},\cG_{X} \rangle,\Ab)$ (i.e.~$\kappa_{X}(F)=0$) and the difference functor $\delta_{X}(F)$ is a finite coefficient system of degree less than or equal to $d-1$.
\end{itemizeb}
Furthermore, a coefficient system $F$ of finite degree $r$ is \emph{split} if the natural transformation $i_{X}(F)$ is \emph{split} injective in the category of coefficient systems, and the difference functor $\delta_{X}(F)$ is a split coefficient system of degree $r-1$.
\end{defn}

\begin{eg}\label{eg:Z_H_poly_functor}
The constant functor at $\bZ$ defines a split coefficient system $\bZ\colon [ \cM_{A,X},\cG_{X} \rangle \to \Ab$ of finite degree $0$.
\end{eg}

\begin{rmk}\label{rmk:generalisation-finite-degree}
The definition of finite degree coefficient systems presented here follows that of \cite[Def.~4.6]{Krannich}. That of \cite[Def.~4.10]{RWW} uses a notion of \emph{lower suspension} instead of the upper suspension $\tau_{X}(F):=F(-\odot \id_{X})$. These two suspensions may be related via the braiding of $\cG$ (see \cite[(4.2)]{RWW}) and these two approaches are equivalent; we refer the reader to \cite[Rem.~4.11,\S7.3]{Krannich} for more details. On another note, we consider here finite degree coefficient systems \emph{at $0$} in the terminology of \cite[Def.~4.6]{Krannich}, i.e.~the above defining conditions could rather be set on $F(X^{n})$ for all $n\ge N$ for a fix bound $N\in\bN$. Although we could extend the framework and results of the present paper using this slightly more general notion of coefficient systems, we prefer this restriction for simplicity and because it seems more natural.
\end{rmk}

\subsubsection{First homologies coefficient systems}\label{sss:first-homologies-coeff-systems}

Some first examples of non-trivial finite degree coefficient systems are given by the first homology groups of the surfaces and the handlebodies that we introduce in this section; see Proposition~\ref{prop:H-Sigma-V_functors}.

First of all, we denote by $\hMan$ the homotopy category of pairs of topological manifolds, where the objects are pairs $(M,A)$, where $M$ is a topological manifolds and $A\subset M$ is a submanifold, and where morphisms are homotopy classes of continuous maps of pairs $(M,A)\to (M',A')$; see \cite[\S3]{Spanier} for instance. We recall that the first relative homology defines a functor $H_{1}(-,-;\bZ) \colon \hMan \to \Ab$; see \cite[\S2.3]{Hatcherbook} for instance. We note from Definition~\ref{def:groupoid-handlebody} that $\sH$ is a category where each object $(V,r,s,i,j)$ has an underlying pair of manifolds $(V,\cD_{r}\sqcup\cP_{s})$, where $\cD_{r}=\Image(i)$ and $\cP_{s}=\Image(j)$ for $r\ge 1$ and $s\ge 0$ are subsets of $\partial V$, and whose morphisms are homotopy classes of continuous maps of pairs (plus some additional data), so there is a canonical forgetful functor $\fF\colon \sH\to \hMan$.
Furthermore, let $\partial\fF\colon \sH\to\hMan$ be the functor defined by assigning $(V,r,s,i,j)\mapsto (\partial V\setminus (\mathring{\cD}_{r}\sqcup\cP_{s}),\emptyset)$ on objects, and by restriction to the boundary on morphisms (recalling that the diffeomorphisms of smooth manifolds are boundary-preserving, and so are their homotopy classes). Therefore, we consider the composite functors
\[
H_{1}(-,-;\bZ)\circ \partial\fF \colon \sH \to \hMan \to \Ab \textrm{ and } H_{1}(-,-;\bZ)\circ \fF \colon \sH \to \hMan \to \Ab,
\]
that we denote by $H_{\Sigma}$ and $H_{\cV}$ respectively.
Namely, these functors are defined by assigning to each object $(V,r,s,i,j)\in\obj(\sH)$ the first homology groups $H_{1}(\partial V\setminus  (\mathring{\cD}_{r}\sqcup\cP_{s});\bZ)$ and $H_{1}(V, \cD_{r}\sqcup\cP_{s};\bZ)$ respectively, together with the natural action of the isotopy classes of the diffeomorphism groups.

\begin{notation}\label{nota:Z-V-V'}
We write $Z$ for either the handlebody with a marked disc $V_{1,1}$ or the $3$-ball with two marked discs $V_{0,2}$. For a fixed compact handlebody $V$ of $\sH$, taking up Notation~\ref{nota:M-AX-G-X}, we consider the monoidal groupoid $(\sH^{\ge i}_{Z},\natural_{i},I_{i})$ and right-module $(\sH^{\ge i}_{V,Z},\natural_{i})$ over it, with the following assignments:
\begin{itemizeb}
    \item if $Z=V_{1,1}$, then $i=1$;
    \item if $Z=V_{0,2}$, then $i=2$ and $r\ge 2$.
\end{itemizeb}
Although we obviously have $\sH^{\ge i}_{V,Z}=\sH_{V,Z}$ as categories, we use the notation $\sH^{\ge i}_{V,Z}$ in order to indicate which monoidal structure is used when relevant.
\end{notation}

We restrict the functors $H_{\Sigma}\colon\sH \to \Ab$ and $H_{\cV}\colon\sH \to \Ab$ to the subgroupoid $\sH_{V,Z}$ of $\sH$. Also, we note from \eqref{item:maximal-subgroupoid} of Proposition~\ref{prop:general-properties-Quillen} that there is a fully faithful canonical functor.
\begin{equation}\label{eq:maximal-subgroupoid-functor}
\fC_{Z}\colon\sH_{V,Z}\hookrightarrow [ \sH^{\ge i}_{V,Z} , \sH^{\ge i}_{Z} \rangle,
\end{equation}
defined as the identity on objects and by $\phi\mapsto [\phi,0]$ for all $\phi\in \Aut_{\sH_{V,Z}}(V\natural_{i}Z^{\natural_{i} n})$ where $n\ge 0$. The functors $H_{\Sigma}$ and $H_{\cV}$ induce coefficients systems as follows:
\begin{prop}\label{prop:H-Sigma-V_functors}
For fixed integers $d\ge 1$ $r\ge 1$ and $s\ge 0$, the $d$-fold tensor powers of the functors $H_{\Sigma}$ and $H_{\cV}$ in $\Fct(\sH_{V,Z},\Ab)$ lift along $\fC_{Z}$ to finite degree $d$ coefficient systems $[ \sH^{\ge i}_{V,Z} , \sH^{\ge i}_{Z} \rangle \to \Ab$ for each $i\in \{1,2\}$, denoted by $T^{d}H_{\Sigma}$ and $T^{d}H_{\cV}$ respectively.
Moreover, the coefficient systems $T^{d}H_{\Sigma}\colon [\sH^{\ge 1}_{V,V{1,1}}, \sH^{\ge 1}_{V{1,1}}\rangle \to \Ab$ and $T^{d}H_{\cV}\colon [\sH^{\ge i}_{V,Z}, \sH^{\ge i}_{Z}\rangle \to \Ab$ for $i \in {1,2}$ are split.
\end{prop}
\begin{proof}
Firstly, we note from \S\ref{sss:Quillen-bracket-construction} that a morphism $[\phi,Z^{\natural_{i}g'-g}]\colon V\natural_{i} Z^{\natural_{i}g} \to V\natural_{i} Z^{\natural_{i}g'}$ in the category $[ \sH^{\ge i}_{V,Z},\sH^{\ge i}_{Z} \rangle$ is the equivalence class consisting of the diffeomorphisms of $V\natural_{i} Z^{\natural_{i}g'}$ which are isotopic to $\phi$ on the closed submanifold $V$ of $V\natural_{i} Z^{\natural_{i}(g'-g)}$.
Hence, the category $[ \sH^{\ge i}_{V,Z},\sH^{\ge i}_{Z} \rangle$ is isomorphic to the category with the same objects, and whose morphisms are the isotopy classes of proper smooth embeddings $V\hookrightarrow V\natural_{i} Z^{\natural_{i}(g'-g)}$ which preserve the sets $\cD_{r}$ and $\cP_{s}$ associated to $V$; see also \cite[Prop.~1.76]{PalmerSoulie} for further details.
We deduce that there is a canonical forgetful functor $\Tilde{\fF} \colon [ \sH^{\ge i}_{V,Z},\sH^{\ge i}_{Z} \rangle \to \hMan$, such that by definition $\Tilde{\fF} \circ \fC_{Z} = \fF$.
Furthermore, by classification of surfaces (see \cite{Richards,Brown-Messer}), the space $\partial V\setminus (\mathring{\cD}_{r}\sqcup\cP_{s})$ is homeomorphic to the surface $\Sigma^{s}_{h,r}$ for some $h\ge 0$. We also recall that the monoidal structure $\natural_{i}$ of Definition~\ref{def:monoidal-structure-cG} has a clear analogous counterpart for the surfaces and their mapping class groups (that we denote in the same way), induced by gluing surfaces along half-circles rather than half-discs; see \cite[\S5.6.1]{RWW} for further details. Then, we note that a proper smooth embedding $\phi\colon V\hookrightarrow V\natural_{i} Z^{\natural_{i}(g'-g)}$ preserving $\cD_{r}\sqcup\cP_{s}$ induces by restriction to the boundary a proper smooth embedding $\varphi_{\partial}\colon \Sigma^{s}_{h,r} \hookrightarrow \Sigma^{s}_{h,r} \natural_{i} (\partial Z)^{\natural_{i}(g'-g)}$ preserving $\cP_{s}$. Hence, we define a functor $\Tilde{\partial\fF} \colon \colon [ \sH^{\ge i}_{V,Z},\sH^{\ge i}_{Z} \rangle \to \hMan$ by assigning each $(V\natural_{i} Z^{\natural_{i}(g'-g)},\cD_{r}\sqcup\cP_{s})$ to $(\Sigma^{s}_{h,r} \natural_{i} (\partial Z)^{\natural_{i}(g'-g)},\emptyset)$ for object, and sending the homotopy class of each $\varphi$ to the homotopy class of $[\varphi_{\partial}]$. Then, it is straightforward to check from these assignments that we thus have constructed a lift of $\partial\fF$ to $[ \sH^{\ge i}_{V,Z},\sH^{\ge i}_{Z} \rangle$, in the sense that $\Tilde{\partial\fF} \circ \fC_{Z} = \partial\fF$. Therefore, the functors $H_{\Sigma}$ and $H_{\cV}$ of $\Fct(\sH_{V,Z},\Ab)$ along $\fC_{Z}$ are given by the composites $H_{1}(-,-;\bZ)\circ \Tilde{\partial\fF}$ and $H_{1}(-,-;\bZ)\circ \Tilde{\fF}$ respectively, which we may postcompose by the $d$-fold tensor power functor $T^{d}\colon \Ab \to \Ab$.

Now, we denote by $H\colon \sH_{V,Z} \to \Ab$ either $H_{\Sigma}$ or $H_{\cV}$, and by $V_{g}$ the object $V\natural_{i} Z^{\natural_{i}g}$ for some integer $g\ge 0$. Let us analyse the natural transformation $i_{Z}(H)$ in each situation. First, denoting by $h$ the genus of the compact handlebody $V_{g}\natural_{i} Z$, we consider the homology classes
\begin{equation}\label{eq:standard-basis-handlebody-Sigma}
\cB^{\Sigma}_{h+1}=\{[a_{1}],[b_{1}],\ldots,[a_{h+1}],[b_{h+1}],[c_{2}],\ldots,[c_{r}],[d_{1}],\ldots,[d_{s}]\};
\end{equation}
\begin{equation}\label{eq:standard-basis-handlebody-cV}
\cB^{\cV}_{h+1}=\{[b_{1}],\ldots,[b_{h+1}],[c'_{2}],\ldots,[c'_{r}],[d'_{1}],\ldots,[d'_{s}]\}.
\end{equation}
Here:
\begin{itemizeb}
    \item $\{a_{i},b_{i}\mid 1\leq i\leq h+g+1\}$ is a standard system of \emph{meridians} and \emph{longitudes} for the surface $\partial(V_{g}\natural_{i} Z)$. Namely for each handle, these are the homotopy classes of simple closed curves in $\partial V_{1,1}$, that respectively bound a nonseparating disc in $V_{1,1}$ for the meridians and the genus for the longitudes.
    \item $\{c_{i},d_{j}\mid 2\leq i\leq r,1\leq j\leq s\}\}$ is a standard system of simple closed curves of $\partial(V_{g}\natural_{i} Z)$ bounding the $r-1$ marked discs $\{\cD^{2}_{V},\ldots,\cD^{r}_{V}\}$ (i.e.~there is no generator for the first marked disc $\cD^{1}_{V}$) and the $s$ marked points of $V_{g}\natural_{i} Z$ respectively.
    \item $\{c'_{i},d'_{j}\mid 2\leq i\leq r,1\leq j\leq s\}\}$ is a standard system of simple closed curves of $\partial(V_{g}\natural_{i} Z)$, where each $c'_{i}$ joins the marked discs $\cD^{i}_{V}$, while each $d'_{j}$ connects the marked point $p^{j}_{V}\in \cP_{s}$ to the marked disc $\cD^{1}_{V}$.
\end{itemizeb}
It is an elementary cellular homology calculation that $H(V_{g}\natural_{i} Z)$ is a finitely generated free abelian group, with a standard basis given by $\cB_{h+1}$ for $H=H_{\Sigma}$, and $\cB_{h+1}\setminus\{[a_{1}],\ldots,[a_{h+1}]\}$ for $H=H_{\cV}$.
Also, a standard basis for $H(V_{g}\natural_{i} Z)$ is similarly provided by the sets $\cB_{h}$ and $\cB'_{h}$ for $H=H_{\Sigma}$ and $H=H_{\cV}$ respectively, i.e.~the bases \eqref{eq:standard-basis-handlebody-Sigma} (resp.~\eqref{eq:standard-basis-handlebody-cV}) minus the classes $[a_{h+1}]$ and $[b_{h+1}]$ (resp.~ $[b_{h+1}]$). Futhermore, for any $n\ge 0$, for a morphism $[\phi,Z^{\natural_{i}n}]\colon V_{g}\to V_{g}\natural_{i} Z^{\natural_{i}n}$ and a representative $\tilde{\phi}$ of this isotopy class, the map $H([\phi,Z^{\natural_{i}n}]\natural_{i} \id_{Z})$ is given by the induced action on the homology classes of the embedding $\tilde{\phi}\colon V_{g}\hookrightarrow V_{g}\natural_{i} Z^{\natural_{i}n}$ extended by the identity over the complement $(n+1)$-st copy of $Z$. In particular, for $n=1$ and $\phi=\id_{V_{g}\natural_{i} Z}$, the morphism $[\id_{V_{g}\natural_{i} Z},Z]$ may be viewed as the proper smooth embedding $V_{g}\hookrightarrow V_{g}\natural_{i} Z$. Also, we note that the support of the action of $H([\phi,Z^{\natural_{i}n}]\natural_{i} \id_{Z})$ is disjoint from the submanifold $Z$. By a straightforward analysis on simple closed curves defining these bases, we deduce that $\delta_{Z}(H)$ is computed as follows.

For $i=1$: for each $g\ge 0$, there is a canonical isomorphism $\delta_{V_{1,1}}(H)(V_{g})\cong H(V_{1,1})$ and the morphism $H([\id_{V_{g}\natural V_{1,1}},V_{1,1}])$ is the canonical inclusion $H(V_{g}) \hookrightarrow H(V_{g})\oplus H(V_{1,1})$. Moreover, by disjointness of the support of the action, the morphism $H([\phi,V_{1,1}^{\natural n}]\natural \id_{V_{1,1}})$ decomposes as $H([\phi,V_{1,1}^{\natural n}])\oplus \id_{H(V_{1,1})}$ for each morphism $[\phi,V_{1,1}^{\natural n}]\colon V_{g}\to V_{g}\natural V_{1,1}^{\natural n}$. Therefore, the natural transformation $i_{V_{1,1}}(H)\colon H\to H(-\natural V_{1,1})$ is split injective in $\Fct([ \sH^{\ge 1}_{V,V_{1,1}} , \sH^{\ge 1}_{V_{1,1}} \rangle,\Ab)$, and the difference functor $\delta_{V_{1,1}}(H)$ is constant with value $H(V_{1,1})$, hence a split coefficient system of finite degree $0$ (see Example~\ref{eg:Z_H_poly_functor}).

For $i=2$: for each $g\ge 0$, the morphism $H_{\cV}([\id_{V_{g}\# V_{0,2}},V_{0,2}])$ is the canonical injection of classes $\cB'_{h}\hookrightarrow \cB'_{h+1}$, while the map $H_{\Sigma}([\id_{V_{g}\# V_{0,2}},V_{0,2}])$ is induced by the canonical injection of classes of $\cB_{h}\setminus\{[c_{2}]\}\hookrightarrow \cB_{h+1}\setminus\{[c_{2}]\}$ along with $H([\id_{V_{g}\# V_{0,2}},V_{0,2}])([c_{2}])=[a_{h+1}]$. We deduce that $\delta_{V_{0,2}}(H_{\cV})(V_{g})\cong \langle [b_{h+1}]\rangle \cong \bZ$ and $\delta_{V_{0,2}}(H_{\Sigma})(V_{g})\cong \langle [c_{2}],[b_{h+1}]\rangle \cong \bZ^{2}$, while $\kappa_{V_{0,2}}(H_{\cV})=\kappa_{V_{0,2}}(H_{\Sigma})=0$.
Furthermore, we consider the standard generating set of the mapping class group of the surface $\partial V_{g}$, given by the Lickorish generators plus the $r+s$ Dehn twists along the meridians of the $h$-th handle, bounding and separating the marked disc and marked points (similarly to \cite[Fig.~4.10]{farbmargalit}); see \cite[\S4.4.2--\S4.4.4]{farbmargalit}. The only one of these generators whose support is not disjoint from the homology classes $[c_{2}]$ and $[b_{h+1}]$ is the Dehn twist $T(c_{1,2})$ along the meridian of the $h$-th handle which separates the marked disc $\cD^{1}_{V}$ and $\cD^{2}_{V}$. By an elementary analysis on the homology classes, we compute that $T(c_{1,2})([c_{2}])=[c_{2}]$ and $T(c_{1,2})([b_{h+1}])=[b_{h+1}]+[a_{h}]$, and so the mapping class group of the surface $\partial V_{g}$ acts trivially on $\delta_{V_{0,2}}(H)(V_{g})$. Collecting all these observations and using Lemma~\ref{lem:H-injects-into-Gamma}, we conclude that the map $\delta_{V_{0,2}}(H)([\phi,V_{0,2}^{\# n}])$ is the identity map for each morphism $[\phi,V_{0,2}^{\# n}]\colon V_{g}\to V_{g}\natural V_{0,2}^{\# n}$, and that $\delta_{V_{0,2}}(H_{\cV})(V_{g})$ canonically splits when $H=H_{\cV}$. Therefore, the natural transformation $i_{V_{0,2}}(H)\colon H\to H(-\# V_{0,2})$ is injective in $\Fct([ \cM_{A,V_{0,2}},\cG_{V_{0,2}} \rangle,\Ab)$, and the difference functor $\delta_{V_{0,2}}(H)$ is constant with value $\delta_{V_{0,2}}(H)(V_{g})$, hence a split coefficient system of finite degree $0$ (see Example~\ref{eg:Z_H_poly_functor}).

Hence, in all cases, the functor $H$ is a coefficient system of finite degree $1$, and it is split when $H=H_{\cV}$, or if $H=H_{\Sigma}$ and $i=1$. Finally, since the target of the functor $H$ factors through the full subcategory of free abelian groups $\ab$, the result on the $d$-fold tensor power $T^{d}H\colon [ \sH^{\ge i}_{V,Z},\sH^{\ge i}_{Z} \rangle\to\ab$ is a classical property of the tensor product of coefficient system; see \cite[Cor.~1.5]{AnnexLS1}.
\end{proof}

\begin{rmk}
An alternative functor $\sH \to \Ab$, analogous but not identical to $H_{\Sigma}$, is defined by assigning to each object $(V,r,s,i,j)\in\obj(\sH)$ the relative homology group $H_{1}(\partial V, \mathring{\cD}_{r}\sqcup\cP_{s};\bZ)$. Repeating mutatis mutandis the proof of Proposition~\ref{prop:H-Sigma-V_functors}, this induces a coefficient system $[ \sH^{\ge i}_{V,Z} , \sH^{\ge i}_{Z} \rangle \to \Ab$ of finite degree $d$ for each $i\in \{1,2\}$. We favour the introduction of $H_{\Sigma}$, as it corresponds to a coefficient system previously studied in the context of twisted homological stability for mapping class groups of surfaces; see \cite{Ivanov}, \cite[\S1.1,Ex.~1]{CohenMadsen} and \cite[Ex.~4.3(i)]{Boldsen}.
\end{rmk}

\subsubsection{Double coefficient systems and coefficient bisystems}\label{sss:double-coeff-systems}

Finally, we introduce here the notions of \emph{double coefficient systems} and \emph{coefficient bisystems}; see Definition~\ref{def:double-coeff-system}. These are required to prove the twisted homological stability results of \S\ref{ss:handle-stabilisation-sharp} and to extend a functor over $\sH$ to some subcategories of the Quillen bracket construction $[\overline{\sH}^{\ge 2}, \sH^{\ge 2} \rangle$.

We consider a functor $F\colon \sH \to \Ab$ and a $3$--dimensional compact handlebody $V$ of $\sH$.
By composing isomorphisms of the form \eqref{eq:natural-iso-decomposition}, we have obvious isomorphisms $(V^{\natural g}_{1,1}\natural V)\# V^{\# m}_{0,2}\cong V\# V^{\# g+m}_{0,2}\cong V^{\natural g+m}_{1,1}\natural V \cong V^{\natural g}_{1,1}\natural (V\# V^{\# m}_{0,2})$ for all integers $g,m\ge 0$, and we denote by $V_{g\mid m}$ this isomorphism class of objects for brevity. We also recall from Definition~\ref{def:groupoids-handlebody-infinity} that $\iota_{g}(V_{0\mid m})$ denotes the morphism $[V_{1,1},\id_{V_{g\mid m}}]$ of $\Hom_{\langle \sH^{\ge 1},\sH^{\ge 1}]}(V_{g-1\mid m}, V_{g\mid m})$ for each $m\ge 0$. We take up the framework set in Notations~\ref{nota:M-AX-G-X} and \ref{nota:Z-V-V'}. We note that, by composing isomorphisms of the form \eqref{eq:natural-iso-decomposition}, the groupoid $\sH^{\ge 1}_{V,V_{1,1}}$ is isomorphic to the full subgroupoid of $\sH$ on the objects $\{V^{\natural n}_{1,1}\natural V\}_{n\in \bN}$, so it is naturally equipped with a left $(\sH^{\ge 1}_{V_{1,1}},\natural,I_{1})$-module structure $(\sH^{\ge 1}_{V,V_{1,1}},\natural)$.

\begin{defn}\label{def:double-coeff-system}
For $d\ge 0$ an integer and $V$ a compact handlebody of $\sH^{\ge 2}$, a functor $F\colon \sH \to \Ab$ is a \emph{double coefficient system at $V$ of degree $D$} if it satisfies the following properties.
\begin{compactenum}[(i)]
    \item\label{item:extension-right} The restriction of $F$ to $\sH^{\ge 2}_{V,V_{0,2}}$ lifts to a coefficient system $[ \sH^{\ge 2}_{V,V_{0,2}} , \sH^{\ge 2}_{V_{0,2}} \rangle  \to \Ab$ of finite degree $d$.
    \item\label{item:extension-left} The restriction of $F$ to $\sH^{\ge 1}_{V,V_{1,1}}$ lifts to a coefficient system $\langle \sH^{\ge 1}_{V_{1,1}}, \sH^{\ge 1}_{V,V_{1,1}} ]  \to \Ab$ of finite degree $d'$, and the maximum of the finite degrees $\mathrm{max}(d,d')$ is equal to $D$
    \item\label{item:extension-mixed} For fixed integers $m\ge 1$ and $n\ge 0$, the functor $F$ defines a commutative diagram of abelian groups:
    \begin{equation}\label{eq:commutative-diag-extension-functor}
    \begin{tikzcd}[column sep=.7em]
    \cdots \arrow[rrrrrrr,"F(\iota_{g-1}(V_{0\mid n}))"]
    &&&&&&&
    F(V_{g-1\mid n})\arrow[rrrrrrrrr,"F(\iota_{g}(V_{0\mid n}))"]\arrow{d}{F([\id_{V_{g-1\mid n+m}}, V^{\# m}_{0,2}])}
    &&&&&&&&& F(V_{g\mid n})\arrow[rrrrrrr,"F(\iota_{g+1}(V_{0\mid n}))"]\arrow{d}{F([\id_{V_{g\mid n+m}}, V^{\# m}_{0,2}])}
    &&&&&&& \cdots
    \\
    \cdots \arrow[rrrrrrr,"F(\iota_{g+1}(V_{0\mid n+m}))"']
    &&&&&&& F(V_{g-1\mid n+m})\arrow[rrrrrrrrr,"F(\iota_{g}(V_{0\mid n+m}))"']
    &&&&&&&&& F(V_{g\mid n+m}) \arrow[rrrrrrr,"F(\iota_{g+1}(V_{0\mid n+m}))"']
    &&&&&&& \cdots
    \end{tikzcd}
    \end{equation}
    Namely, the vertical and horizontal arrows are well-defined by Conditions~\eqref{item:extension-right} and \eqref{item:extension-left} respectively, and we assume that the assignments of $F$ are such that the diagram \eqref{eq:commutative-diag-extension-functor} is commutative.
\end{compactenum}
Moreover, if the restriction of $F$ to $\sH^{\ge 1}_{V,V_{1,1}}$ also lifts to a coefficient system $[ \sH^{\ge 1}_{V,V_{1,1}} , \sH^{\ge 1}_{V_{1,1}} \rangle  \to \Ab$ of finite degree $D'$, then we say that the functor $F\colon \sH \to \Ab$ is a \emph{coefficient bisystem} at $V$ of degree $\mathrm{max}(D,D')$.
\end{defn}
\begin{rmk}
Although Definition~\ref{def:double-coeff-system}\eqref{item:extension-mixed} is mild and easy to check in our examples (see Example~\ref{eg:extending-functors-first-homologies}), it is not automatic and does not simply follow from the functoriality of $F$. Indeed, an analogous commutative diagram to \eqref{eq:commutative-diag-extension-functor} before applying $F$ does not make sense, since the vertical maps are images of morphisms in $[ \sH^{\ge 2}_{V,V_{0,2}} , \sH^{\ge 2}_{V_{0,2}} \rangle$, while the horizontal ones come from $\langle \sH^{\ge 1}_{V_{1,1}}, \sH^{\ge 1}_{V,V_{1,1}} ]$. We believe it would be possible to introduce a generalisation of \S\ref{sss:Quillen-bracket-construction} encompassing these two categories, which would contain the necessary morphisms for the diagram \eqref{eq:commutative-diag-extension-functor} before applying $F$ to be well-defined and commutative, but this is out of the scope of this paper. Furthermore, the additional property for a double coefficient system to be a coefficient bisystem generally follows automatically from Definition~\ref{def:double-coeff-system}\eqref{item:extension-left}; see Example~\ref{eg:extending-functors-first-homologies} for an illustration.
\end{rmk}
\begin{eg}\label{eg:extending-functors-first-homologies}
For each $d\ge 1$, the functors $T^{d}H_{\Sigma}$ and $T^{d}H_{\cV}$ of $\Fct(\sH,\Ab)$ introduced in \S\ref{sss:first-homologies-coeff-systems} are both coefficient bisystems of degree $d$ at any compact handlebody. Indeed, Proposition~\ref{prop:H-Sigma-V_functors} with $Z=V_{0,2}$ and $Z=V_{1,1}$ respectively implies that Definition~\ref{def:double-coeff-system}\eqref{item:extension-right} and the further property to have a coefficient bisystem are verified. Also, a mutatis mutandis routine adaptation of the first part of the proof of Proposition~\ref{prop:H-Sigma-V_functors} with $Z=V_{1,1}$ shows that Definition~\ref{def:double-coeff-system}\eqref{item:extension-left} is also satisfied. Moreover, for all integers $g\ge 1$ and $n\ge 0$, denoting by $H$ either $H_{\Sigma}$ or $H_{\cV}$, it directly follows from their assignments that the composites $T^{d}H(\iota_{g}(V_{0\mid n+m}))\circ T^{d}H([\id_{V_{g-1\mid n+m}}, V^{\# m}_{0,2}])$ and $T^{d}H([\id_{V_{g\mid n+m}}, V^{\# m}_{0,2}]) \circ T^{d}H(\iota_{g}(V_{0\mid n}))$ are both the maps induced in homology by the proper smooth embedding $V_{g-1 \mid n} \hookrightarrow V_{1,1} \natural V_{g-1 \mid n} \# V^{\# m}_{0,2}$, and so they are equal.
Hence the abelian group diagram \eqref{eq:commutative-diag-extension-functor} is commutative, and thus Definition~\ref{def:double-coeff-system}\eqref{item:extension-mixed} is satisfied.
\end{eg}
We consider a double coefficient system $F\colon \sH \to \Ab$ at a compact handlebody $V$ of $\sH^{\ge 2}$ of degree $d$. We extend the restriction of $F$ along $\sH^{\ge 2}\hookrightarrow \sH$ to the groupoid $\overline{\sH}^{\ge 2}$ introduced in Definition~\ref{def:groupoids-handlebody-infinity}, denoting it by $\overline{F}$, as follows. Recall that the category of abelian groups $\Ab$ is cocomplete (see \cite[\S VIII.3.]{MacLane1} for instance) and that the functor $F$ defines a coefficient system $\langle \sH^{\ge 1}_{V_{1,1}}, \sH^{\ge 1}_{V,V_{1,1}} ]  \to \Ab$ by Definition~\ref{def:double-coeff-system}\eqref{item:extension-left}.
\begin{itemizeb}
    \item For each integer $n\ge 0$, we assign $\overline{F}((V_{\infty,1}\natural V)\# V^{\# n}_{0,2})$ to be the colimit of abelian groups defined by the diagram
\begin{equation}\label{eq:assignment-extension-infinity}
\begin{tikzcd}
    \cdots\arrow[rr,"F(\iota_{g}(V_{0\mid n}))"]&& F(V_{g\mid n})\arrow[rr,"F(\iota_{g+1}(V_{0\mid n}))"]&& F(V_{g+1\mid n}) \arrow[rr,"F(\iota_{g+2}(V_{0\mid n}))"]&& \cdots.
\end{tikzcd}
\end{equation}
    In other words, we take the colimit of the sequence defined by applying $F$ to \eqref{eq:def-infinite-in-barGH}.
    \item For an isomorphism $\phi$ of $(V_{\infty,1}\natural V)\# V^{\# n}_{0,2}$, by Definition~\ref{def:groupoids-handlebody-infinity}, there exists an integer $j\ge 0$ and a set of morphisms $\{\phi_{i}\in \Hom_{\sH}(V^{\natural i+j}_{1,1}\natural (V\# V^{\# n}_{0,2}),V^{\natural i+j}_{1,1}\natural (V\# V^{\# n}_{0,2}))\}_{i\ge 0}$ such that the associated diagram of the form \eqref{eq:morphism-infty-handlebody} is commutative. We assign $\overline{F}(\phi)$ to be the colimit induced by the sequence of abelian group morphisms $\{F(\phi_{i})\}_{i\ge 0}$, i.e.~the image by applying $F$ of a diagram of the form \eqref{eq:morphism-infty-handlebody}.
    \item By Definition~\ref{def:double-coeff-system}\eqref{item:extension-mixed}, we assign $\overline{F}([\id_{(V_{\infty,1}\natural V)\# V^{\# n+m}_{0,2}},V^{\# m}_{0,2}])$ for each $m\ge 1$ and $n\ge 0$ to be the colimit of the vertical abelian group maps defined by diagram \eqref{eq:commutative-diag-extension-functor}. Since any morphism of $[\overline{\sH}^{\ge 2}_{V_{\infty,1}\natural V,V_{0,2}},\sH^{\ge 2}_{V_{0,2}}\rangle$ is of the form $[\phi,V^{\# m}_{0,2}]=\phi\circ [\id_{(V_{\infty,1}\natural V)\# V^{\# m}_{0,2}},V^{\# n+m}_{0,2}]$ with $\phi$ an automorphism of $(V_{\infty,1}\natural V)\# V^{\# n+m}_{0,2}$ and $m,n\ge 0$, we assign $\overline{F}([\phi,V^{\# m}_{0,2}])=\overline{F}(\phi)\circ \overline{F}([\id_{(V_{\infty,1}\natural V)\# V^{\# n+m}_{0,2}},V^{\# m}_{0,2}])$ if $m\ge 1$ while $\overline{F}(\phi)$ is defined at the previous point otherwise.
\end{itemizeb}
The composition axiom for the above assignments on morphisms is obviously satisfied by construction, so we have a well-defined functor $\overline{F}\colon [\overline{\sH}^{\ge 2}_{V_{\infty,1}\natural V,V_{0,2}},\sH^{\ge 2}_{V_{0,2}}\rangle \to \Ab$.

\begin{prop}\label{prop:extension-coeff-system-infty}
The coefficient system $\overline{F} \colon [\overline{\sH}^{\ge 2}_{V_{\infty,1}\natural V,V_{0,2}},\sH^{\ge 2}_{V_{0,2}}\rangle \to \Ab$ is of finite degree $d$.
\end{prop}
\begin{proof}
By the above construction, the map $i_{V_{0,2}}(\overline{F})_{(V_{\infty,1}\natural V)\# V^{\# n}_{0,2}}$ for each integer $n\ge 0$ is the colimit of the abelian group morphisms $\{i_{V_{0,2}}(F)_{V_{g\mid n}}=F([\id_{V_{g\mid n+1}}, V_{0,2}])\}_{g\ge 0}$ defining the commutative diagram \eqref{eq:commutative-diag-extension-functor} with $m=1$. We recall that filtered colimits commute with kernels and cokernels; see \cite[\S IX.2.]{MacLane1} for instance. Then, by considering the colimit of the cokernels and kernels of the vertical maps in the commutative diagram \eqref{eq:commutative-diag-extension-functor} with $m=1$, it follows from Definition~\ref{def:double-coeff-system}\eqref{item:extension-right} that $\delta_{V_{0,2}}\overline{F}$ is a coefficient system of finite degree at most $d-1$ and that $\kappa_{V_{0,2}}\overline{F}$ vanishes.
\end{proof}

\subsection{Framework for twisted homological stability}\label{ss:framework-HS}

This section presents the meta-method we follow to prove our homological stability results for handlebody groups. For all of \S\ref{ss:framework-HS}, we consider a braided monoidal groupoid $(\cG,\odot,0)$ with no zero-divisors and such that $\mathrm{Aut}_{\cG}(0)=\{\id_{0}\}$, and a right-module $(\cM,\odot)$ over it, where $\cM$ is a groupoid.
We also fix objects $A \in\obj(\cM)$ and $X\in\obj(\cG)$, and a coefficient system $F\colon [ \cM_{A,X},\cG_{X} \rangle\to\Ab$. Recall from Notation~\ref{nota:M-AX-G-X} that $\cG_{X}$ and $\cM_{A,X}$ respectively denote the full subgroupoids of $\cG$ and $\cM$ on the objects $\{X^{\odot n}\}_{n\in \bN}$ and $\{A_{n}\}_{n\in \bN}$, where $A_{n}=A\odot X^{\odot n}$ for each $n\ge 0$.
\begin{notation}\label{nota:standard-notation-HS}
For each integer $n\ge 0$, we write $G_{n}$ for the group $\Aut_{\cM}(A_{n})$, $F_{n}$ for the abelian group $F(A_{n})$ and $c_{n}$ for the canonical morphism $[\id_{A_{n+1}},X]\in\Hom_{[ \cM_{A,X} ,\cG_{X} \rangle}(A_{n},A_{n+1})$ (see \S\ref{sss:Quillen-bracket-construction}).
\end{notation}

\subsubsection{Twisted homological stability meta-theorem}\label{sss:HS-meta-theorem}

We recall here the general framework from \cite{RWW} and \cite[\S7]{Krannich} to prove homological stability for sequences of groups coming from a monoidal groupoid as in \S\ref{sss:categorical_framework}; see Theorem~\ref{thm:RWW-main-thm}.

For each integer $n\ge 0$, the right $(\cG,\odot,0)$-module structure of $\cM$ induces a canonical group morphism $\st_{n}:=(-)_{n}\odot\id_{X}\colon G_{n}\to G_{n+1}$ called a \emph{stabilisation map}, while the functoriality of $F$ ensures that the morphism $F(c_{n})$ is $G_{n}$-equivariant along the map $\st_{n}$. Therefore, it follows from the functoriality of group homology (see \cite[Chap.~III,\S8]{Brown} for instance) that the maps $\st_{n}$ and $F(c_{n})$ induce a map
\begin{equation}\label{eq:stabilisation-HS}
(\st_{n};F(c_{n}))_{i}\colon H_{i}(G_{n};F_{n})\to H_{i}(G_{n+1};F_{n+1}).
\end{equation}
for each $i\ge 0$ and $n\ge 0$.
Theorem~\ref{thm:RWW-main-thm} roughly states that if $F$ is of finite degree, the map \eqref{eq:stabilisation-HS} is an isomorphism in a range determined by the connectivity of a semi-simplicial set associated to this data.

\paragraph*{Local homogeneity.} The framework to prove homological stability for the family of groups $\{G_{n}\}_{n\in \bN}$ requires the mild assumption that the groupoid $\cM$ has the following property.
\begin{defn}\label{def:homogeneity}
We say that $(\cM,\odot)$ is \emph{locally homogeneous} at $(A,X)$ if the two following conditions hold:
\begin{compactenum}[(i)]
\item\label{(I)} \emph{Injectivity}: for each $n\ge 1$ and for each $0\le p<n$, the map $(-)_{n-p-1}\odot\id_{X^{\odot p+1}}\colon G_{n-p-1}\to G_{n}$ is injective.
\item\label{(C)} \emph{Cancellation}: for each $0\le p<n$, if $Y\in \obj(\cM)$ is such that $Y\odot X^{\odot p+1}\cong A_{n}$, then $Y\cong A_{n-p-1}$.
\end{compactenum}
\end{defn}
\begin{rmk}\label{rmk:equivalence-homogeneity}
The notion of local homogeneity was originally introduced in \cite[Def.~1.2]{RWW}, with conditions equivalent to those of injectivity \eqref{(I)} and cancellation \eqref{(C)} of Definition~\ref{def:homogeneity}; see \cite[Th.~1.10(a)--(b)]{RWW}.
\end{rmk}

\paragraph*{Semi-simplicial set of destabilisations.} The main hypothesis to prove homological stability for the map \eqref{eq:stabilisation-HS} is the high-connectivity of a semi-simplicial set associated to the sets of morphisms of $\cM$ that we now introduce.
\begin{defn}\label{def:semi-simplicial-set-destabilisations}
For any $n\ge 0$, we define the semi-simplicial set of \emph{destabilisations} $W_{n}(A,X)_{\bullet}$ to be the semi-simplicial set whose $p$-simplices are given by
\[
W_{n}(A,X)_{p}:=\colim_{\cM}\left(\Hom_{\cM}(- \odot X^{\odot p+1},A_{n})\right),
\]
for any $0\le p < n$. By an analysis similar to that of the Hom-sets in \S\ref{sss:Quillen-bracket-construction}, we note that any element of $W_{n}(A,X)_{p}$ is thus an equivalence class of pairs $(B,\phi)$, denoted by $[B,\phi]$, where $B\in\obj(\cG)$ and $\phi \in \Aut_{\cG}(B\odot X^{\odot p+1},A_{n})$. The face maps for $0\le i\le p$ are then defined as
\begin{align*}
d_{i}\colon W_{n}(A,X)_{p}&\to W_{n}(A,X)_{p-1}\\
[B,\phi]&\mapsto [B,\phi]\circ [X,(b^{\cG}_{X^{\odot i},X})^{-1}\odot\id_{X^{\odot p-i}}].
\end{align*}

\begin{rmk}\label{rmk:colim-as-Hom+interpretation-W-hom-Quillen}
When we consider $\cG=\cM$ equipped with the canonical right-module structure over itself, there is a canonical isomorphism $W_{n}(A,X)_{p} \cong \Hom_{\langle \cG,\cG ]}(X^{\odot p+1},A_{n})$ for all integers $n\ge 1$ and $0\le p < n$; see \S\ref{sss:Quillen-bracket-construction}.
Also, thanks to the formula of \eqref{item:extension-monoidal-structure} in Proposition~\ref{prop:general-properties-Quillen}, we note that the morphism $[X,(b^{\cG}_{X^{\odot i},X})^{-1}\odot\id_{X^{\odot p-i}}]$ defining the face map $d_{i}$ is equal to the morphism $\id_{X^{\odot i}}\odot [X,\id_{X}]\odot \id_{X^{\odot p-i}}$ of $\langle \cG,\cG ]$, thus matching \cite[Def.~2.1]{RWW}.
\end{rmk}

\paragraph*{Twisted homological stability.}
We now introduce the twisted homological stability ``meta-theorem'' for a general family of groups, that will be applied to prove our results for handlebody groups in \S\ref{s:homological_stability}.
\end{defn}
\begin{thm}\label{thm:RWW-main-thm}
Assume that $\cM$ is locally homogeneous at $(A,X)$ and that there exist integers $a\ge2$ and $k\ge 2$ such that $W_{n}(A,X)_{\bullet}$ is at least $\frac{n-a}{k}$-connected for each $n\ge 1$. If $F\colon [ \cM_{A,X},\cG_{X} \rangle\to\Ab$ is a coefficient system of degree $d$, then
\[(\st_{n};F(c_{n}))_{i} \colon H_{i}(G_{n};F_{n})\to H_{i}(G_{n+1};F_{n+1})\]
is an isomorphism for each $i\le \frac{n+2-a}{k}-d-1$. If $F$ is in addition split, it is an isomorphism for each $i\le \frac{n+2-a-d}{k}-1$.
\end{thm}
\begin{proof} The case $a=2$ is given by \cite[Thm.~C]{Krannich} as follows. By their definitions in Notation~\ref{nota:M-AX-G-X}, the groupoids $\cG_{X}$ and $\cM_{A,X}$ have canonical functors $\mathfrak{g}_{\cG}\colon \cG_{X}\to\bN$ and $\mathfrak{g}_{\cM}\colon \cM_{A,X}\to\bN$ respectively, compatible with each other in the sense that $\mathfrak{g}_{\cM}(A_{n})= \mathfrak{g}_{\cM}(A_{n-1})+g_{\cG}(X)$ for all $n\ge 1$. As explained at the beginning of \cite[\S7.2]{Krannich}, taking the geometric realisations of $\cM_{A,X}$ and $\cG_{X}$ gives us a graded $E_{1}$-module $\B\cM_{A,X}$ over the graded  $E_{2}$-algebra $\B\cG_{X}$, which is a suitable input for \cite[Thm.~C]{Krannich}. The ranges of stability in this theorem are determined by the connectivity of the \emph{canonical resolution} $R_{\bullet}(\B\cM_{A,X})\to \B\cM_{A,X}$ defined in \cite[\S2]{Krannich}: they hold precisely when, for each $n\ge 1$, the restriction $|R_{\bullet}(\B\cM_{A,X})_{n}|\to (\B\cM_{A,X})_{n}$ to the preimage of the degree $n$ part of $\B\cM_{A,X}$ is $\frac{n-2+k}{k}$-connected, i.e.~the homotopy fibre of $R_{\bullet}(\B\cM_{A,X})\to \B\cM_{A,X}$ is $\frac{n-2}{k}$-connected; see \cite[Rem.~2.17]{Krannich}. By \cite[Lem.~7.6]{Krannich}, the injectivity condition \eqref{(I)} of Definition~\ref{def:homogeneity} of $\cM$ at $(A,X)$ implies that this homotopy fibre is homotopy equivalent to $|W_{n}(A,X)_{\bullet}|$, giving us the stated result.

If $a>2$, we set $A':=A\odot X^{\odot(a-2)}$ and note that $W_{n}(A',X)_{\bullet}=W_{n+2-a}(A,X)_{\bullet}$ by definition, so the result follows from applying the first case to $W_{n}(A',X)_{\bullet}$.
\end{proof}

\begin{rmk}
The Cancellation condition~\eqref{(C)} of Definition~\ref{def:homogeneity} is not necessary to apply \cite[Thm.~C]{Krannich} and thus prove Theorem~\ref{thm:RWW-main-thm}. However, this assumption is needed along with local standardness (see \S\ref{sss:simplicial-complex-destabilisations}) in order to prove Proposition~\ref{prop:LS-iff-Wn-determined-by-vertices} by applying \cite[Prop.~2.6]{RWW}; see also \cite[\S7.3]{Krannich}. Hence we choose to keep that assumption in the standard framework for homological stability in Theorem~\ref{thm:RWW-main-thm}.
\end{rmk}

\subsubsection{Categorical standardness and simplicial complex of destabilisations}\label{sss:simplicial-complex-destabilisations}

The main difficulty in applying Theorem~\ref{thm:RWW-main-thm} is to check the connectivity condition on the semi-simplicial set $W_{n}(A,X)_{\bullet}$ of Definition~\ref{def:semi-simplicial-set-destabilisations}. For this purpose, we present here properties and tools which will be of key use, walking in the footsteps of \cite[\S2.1]{RWW}.

\paragraph*{Local standardness property.}
We recall from \cite[Def.~2.5]{RWW} the notion of local standardness for the groupoid $\cM$, which provides a description of the simplices of the semi-simplicial set $W_{n}(A,X)_{\bullet}$; see Proposition~\ref{prop:LS-iff-Wn-determined-by-vertices}. We fix objects $A \in\obj(\cM)$ and $X\in\obj(\cG)$ and take up Notation~\ref{nota:M-AX-G-X}.
Similarly to the hom-sets of \S\ref{sss:Quillen-bracket-construction}, we note that an element of $\colim_{\cM}(\Hom_{\cM}(- \odot X,A_{n}))$ is an equivalence class $[A_{n-1},\phi]$ of the pair $(A_{n-1},\phi)$ with $\phi \in G_{n}$, where the pairs $(A_{n-1},\phi)$ and $(A_{n-1},\phi')$ are equivalent if there exists an isomorphism $\chi \in G_{n-1}$ such that $\phi'=\phi\circ (\chi\odot\id_{X})$. Hence there is a natural bijection for each $n\ge 1$
\begin{equation}\label{eq:colim-as-quotient-Aut}
\colim_{\cM}\left(\Hom_{\cM}(- \odot X,A_{n})\right)\cong G_{n}/G_{n-1},
\end{equation}
where $G_{n}/G_{n-1}$ denotes the quotient set associated to the action of $G_{n-1}$ by precomposition on the left on $G_{n}$ via the map $(-)_{n-1}\odot\id_{X}\colon G_{n-1}\to G_{n}$.

For each integer $n\ge 2$, we consider the morphism
\[
\theta^{n}_{A,X}\colon \colim_{\cM}\left(\Hom_{\cM}(- \odot X_{1},A_{n-2}\odot X_{1})\right) \to \colim_{\cM}\left(\Hom_{\cM}(- \odot X_{1},A_{n-2}\odot X_{1}\odot X_{2})\right)
\]
defined by $[A_{n-2},\phi]\mapsto [A_{n-2}\odot X_{2},(\phi\odot \id_{X_{2}})\circ(\id_{A_{n-2}}\odot (b^{\cG}_{X_{1} , X_{2}})^{-1})]$, where we write $A_{n} \cong A_{n-2}\odot X_{1}\odot X_{2} \cong A_{n-2}\odot X_{2}\odot X_{1}$ with $X_{1}=X_{2}=X$ as a notational device to distinguish the two copies of $X$ in the monoidal product.

\begin{defn}\label{def:local-standardness}
We consider objects $A \in\obj(\cM)$ and $X\in\obj(\cG)$ such that $\cM$ is locally homogeneous at $(A,X)$. We say that $\cM$ is \emph{locally standard} at $(A,X)$ if the two following conditions hold.
\begin{compactenum}[(i)]
\item\label{(LS1)} The maps $[A\odot X,\id_{A}\odot (b^{\cG}_{X,X})^{-1}]$ and $[A\odot X,\id_{A_{2}}]$ are distinct in $\colim_{\cM}\left(\Hom_{\cM}(- \odot X,A_{2})\right)$.
\item\label{(LS2)} For all $n\ge 2$, the map $\theta^{n}_{A,X}$ is injective.
\end{compactenum} 
\end{defn}

\begin{rmk}\label{rmk:Theta-alternative-formula}
Following Remark~\ref{rmk:colim-as-Hom+interpretation-W-hom-Quillen}, we consider the case where $\cG=\cM$ equipped with the canonical right-module structure over itself. By the formula of \eqref{item:extension-monoidal-structure} in Proposition~\ref{prop:general-properties-Quillen}, we deduce that $[A\odot X,\id_{A}\odot (b^{\cG}_{X,X})^{-1}]= [A,\id_{A}] \odot \id_{X} \odot [X,\id_{X}]$ and that $\theta^{n}_{A,X}([A_{n-2},\phi])=[A_{n-2},\phi]\odot [X_{2},\id_{X_{2}}]$ as morphisms of $\langle \cG,\cG ]$. In particular, Conditions~\eqref{(LS1)} and \eqref{(LS2)} exactly match the conditions $\textbf{LS1}$ and $\textbf{LS2}$ of \cite[Def.~2.5]{RWW} respectively.
\end{rmk}

The following Lemma~\ref{lem:pullback-standardness-LS2} provides an equivalent property to Condition~\eqref{(LS2)} of Definition~\ref{def:local-standardness}, which will be more convenient to check in practice; see Proposition~\ref{prop:local-standardness}.
Beforehand, we note that since the braiding $b^{\cG}_{-,-}$ is a natural isomorphism in $\cG$, it induces an automorphism for each $n\ge 2$
\[
\fC_{n}(X_{1} , X_{2})\colon \Aut_{\cM}(A_{n-1}\odot X_{2} \odot X_{1})\overset{\cong}{\longrightarrow} \Aut_{\cM}(A_{n-1}\odot X_{1} \odot X_{2}),
\]
defined by $\phi \mapsto (\id_{A_{n-2}}\odot b^{\cG}_{X_{1} , X_{2}})\circ \phi \circ (\id_{A_{n-2}}\odot (b^{\cG}_{X_{1} , X_{2}})^{-1})$.
\begin{lem}\label{lem:pullback-standardness-LS2}
Condition~\eqref{(LS2)} of Definition~\ref{def:local-standardness} is satisfied if and only if the following diagram is a pullback square for each $n\ge 2$
\begin{equation}\label{eq:pullback-standardness-LS2}
\centering
\begin{split}
\begin{tikzpicture}
[x=1mm,y=1mm]
\node (tl) at (0,15) {$\Aut_{\cM}(A_{n-2})$};
\node (tr) at (80,15) {$\Aut_{\cM}(A_{n-2}\odot X_{2})$};
\node (bl) at (0,0) {$\Aut_{\cM}(A_{n-2}\odot X_{1})$};
\node (br) at (80,0) {$\Aut_{\cM}(A_{n-2}\odot X_{2} \odot X_{1})$,};
\draw[->] (tl) to node[above,font=\small]{$(-)_{n-2}\odot\id_{X_{2}}$} (tr);
\draw[->] (bl) to node[above,font=\small]{$\fC_{n}(X_{1} , X_{2})\circ ((-)_{n-1}\odot\id_{X_{2}})$} (br);
\draw[->] (tl) to node[left,font=\small]{$(-)_{n-2}\odot\id_{X_{1}}$} (bl);
\draw[->] (tr) to node[right,font=\small]{$(-)_{n-1}\odot\id_{X_{1}}$} (br);
\node at (5,10) {$\lrcorner$};
\end{tikzpicture}
\end{split}
\end{equation}
\end{lem}
\begin{proof}
We fix $n\ge 2$. First of all, we note that $\id_{X_{1}\odot X_{2}} \circ (b^{\cG}_{X_{1},X_{2}})^{-1}=(b^{\cG}_{X_{1},X_{2}})^{-1}\circ \id_{X_{2}\odot X_{1}}$ by definition of the braiding $b^{\cG}_{-,-}$. This implies the commutativity of the diagram \eqref{eq:pullback-standardness-LS2} by definition of the conjugate $\fC_{n}(X_{1} , X_{2})$.
Furthermore, using the natural bijection \eqref{eq:colim-as-quotient-Aut}, it follows from the universal property of a quotient set that the morphism $\theta^{n}_{A,X}$ is the induced map between the quotients of the vertical maps of the diagram \eqref{eq:pullback-standardness-LS2}.

Now, let us prove the equivalence between Condition~\eqref{(LS2)} and the fact that the commutative diagram \eqref{eq:pullback-standardness-LS2} is a pullback square. We recall that the category of groups is complete (see \cite[\S V.1.]{MacLane1} for instance), so the pullback $P$ of the maps $\fC_{n}(X_{1} , X_{2})\circ ((-)_{n-1}\odot\id_{X_{2}})$ and $(-)_{n-1}\odot\id_{X_{1}}\colon \Aut_{\cM}(A_{n-2}\odot X_{2})\to \Aut_{\cM}(A_{n-2}\odot X_{2} \odot X_{1})$ targeting the bottom-right corner of the diagram \eqref{eq:pullback-standardness-LS2} is well-defined.
By diagram chasing, it follows from the universal properties of the kernel $ \Ker(\theta^{n}_{A,X})$ and of the pullback $P$ that there exist unique maps $\Upsilon\colon\Aut_{\cM}(A_{n-2}) \to P$ and $p\colon P \to \Ker(\theta^{n}_{A,X})$ respecting the kernel, pullback and quotient set structures. In particular, a clear diagram chasing factoring through the quotient $\Aut_{\cM}(A_{n-2}\odot X_{2})/\Aut_{\cM}(A_{n-2})$ shows that the composite $p\circ \Upsilon$ is the zero group morphism $0_{\Grp}$.

Moreover, picking a lift $\xi\colon \Ker(\theta^{n}_{A,X}) \to \Aut_{\cM}(A_{n-2}\odot X_{1})$ of the canonical inclusion $\Ker(\theta^{n}_{A,X}) \subset \colim_{\cM}\left(\Hom_{\cM}(- \odot X_{1},A_{n-2}\odot X_{1})\right)$, the composite $\fC_{n}(X_{1} , X_{2})\circ ((-)_{n-1}\odot\id_{X_{2}})\circ \xi$ then lifts to $\Aut_{\cM}(A_{n-2}\odot X_{2})$ along $(-)_{n-1}\odot\id_{X_{1}}$, because its image in the quotient $\Aut_{\cM}(A_{n-2}\odot X_{2}\odot X_{1})/\Aut_{\cM}(A_{n-2}\odot X_{2})$ is zero by the commutativity of the diagram \eqref{eq:pullback-standardness-LS2}.
Hence, by the universal property of the pullback $P$, there exists a unique map $\xi'\colon \Ker(\theta^{n}_{A,X}) \to P$ respecting the pullback structure. Then, it follows from a diagram chasing and the universal property of the kernel $\Ker(\theta^{n}_{A,X})$ that $p\circ \xi'=\id_{\Ker(\theta^{n}_{A,X})}$, and so the morphism $p$ is surjective.

By the universal property of the pullback $P$, if the commutative diagram \eqref{eq:pullback-standardness-LS2} is a pullback square, then the map $\Upsilon$ is an isomorphism. This along with the fact that $p\circ \Upsilon=0_{\Grp}$ implies that $p=0_{\Grp}$. It then follows from the surjectivity of $p$ that $\Ker(\theta^{n}_{A,X})=0$, i.e.~Condition~\eqref{(LS2)} is satisfied.
Conversely, if $\theta^{n}_{A,X}$ is injective, then the composite $P\overset{\upsilon}{\to} \Aut_{\cM}(A_{n-2}\odot X_{1})\twoheadrightarrow \Aut_{\cM}(A_{n-2}\odot X_{1})/\Aut_{\cM}(A_{n-2})$ (where $\upsilon$ is the defining map of the pullback $P$, while the second one is the canonical surjection induced by \eqref{eq:colim-as-quotient-Aut}) is equal to $0_{\Grp}$. So there exists a lift $\upsilon'\colon P \to \Aut_{\cM}(A_{n-2})$ of $\upsilon$, such that $\upsilon = ((-)_{n-2}\odot\id_{X_{1}}) \circ \upsilon'$, and we know by definition that $\upsilon \circ \Upsilon = (-)_{n-2}\odot\id_{X_{1}}$. By diagram chasings, we deduce from the universal property of the pullback $P$ that $\Upsilon\circ \upsilon'=\id_{P}$, while it follows from the injectivity of $(-)_{n-2}\odot\id_{X_{1}}$ (by the injectivity condition \eqref{(I)} of Definition~\ref{def:homogeneity}, since $(\cM,\odot)$ is locally homogeneous at $(A,X)$) that $\upsilon'\circ \Upsilon=\Aut_{\cM}(A_{n-2})$. Hence $\Upsilon$ is an isomorphism respecting the pullback structure, and so the commutative diagram \eqref{eq:pullback-standardness-LS2} is a pullback square.
\end{proof}
Furthermore, the following result highlights the significance of the notion of local standardness:
\begin{prop}\label{prop:LS-iff-Wn-determined-by-vertices}
We consider objects $A \in\obj(\cM)$ and $X\in\obj(\cG)$ such that $\cM$ is locally homogeneous at $(A,X)$. Then the right $(\cG,\odot,0)$-module $(\cM,\odot)$ is locally standard at $(A,X)$ if and only if all simplices of $W_{n}(A, X)_{\bullet}$ for all $n$ are determined by their (ordered sets of) vertices and their vertices are all distinct. 
\end{prop}
\begin{proof}
This result corresponds to \cite[Prop.~2.6]{RWW}, in particular using \cite[Lem.~2.7]{RWW} which gives a equivalent characterisation of Condition~\eqref{(LS2)} of Definition~\ref{def:local-standardness}. The proofs of these results repeat verbatim for the present framework equivalent to that of \cite[\S2.1]{RWW}.
\end{proof}

\paragraph*{Simplicial complex of destabilisations.} In the cases we consider in \S\ref{s:homological_stability}, we see that the connectivity of the semi-simplicial set $W_{n}(A,X)_{\bullet}$ of Definition~\ref{def:semi-simplicial-set-destabilisations} is determined by that of its underlying simplicial complex; see \S\ref{sss:connectivity-ss} and \S\ref{sss:infinite-HS}. We introduce the following notation:
\begin{notation}\label{nota:simplicial-complex-destabilisation}
For each integer $n\ge 0$, we write $S_{n}(A,X)_{\bullet}$ for the underlying simplicial complex of $W_{n}(A,X)_{\bullet}$. In other words, this is the simplicial complex whose set of vertices is given by
\[
S_{n}(A,X)_{0}:=\colim_{\cM}\left(\Hom_{\cM}(- \odot X,A_{n})\right),
\]
and where a $(p+1)$-tuple $([B_{0},\phi_{0}],\ldots,[B_{p},\phi_{p}])$ of vertices spans a $p$-simplex if and only if there exists $[B,\phi]\in\colim_{\cM}\left(\Hom_{\cM}(- \odot X^{\odot p+1},A_{n})\right)$ such that 
\[
[B_{i},\phi_{i}]=[B,\phi]\circ [X^{\odot p},\id_{X^{\odot i}}\odot(b^{\cG}_{X,X^{\odot p-i}})^{-1}]
\]
for each $0\le i\le p$.
For each $p\ge 0$, there is a surjection $\pi_{p}\colon W_{n}(A,X)_{p}\to S_{n}(A,X)_{p}$ and we write
\[
\pi\colon |W_{n}(A,X)_{\bullet}|\to|S_{n}(A,X)_{\bullet}|
\]
for the induced map on geometric realisations.
\end{notation}

\subsubsection{Homogeneity and standardness properties for handlebody groups}\label{sss:hom-stdness-properties}

In this section, we verify that the key properties of Definitions~\ref{def:homogeneity} and \ref{def:local-standardness} are satisfied by the categories associated to the handlebody groups we have defined in \S\ref{sss:categorical_framework}, which are necessary to prove our homological stability results in \S\ref{s:homological_stability}.

\paragraph*{Stabilisation maps.}
We introduce here the suitable stabilisation maps of the form $\st_{n}:=(-)_{n}\odot\id_{X}\colon G_{n}\to G_{n+1}$ for homological stability, in the context of handlebody groups. Beforehand, we recall that the monoidal structures $\natural$ and $\#$ of Definition~\ref{def:monoidal-structure-cG} have clear analogous counterparts for the mapping class groups of surfaces, induced by gluing surfaces along half-intervals rather than half-discs. We refer the reader to \cite[\S5.6.1]{RWW} for further details and abuse the notation by using the same symbol for these operations. 
Now, for integers $g,s\ge 0$ and $r\ge 1$, these monoidal structures induce group homomorphisms
\begin{equation}\label{eq:canonical_stabilisation-natural}
\sigma^{s}_{g,r}:=(-)_{g}\natural\id_{V_{1,1}}\colon \cH^{s}_{g,r}\to\cH^{s}_{g+1,r} \,\,\,\,\text{ and }\,\,\,\,
\varsigma^{s}_{g,r}:=(-)_{g}\natural\id_{\Sigma_{1,1}}\colon \MCG^{s}_{g,r}\to\MCG^{s}_{g+1,r},
\end{equation}
and if $r\ge 2$
\begin{equation}\label{eq:canonical_stabilisation-sharp}
\rho^{s}_{g,r}:=(-)_{g}\#\id_{V_{0,2}}\colon \cH^{s}_{g,r}\to\cH^{s}_{g+1,r} \,\,\,\,\text{ and }\,\,\,\,
\varrho^{s}_{g,r}:=(-)_{g}\#\id_{\Sigma_{0,2}}\colon \MCG^{s}_{g,r}\to\MCG^{s}_{g+1,r}.
\end{equation}
We also recall from Lemma~\ref{lem:H-injects-into-Gamma} that we have a canonical group injection $\tti^{s}_{g,r}\colon\cH^{s}_{g,r}\hookrightarrow\MCG^{s}_{g,r}$ relating handlebody groups and mapping class groups of surfaces.
\begin{prop}\label{prop:pullback-square-handlebody-MCG}
For integers $g,s\ge 0$ and $r\ge 1$, the maps \eqref{eq:canonical_stabilisation-natural} and \eqref{eq:canonical_stabilisation-sharp} are injective and define the following pullback commutative squares of injective group homomorphisms
\begin{equation}\label{eq:commutative-square-injections-sigma}
\centering
\begin{split}
\begin{tikzpicture}
[x=1mm,y=1mm]
\node (tl) at (0,12) {$\cH^{s}_{g,r}$};
\node (tr) at (35,12) {$\cH^{s}_{g+1,r}$};
\node (bl) at (0,0) {$\MCG^{s}_{g,r}$};
\node (br) at (35,0) {$\MCG^{s}_{g+1,r},$};
\draw[right hook->] (tl) to node[above,font=\small]{$\sigma^{s}_{g,r}$} (tr);
\draw[right hook->] (bl) to node[above,font=\small]{$\varsigma^{s}_{g,r}$} (br);
\draw[right hook->] (tl) to node[left,font=\small]{$\tti^{s}_{g,r}$} (bl);
\draw[right hook->] (tr) to node[right,font=\small]{$\tti^{s}_{g+1,r}$} (br);
\node at (4,8) {$\lrcorner$};
\end{tikzpicture}
\end{split}
\end{equation}
and if $r\ge 2$
\begin{equation}\label{eq:commutative-square-injections-rho}
\centering
\begin{split}
\begin{tikzpicture}
[x=1mm,y=1mm]
\node (tl) at (0,12) {$\cH^{s}_{g,r}$};
\node (tr) at (35,12) {$\cH^{s}_{g+1,r}$};
\node (bl) at (0,0) {$\MCG^{s}_{g,r}$};
\node (br) at (35,0) {$\MCG^{s}_{g+1,r}$.};
\draw[right hook->] (tl) to node[above,font=\small]{$\rho^{s}_{g,r}$} (tr);
\draw[right hook->] (bl) to node[above,font=\small]{$\varrho^{s}_{g,r}$} (br);
\draw[right hook->] (tl) to node[left,font=\small]{$\tti^{s}_{g,r}$} (bl);
\draw[right hook->] (tr) to node[right,font=\small]{$\tti^{s}_{g+1,r}$} (br);
\node at (4,8) {$\lrcorner$};
\end{tikzpicture}
\end{split}
\end{equation}
\end{prop}
\begin{proof}
First, that the morphisms $\varsigma^{s}_{g,r}$ and $\varrho^{s}_{g,r}$ for mapping class groups of surfaces are injective is a classical property of these maps. Proofs (or some mutatis mutandis adaptations) of these properties may be found in \cite[Prop.~5.18(i)]{RWW} or \cite[Lem.~4.2]{PalmerWu} for $\varsigma^{s}_{g,r}$, and in \cite[Prop.~3.4]{HarrVistrupWahl-DisorderedArcs} for $\varrho^{s}_{g,r}$. Also, it is a straightforward consequence of the definitions of the stabilisation morphisms \eqref{eq:canonical_stabilisation-natural} and \eqref{eq:canonical_stabilisation-sharp} as extensions by the identity, along with those of the canonical injections $\tti^{s}_{g,r}$ and $\tti^{s}_{g+1,r}$, that the diagrams \eqref{eq:commutative-square-injections-sigma} and \eqref{eq:commutative-square-injections-rho} are commutative. In particular, we know that the vertical arrows and the bottom horizontal arrows of both \eqref{eq:commutative-square-injections-sigma} and \eqref{eq:commutative-square-injections-rho} are injective, so the maps $\sigma^{s}_{g,r}$ and $\rho^{s}_{g,r}$ are also injective (since they are the first factor of a composition which is a monomorphism).

It now remains to prove that the commutative diagrams \eqref{eq:commutative-square-injections-sigma} and \eqref{eq:commutative-square-injections-rho} are pullback squares. Let us focus on the case \eqref{eq:commutative-square-injections-sigma}, as the proof for \eqref{eq:commutative-square-injections-rho} repeats verbatim what follows by replacing $\sigma^{s}_{g,r}$ and $\varsigma^{s}_{g,r}$ with $\rho^{s}_{g,r}$ and $\varrho^{s}_{g,r}$ respectively.
Since monomorphisms are stable under pullback (see \cite[Prop.~I.7.1]{Mitchell} for instance), it follows from the general structure of pullbacks in the category of sets (see \cite[\S5.3, Ex.~5.9]{Awodey} for example) that the pullback $\MCG^{s}_{g,r} \times_{\MCG^{s}_{g+1,r}} \cH^{s}_{g+1,r}$ of $\varsigma^{s}_{g,r}$ and $\tti^{s}_{g+1,r}$ is isomorphic to the intersection subgroup $\MCG^{s}_{g,r} \cap \cH^{s}_{g+1,r}$. Note that, for an element $\phi\in \cH^{s}_{g+1,r}$, if the restriction to the boundary $\tti^{s}_{g+1,r}(\phi)$ lifts to $\MCG^{s}_{g,r}$ along $\varsigma^{s}_{g,r}=(-)_{g}\natural\id_{\Sigma_{1,1}}$, then $\phi$ is necessarily of the form $\phi'\natural \id_{V_{1,1}}=\sigma^{s}_{g,r}(\phi')$ where $\phi'\in \cH^{s}_{g,r}$. This defines an injection $\MCG^{s}_{g,r} \cap \cH^{s}_{g+1,r} \hookrightarrow \cH^{s}_{g,r}$, which has a canonical inverse induced by the injections $\tti^{s}_{g,r}$ and $\sigma^{s}_{g,r}$. Hence there is an isomorphism $\MCG^{s}_{g,r} \cap \cH^{s}_{g+1,r} \cong \cH^{s}_{g,r}$, such that the induced diagram connecting the pullback square $\MCG^{s}_{g,r} \times_{\MCG^{s}_{g+1,r}} \cH^{s}_{g+1,r}$ to \eqref{eq:commutative-square-injections-sigma} is commutative by construction. We then deduce from the universal property of a pullback that \eqref{eq:commutative-square-injections-sigma} is a pullback square.
\end{proof}

\paragraph*{Local homogeneity.} We now check that the groupoids associated to the handlebody groups satisfy the local homogeneity property of Definition~\ref{def:homogeneity}.

\begin{prop}\label{prop:homogeneity-handlebody}
For any $s\ge 0$, the following categories are locally homogeneous:
\begin{compactenum}[(a)]
    \item\label{item:first-stabilisation} the right $(\sH^{\ge 1},\natural,I_{1})$-module $(\sH^{\ge 1},\natural)$ at $(V^{s}_{0,r},V_{1,1})$ for any $r\ge 1$;
    \item\label{item:second-stabilisation} the right $(\sH^{\ge 2},\#,I_{2})$-module $(\sH^{\ge 2},\#)$ at $(V^{s}_{0,r},V_{0,2})$ for any $r\ge 2$;
    \item\label{item:infinite-stabilisation} the right $(\sH^{\ge 2},\#,I_{2})$-module $(\overline{\sH}^{\ge 2},\#)$ at $(V^{s}_{\infty,r},V_{0,2})$ for any $r\ge 2$.
\end{compactenum}
\end{prop}
\begin{proof}
We fix $r\ge 1$ and $s\ge 0$. We start by proving the local homogeneity properties \eqref{item:first-stabilisation}--\eqref{item:second-stabilisation}.
Unwinding the definitions, the injectivity condition \eqref{(I)} for the properties \eqref{item:first-stabilisation} and \eqref{item:second-stabilisation} correspond to the injectivities for each $0\le p<n$ of the iterates $\sigma^{s}_{n-1,r} \circ \cdots \circ \sigma^{s}_{n-p+1,r}$ (for $r\ge 1$) and $\rho^{s}_{n-1,r}\circ \cdots \circ \rho^{s}_{n-p+1,r}$ (for $r\ge 2$) of the stabilisation maps \eqref{eq:canonical_stabilisation-natural} and \eqref{eq:canonical_stabilisation-sharp} respectively. These facts then follow from Proposition~\ref{prop:pullback-square-handlebody-MCG}.
For the cancellation condition \eqref{(C)}, we consider an object $Y$ of either $\sH^{\ge 1}$ or $\sH^{\ge 2}$, such that $Y\natural V_{1,1}^{\natural p+1}\cong V^{s}_{g,r}$ or $Y\# V_{0,2}^{\# p+1}\cong V^{s}_{g,r}$ respectively. By the definitions of $\natural$ and $\#$ as boundary connected sums (see Definition~\ref{def:monoidal-structure-cG}), it follows from the classification of surfaces (see \cite{Richards,Brown-Messer}) that the boundary of $Y$ is in both cases a compact connected orientable surface of genus $g-p-1$, which has $r$ marked discs and $s$ marked points. Hence the object $Y$ cannot be a disjoint union of $3$-balls, and is thus a compact handlebody; see Definition~\ref{def:groupoid-handlebody}. Also, it follows from \cite[Th.~2.1]{Bonahon} that compact handlebodies are determined up to diffeomorphisms by their boundaries. We deduce that $Y\cong V^{s}_{g-p-1,r}$ in any case, which finish the proof of \eqref{item:first-stabilisation} and \eqref{item:second-stabilisation}.

We now deal with the case \eqref{item:infinite-stabilisation}. Let $\phi$ and $\phi'$ be elements of $\Aut_{\overline{\sH}^{\ge 2}}(V^{s}_{\infty,r}\# V_{0,2}^{\#(n-p-1)})$ such that $\phi\#\id_{V_{0,2}^{\#(p+1)}}=\phi'\#\id_{V_{0,2}^{\#(p+1)}}$.
We deduce from Definition~\ref{def:groupoids-handlebody-infinity} that there exist an integer $j\ge 0$ and sets of morphisms $\{\phi_{g}\}_{g\ge j}$ and $\{\phi'_{g}\}_{g\ge j}$ with $\phi_{g},\phi'_{g}\in \Aut_{\sH^{\ge 2}}(V^{s}_{g+n-p-1,r})$ making the corresponding diagrams of the form \eqref{eq:morphism-infty-handlebody} commutative, such that $\phi=\colim_{g\ge j}(\phi_{g})$ and $\phi'=\colim_{g\ge j}(\phi'_{g})$. Since the maps $\iota_{i}(V^{s}_{0,r}\# V_{0,2}^{\# n})$ are monomorphisms in the category $\langle \sH^{\ge 1},\sH^{\ge 1}]$ for all $i\ge 0$, it directly follows from the explicit description of the form $\coprod_{g\ge 0} (V^{s}_{g,r}\#V^{\# n}_{0,2})/ \sim$ of the colimit $V^{s}_{\infty,r}\# V_{0,2}^{\# n}$ (see \cite[Prop.~2.13.3]{Borceux1}) that the structural map $\alpha_{i} \colon V^{s}_{i,r}\# V_{0,2}^{\# n}\to V^{s}_{\infty,r}\# V_{0,2}^{\# n}$ is a monomorphism for each $i\ge 0$. Restricting the equality $\phi\#\id_{V_{0,2}^{\#(p+1)}}=\phi'\#\id_{V_{0,2}^{\#(p+1)}}$ along each $\alpha_{g}$ for $g\ge j$ then implies that we have an equality $\phi_{g}\#\id_{V_{0,2}^{\#(p+1)}}=\phi'_{g}\#\id_{V_{0,2}^{\#(p+1)}}$ for each $g\ge j$. Since the map $(-)_{g+n-p-1}\#\id_{V_{0,2}^{\#(p+1)}}$ is injective by Proposition~\ref{prop:homogeneity-handlebody}\eqref{item:second-stabilisation}, we deduce that $\phi_{g}=\phi'_{g}$ for all $g\ge j$, and so $\phi=\phi'$ by the universal property of a colimit: this proves the injectivity condition \eqref{(I)} for the case \eqref{item:infinite-stabilisation}.
Furthermore, for an object $Y$ of $\overline{\sH}^{\ge 2}$ such that
$Y\# V_{0,2}^{\#(p+1)}\cong V^{s}_{\infty,r}\# V_{0,2}^{\# n}$, it is immediate from the definition of $-\# -$ (see Lemma~\ref{lem:right-module-infinite}) that $Y$ is necessarily an infinite object of $\overline{\sH}^{\ge 2}$. By applying the isomorphism \eqref{eq:canonical-iso-dash-natural} of Lemma~\ref{lem:stabilising-handlebody-of-inf-rank}, we deduce that $Y\cong Y\# V_{0,2}^{\#(p-1)}\cong V^{s}_{\infty,r}\# V_{0,2}^{\# n}\cong V^{s}_{\infty,r}$, and so the cancellation condition~\eqref{(C)} holds for \eqref{item:infinite-stabilisation}.
\end{proof}

\paragraph*{Standardness.} We finally verify that the local standardness property of Definition~\ref{def:local-standardness} hold for the groupoids associated to the handlebody groups.

\begin{prop}\label{prop:local-standardness}
For any $s\ge 0$, the following categories are locally standard:
\begin{compactenum}[(a)]
    \item\label{item:first-stabilisation-standard} the right $(\sH^{\ge 1},\natural,I_{1})$-module $(\sH^{\ge 1},\natural)$ at $(V^{s}_{g,r},V_{1,1})$ for any $g\ge 0$ and $r\ge 1$;
    \item\label{item:second-stabilisation-standard} the right $(\sH^{\ge 2},\#,I_{2})$-module $(\sH^{\ge 2},\#)$ at $(V^{s}_{g,r},V_{0,2})$ for any $g\ge 0$ and $r\ge 2$;
    \item\label{item:infinite-stabilisation-standard} the right $(\sH^{\ge 2},\#,I_{2})$-module $(\overline{\sH}^{\ge 2},\#)$ at $(V^{s}_{\infty,r},V_{0,2})$ for any $r\ge 2$.
\end{compactenum}
\end{prop}
\begin{proof}
We first treat Condition~\eqref{(LS1)}. We take the conventions of Notation~\ref{nota:Z-V-V'}. It follows from an elementary computation that $H_{1}(V^{s}_{g,r}\natural_{i} Z_{1} \natural_{i} Z_{2}, \cD_{r}\sqcup \cP_{s} ;\bZ)$ is isomorphic to the free abelian group $\bZ^{2g+r-1+s}$, where each $Z_{i}$ supports the generator of exactly one copy of $\bZ$.
For any $\phi\in \Aut_{\sH^{\ge i}}(V^{s}_{g,r}\natural_{i} Z_{1})$, the morphism $\phi\natural_{i} \id_{Z_{2}}$ acts trivially on the homology class supported on $Z_{2}$, whereas the braiding $(b^{\sH^{\ge i}}_{Z_{1},Z_{2}})^{-1}$ acts non-trivially on it. Therefore the equivalence classes $[V^{s}_{g+1,r},\id_{V^{s}_{g,r}}\natural (b^{\sH^{\ge i}}_{Z_{1},Z_{2}})^{-1}]$ and $[V^{s}_{g+1,r},\id_{V^{s}_{g+2,r}}]$ are distinct in $\colim_{\sH^{\ge i}}\left(\Hom_{\sH^{\ge i}}(- \natural_{i} X,A_{2})\right)$, which proves Condition~\eqref{(LS1)} for the cases \eqref{item:first-stabilisation-standard} and \eqref{item:second-stabilisation-standard}.
Furthermore, if we assume that the equivalence classes $[V^{s}_{\infty,r}\# V_{0,2},\id_{V^{s}_{\infty,r}}\natural (b^{\sH^{\ge 2}}_{V_{0,2},V_{0,2}})^{-1}]$ and $[V^{s}_{\infty,r}\# V_{0,2},\id_{V^{s}_{\infty,r}}\# \id_{V_{0,2}} \# \id_{V_{0,2}}]$ are equal in $\colim_{\overline{\sH}^{\ge 2}}(\Hom_{\overline{\sH}^{\ge 2}}(- \# V_{0,2}, V^{s}_{\infty,r}\# V^{\# 2}_{0,2}))$, then there exists $\psi\in \Aut_{\overline{\sH}^{\ge 2}}(V^{s}_{\infty,r}\# V_{0,2})$ such that $\psi \# \id_{V_{0,2}}= \id_{V^{s}_{\infty,r}}\# (b^{\sH^{\ge 2}}_{V_{0,2},V_{0,2}})^{-1}$. Following Definition~\ref{def:groupoids-handlebody-infinity}, there exist an integer $j\ge 0$ and a set of morphisms $\{\psi_{g}\}_{g\ge j}$ with $\psi_{g}\in \Aut_{\sH^{\ge 2}}(V^{s}_{g,r}\# V_{0,2})$ making the corresponding diagrams of the form \eqref{eq:morphism-infty-handlebody} commutative, and such that $\psi=\colim_{g\ge j}(\psi_{g})$.
The maps $\iota_{i}(V^{s}_{0,r}\# V^{\# 2}_{0,2})$ are monomorphisms in the category $\langle \sH^{\ge 1},\sH^{\ge 1}]$ for all $i\ge 0$, so it directly follows from the explicit description of the form $\coprod_{g\ge 0} (V^{s}_{g,r}\#V^{\# 2}_{0,2})/ \sim$ of the colimit $V^{s}_{\infty,r}\# V_{0,2}^{\# 2}$ (see \cite[Prop.~2.13.3]{Borceux1}) that the structural map $\alpha_{i} \colon V^{s}_{i,r}\# V_{0,2}^{\# 2}\to V^{s}_{\infty,r}\# V_{0,2}^{\# 2}$ is a monomorphism for each $i\ge 0$. Restricting the equality $\psi \# \id_{V_{0,2}}= \id_{V^{s}_{\infty,r}}\# (b^{\sH^{\ge 2}}_{V_{0,2},V_{0,2}})^{-1}$ along each $\alpha_{g}$ for $g\ge j$, this implies that the morphisms $\psi_{g} \# \id_{V_{0,2}}$ and $\id_{V^{s}_{g,r}}\# (b^{\sH^{\ge 2}}_{V_{0,2},V_{0,2}})^{-1}$ are equal for each $g\ge j$. This contradicts the case \eqref{item:second-stabilisation-standard}, and so Condition~\eqref{(LS1)} holds for \eqref{item:infinite-stabilisation-standard}.

Let us now deal with the Condition~\eqref{(LS2)}. We start with the case \eqref{item:first-stabilisation-standard}. First, we note from the definitions of the morphisms and repeating verbatim the first part of the proof of Lemma~\ref{lem:pullback-standardness-LS2} that the following diagram is commutative for each $g\ge 0$ and $n\ge 2$
\begin{equation}\label{eq:pullback-standardness-LS2-HMCG}
\centering
\begin{split}
\begin{tikzpicture}
[x=1mm,y=1mm]
\node (tl) at (0,15) {$\cH^{s}_{g+n-2,r}$};
\node (tr) at (50,15) {$\cH^{s}_{g+n-1,r}$};
\node (bl) at (0,0) {$\cH^{s}_{g+n-1,r}$};
\node (br) at (50,0) {$\cH^{s}_{g+n,r}$.};
\draw[right hook->] (tl) to node[above,font=\small]{$\sigma^{s}_{g+n-1,r}$} (tr);
\draw[right hook->] (bl) to node[above,font=\small]{$\fC_{n}(V_{1,1},V_{1,1})\circ \sigma^{s}_{g+n,r}$} (br);
\draw[right hook->] (tl) to node[left,font=\small]{$\sigma^{s}_{g+n-1,r}$} (bl);
\draw[right hook->] (tr) to node[right,font=\small]{$\sigma^{s}_{g+n,r}$} (br);
\end{tikzpicture}
\end{split}
\end{equation}
It is a classical fact that the analogous diagram to \eqref{eq:pullback-standardness-LS2-HMCG} for the mapping class groups of surfaces (i.e.~replacing $\cH$ with $\MCG$, $\sigma$ with $\varsigma$, and $V$ with $\Sigma$), that we denote by $(\eqref{eq:pullback-standardness-LS2-HMCG}_{\MCG}$, is a pullback square; see \cite[Prop.~5.6.1]{RWW} or \cite[Prop.~4.11(2)]{PalmerWu} for instance. We consider the cubical commutative diagram $\eqref{eq:pullback-standardness-LS2-HMCG}\to \eqref{eq:pullback-standardness-LS2-HMCG}_{\MCG}$ induced by the canonical injections of the form $\tti^{s}_{g,r}$. The four commutative squares between \eqref{eq:pullback-standardness-LS2-HMCG} and $\eqref{eq:pullback-standardness-LS2-HMCG}_{\MCG}$ are of the form of diagram \eqref{eq:commutative-square-injections-sigma}, and so they are pullback squares by Proposition~\ref{prop:pullback-square-handlebody-MCG}. It then follows from  the Pullback Lemma (see \cite[Lem.~5.10.2.]{Awodey} for instance) that the commutative diagram~\eqref{eq:pullback-standardness-LS2-HMCG} is a pullback square, and we deduce from Lemma~\ref{lem:pullback-standardness-LS2} that Condition~\eqref{(LS2)} is satisfied for \eqref{item:first-stabilisation-standard}.
The proof for the case \eqref{item:second-stabilisation-standard} simply repeats mutatis mutandis the above arguments and reasoning, replacing $\sigma$ with $\rho$, $\varsigma$ with $\varrho$, $\natural$ with $\#$, $V_{1,1}$ with $V_{0,2}$, and \eqref{eq:commutative-square-injections-sigma} with \eqref{eq:commutative-square-injections-rho}. (In this case, that the analogue of $\eqref{eq:pullback-standardness-LS2-HMCG}_{\MCG}$ is a pullback square follows from a straightforward adaptation of \cite[Prop.~4.3]{PalmerWu} using $\#$ instead of $\natural$ for instance.)
Finally, for the case \eqref{item:infinite-stabilisation-standard}, let us first describe what is an element $\varphi$ of $\colim_{\sH^{\ge 2}}(\Hom_{\sH^{\ge 2}}(- \# V_{0,2},V^{s}_{\infty,r}\# V_{0,2}^{\# n-1}))$ for $n\ge 2$.
We deduce from Definition~\ref{def:groupoids-handlebody-infinity} that there exists an integer $j\ge 0$ and a set of morphisms $\{\varphi_{g}=[V^{s}_{g,r}\# V_{0,2}^{\# n-2},\phi_{g}]\}_{g\ge j}$ with $\phi_{g}\in \Aut_{\sH^{\ge 2}}(V^{s}_{g+n-1,r})$ making the following diagram commutative
such that $\psi=\colim_{g\ge j}(\psi_{g})$
\[
\begin{tikzcd}[column sep=.7em]
    V_{0,2}\arrow{d}{\varphi_{j}}\arrow[rrrrrr,"\id_{V_{0,2}}"]&&&&&&
    V_{0,2}\arrow{d}{\varphi_{j+1}}\arrow[rrrrrr,"\id_{V_{0,2}}"] &&&&&&
    \cdots \arrow[rrrrrr,"\id_{V_{0,2}}"] &&&&&&
    V_{0,2}\arrow{d}{\varphi_{g}} \arrow[rrrrrr,"\id_{V_{0,2}}"] &&&&&&
    \cdots
    \\
    V^{s}_{j+n-1,r}\arrow[rrrrrr,"\iota_{j+1}( V^{\# (n-1)}_{0,2})"] &&&&&&
    V^{s}_{j+n,r}\arrow[rrrrrr,"\iota_{j+2}( V^{\# (n-1)}_{0,2})"] &&&&&& 
    \cdots \arrow[rrrrrr,"\iota_{g}( V^{\# (n-1)}_{0,2})"] &&&&&&
    V^{s}_{g+n-1,r}\arrow[rrrrrr,"\iota_{g+1}( V^{\# (n-1)}_{0,2})"]  &&&&&&
    \cdots.
\end{tikzcd}
\]
Furthermore, similarly to the above proof of Condition~\eqref{(LS1)} for \eqref{item:infinite-stabilisation-standard}, we again note that the structural map $V^{s}_{i,r}\# V_{0,2}^{\# n}\to V^{s}_{\infty,r}\# V_{0,2}^{\# n}$ is a monomorphism for each $i\ge 0$, essentially because since map $\iota_{g}(V^{\# n}_{0,2})$ is a monomorphism and thanks to the description \cite[Prop.~2.13.3]{Borceux1} of $V^{s}_{\infty,r}\# V_{0,2}^{\# n}$.

Now, if $\theta^{n}_{V^{s}_{\infty,r},V_{0,2}}(\varphi)=\theta^{n}_{V^{s}_{\infty,r},V_{0,2}}(\varphi')$ for $\varphi,\varphi'\in\colim_{\sH^{\ge 2}}(\Hom_{\sH^{\ge 2}}(- \# V_{0,2},V^{s}_{\infty,r}\# V_{0,2}^{\# (n-1)}))$, it follows from the above description and the restriction along the structural maps that there exists $J\ge j\ge 0$ such that $\theta^{n}_{V^{s}_{g,r},V_{0,2}}(\varphi_{g})=\theta^{n}_{V^{s}_{g,r},V_{0,2}}(\varphi'_{g})$ for each $g\ge J$. 
Since the map $\theta^{n}_{V^{s}_{g,r},V_{0,2}}$ is injective by Proposition~\ref{prop:local-standardness}\eqref{item:second-stabilisation-standard}, we deduce that $\varphi_{g}=\varphi'_{g}$ for all $g\ge J$, and so $\varphi=\varphi'$ by the universal property of a colimit. Hence Condition~\eqref{(LS2)} holds for \eqref{item:infinite-stabilisation-standard}.
\end{proof}

\section{Homological stability with twisted coefficients}\label{s:homological_stability}

In this section we prove Theorem~\ref{thm:mainA} and Theorem~\ref{thm:mainB}.
Given the categorical properties we have verified in \S\ref{sss:hom-stdness-properties}, we recall from \S\ref{sss:HS-meta-theorem} that proving Theorems~\ref{thm:mainA} and \ref{thm:mainB} in both cases boils down to showing that an associated semi-simplicial set is sufficiently highly connected. In both cases, we are able to obtain the required connectivity from results of \cite{HatcherWahl}.
We start with Theorem~\ref{thm:mainA}, concerning the standard genus stabilisation for surfaces, utilising the monoidal structure $\natural$; see \S\ref{ss:handle-stabilisation-natural}. We next prove twisted homological stability for the second, less standard, genus stabilisation mentioned in the introduction, which comes from the monoidal structure $\#$; see \S\ref{ss:handle-stabilisation-sharp}. Combining these two homological stability results, Theorem~\ref{thm:mainB} follows as a corollary; see \S\ref{ss:HS-marked-discs}.

\subsection{Handle stabilisation with respect to $\natural$}\label{ss:handle-stabilisation-natural}

We consider the standard stabilisation map given by attaching the genus-one handlebody with a marked disc $V_{1,1}$ via the operation $\natural$. For all \S\ref{ss:handle-stabilisation-natural}, we fix integers $r\ge 1$ and $s\ge 0$ and write for brevity $\sH^{\ge 1}_{r,s}$ (resp.~$\sH^{\ge 1}_{1}$) for $\sH^{\ge 1}_{V^{s}_{0,r},V_{1,1}}$ (resp.~$\sH^{\ge 1}_{V_{1,1}}$). In terms of our categorical framework set in \S\ref{sss:categorical_framework}, we work in the category $[ \sH^{\ge 1}_{r,s}, \sH^{\ge 1}_{1} \rangle$ defined by the Quillen bracket construction (see \S\ref{sss:Quillen-bracket-construction}) applied to the right-module $(\sH^{\ge 1}_{r,s},\natural)$ over the braided monoidal category $(\sH^{\ge 1}_{1},\natural,I_{1})$ (see \S\ref{sss:groupoids-handlebodies}).
We recall that the objects of $[ \sH^{\ge 1}_{r,s}, \sH^{\ge 1}_{1} \rangle$ are the handlebodies $V^{s}_{0,r}\natural V_{1,1}^{\natural g} \cong V^{s}_{g,r}$ (see \eqref{eq:natural-iso-decomposition}) for all $g\ge 0$. In particular, we have isomorphisms (where the second one holds by Corollary~\ref{coro:maximal-subgroupoid-of-UGH}):
\[
\cH^{s}_{g,r}\cong\Aut_{\sH^{\ge 1}_{1}}(V^{s}_{g,r})\cong\Aut_{[ \sH^{\ge 1}_{r,s}, \sH^{\ge 1}_{1} \rangle}(V^{s}_{g,r}).
\]
The stabilisation map is then defined by the group injection $\sigma^{s}_{g,r}:=(-)_{g}\natural\id_{V_{1,1}}\colon \cH^{s}_{g,r}\hookrightarrow\cH^{s}_{g+1,r}$; see \eqref{eq:canonical_stabilisation-natural}.
Furthermore, we write $c^{s}_{g,r}$ for the canonical morphism $[\id_{V^{s}_{g+1,r}},V_{1,1}]\colon V^{s}_{g,r} \to V^{s}_{g+1,r}$ of $[ \sH^{\ge 1}_{r,s}, \sH^{\ge 1}_{1} \rangle$. The following twisted homological stability result, whose proof is done in \S\ref{sss:proof-HS-1}, is a more precise version of Theorem~\ref{thm:mainA}:
\begin{thm}\label{thm:HS_stabilistation-1}
We fix integers $r\ge 1$ and $s\ge 0$. Let $F\colon [ \sH^{\ge 1}_{r,s},\sH^{\ge 1}_{1} \rangle\to \Ab$ be a coefficient system of degree $d$ at $0$ (see Definition~\ref{def:finite-deg-coeff-system}). Then
\[
(\sigma^{s}_{g,r};F(c^{s}_{g,r}))_{i}\colon H_{i}(\cH^{s}_{g,r};F(V^{s}_{g,r}))\to H_{i}(\cH^{s}_{g+1,r};F(V^{s}_{g+1,r}))
\]
is an isomorphism for $i\le \frac{g-1}{2}-d-1$. If $F$ is split, it is an isomorphism for $i\le\frac{g-1-d}{2}-1$.
\end{thm}
\begin{eg}\label{eg:HS-examples-recover}
Theorem~\ref{thm:HS_stabilistation-1} is a generalisation of the homological stability result \cite[Th.~5.31]{RWW} to any $s\ge 0$. Namely, assigning $s=0$, Theorem~\ref{thm:HS_stabilistation-1} then recovers the following previous works:
\begin{itemizeb}
    \item \cite[Th.~1.8(i)]{HatcherWahl} (for handlebodies) by assigning $F=\bZ$ (see Example~\ref{eg:Z_H_poly_functor});
    \item \cite[Th.~1.2]{IshidaSato} by assigning $r=1$ and $F=H_{\Sigma}$ (see Proposition~\ref{prop:H-Sigma-V_functors});
    \item \cite[Th.~5.31]{RWW} by assigning $r=1$ and for any finite degree coefficient system $F$.
\end{itemizeb}
Moreover, for fixed $r\ge 1$ and $s\ge 0$, the functors $T^{d}H_{\Sigma}$ and $T^{d}H_{\cV}$, which encode the homological representations $\{H^{\otimes d}_{1}(\partial V\setminus  (\mathring{\cD}_{r}\sqcup\cP_{s});\bZ)\}_{g\ge0}$ and $\{H^{\otimes d}_{1}(V, \cD_{r}\sqcup\cP_{s};\bZ)\}_{g\ge0}$ of the handlebody groups $\{\cH^{s}_{g,r}\}_{g\ge0}$ respectively, are split coefficient systems at $V^{s}_{0,r}$ of finite degree $d \ge 0$; see Proposition~\ref{prop:H-Sigma-V_functors}. Twisted homological stability with these coefficients thus follows from Theorem~\ref{thm:HS_stabilistation-1}.
\end{eg}

\subsubsection{Properties of the simplicial complex of destabilisations}\label{sss:connectivity-ss}

In order to prove Theorem~\ref{thm:HS_stabilistation-1}, we first study the simplicial complex of destabilisations introduced in Notation~\ref{nota:simplicial-complex-destabilisation} involved in this context. Namely, we prove that this simplicial complex is isomorphic to a simplicial complex originally defined by Hatcher and Wahl \cite{HatcherWahl} (see Proposition~\ref{prop:stab1-S-g-iso-Y}), which satisfies a high-connectivity property (see Theorem~\ref{thm:HatcherWahl-connectivity1}).

\paragraph*{An alternative simplicial complex.}
First, here follows the simplicial complex following that introduced in \cite[\S8.2]{HatcherWahl}:
\begin{defn}\label{def:sc-stabilisation}
We consider an object $(V,r,s,i,j)$ of $\sH$ where $V$ is a compact handlebody, the set $\Image(i)\sqcup\Image(j)=\cD_{r}\sqcup\cP_{s}$ and we fix points $x_{0},x_{1}\in\partial\cD_{r}$ (where we allow $x_{0}=x_{1}$). Recalling that $I$ denotes the unit interval $[0,1]$, let $I+\bD^{2}$ be the space defined from the union $I\sqcup \bD^{2}$ by identifying $1/2\in I$ with some base-point $p_{0}\in\partial\bD^{2}$. We consider maps $f\colon I+\bD^{2}\to V$ such that:
\begin{itemizeb}
    \item the restriction of $f$ to $\mathring{I}+\bD^{2}$ is an embedding into $V\setminus \cD_{r}\sqcup\cP_{s}$, while the restrictions of $f$ to $I$ and $\bD^{2}$ are smooth; 
    \item $f(I)$ intersects $f(\partial\bD^{2})$ transversely and $f(\partial\bD^{2})$ does not separate $\partial V$;
    \item $f(I+\partial\bD^{2})\subset\partial V$ intersects $\cD_{r}\sqcup\cP_{s}$ only at $f(0)=x_{0}$ and $f(1)=x_{1}$.
\end{itemizeb}
We define $Y^{A}(V,\cD_{r}\sqcup\cP_{s},x_{0},x_{1})$ to be the simplicial complex whose vertices are isotopy classes of such maps $f$, where $(f_{0},\ldots,f_{p})$ forms a $p$-simplex if:
\begin{itemizeb}
    \item the maps have pairwise disjoint images (outside of the endpoints);
    \item the union $f_{0}(\partial\bD^{2})\cup\cdots\cup f_{p}(\partial\bD^{2})$ is a \emph{coconnected} system of discs in $V$, i.e.~it does not separate $\partial V$.
    \end{itemizeb}
\end{defn}
Hatcher and Wahl do not consider marked points in \cite[\S8.2]{HatcherWahl}, i.e.~they consider the simplicial complex of Definition~\ref{def:sc-stabilisation} with the assignment $\cP_{s}=\emptyset$. However, their restricted definition recovers our more general one as follows:
\begin{lem}\label{lem:iso-simplicial-complex-marked-points}
For any object $(V,r,s,i,j)$ of $\sH$, there is an isomorphism of simplicial complexes
\[
\Psi^{s}_{V,r} \colon Y^{A}(V,\cD_{r}\sqcup\cP_{s},x_{0},x_{1}) \overset{\cong}{\longrightarrow} Y^{A}(V,\cD_{r+s},x_{0},x_{1}).
\]
\end{lem}
\begin{proof}
For any $f\in Y^{A}(V,\cD_{r}\sqcup\cP_{s},x_{0},x_{1})$, we define $\Psi^{s}_{V,r}(f)$ as follows. By the conditions of Definition~\ref{def:sc-stabilisation}, the set $\Image(f)$ does not intersect $\cP_{s}$ and, for each marked point $p_{i}\in \cP_{s}$, we may choose an open neighbourhood $\cN(p_{i})\subset V$ of $p_{i}$ such that the map $f$ is a smooth embedding in $\cN(p_{i})$. Also, since $V$ is a smooth manifold with boundary, we may pick a closed disc $\bD^{2}_{p_{i}}\subset \partial V \cap \cN(p_{i})$ centered on $p_{i}$ for each $1\le i\le s$.
Therefore, either $\Image(f)\cap \bD^{2}_{p_{i}} = \emptyset$, or we may apply an isotopy of $\cN(p_{i})$ to $f$ defining isotopic map $f'$ so that $\Image(f')$ does not intersect $\bD^{2}_{p_{i}}$. Denoting by $[f]$ the isotopy class of $f$, we then assign $\Psi^{s}_{V,r}(f)$ to be the isotopy class of the map obtained by this procedure. It is a straightforward routine to check from the conditions of Definition~\ref{def:sc-stabilisation} that $\Psi^{s}_{V,r}(f)$ belongs to $Y^{A}(V,\cD_{r+s},x_{0},x_{1})$, and $\Psi^{s}_{V,r}([f])=\Psi^{s}_{V,r}([g])$ for $g\colon I+\bD^{2}\to V$ such that $[f]=[g]$, and so $\Psi^{s}_{V,r}$ is a well-defined morphism of simplicial complexes. Now, the inverse map of $\Psi^{s}_{V,r}$ is the obvious morphism of simplicial complexes induced by fixing the marked point $p_{i}$ to be center of the marked disc $\cD^{i}_{V}$ and forgetting all the other points $\cD^{i}_{V}\setminus p_{i}$ for each $1\le i\le s$, and so the map $\Psi^{s}_{V,r}$ is clearly bijective.
\end{proof}
Then, the following fundamental connectivity property is a direct consequence of \cite[Thm.~8.5]{HatcherWahl}, using the isomorphism of simplicial complexes of Lemma~\ref{lem:iso-simplicial-complex-marked-points}:
\begin{thm}\label{thm:HatcherWahl-connectivity1}
Let $(V,r,s,i,j)$ be an object of $\sH$, $\cD_{r}\sqcup\cP_{s}$ and $x_{0}=x_{1}\in\partial\cD_{r}$ as in Definition~\ref{def:sc-stabilisation}. For $h\ge 0$ the number of $\bS^{1}\times \bD^{2}$-boundary connected summands in $V$, the simplicial complex $Y^{A}(V,\cD_{r}\sqcup\cP_{s},x_{0},x_{0})$ is $\frac{h-3}{2}$-connected.
\end{thm}

\paragraph*{A simplicial complex isomorphism.}
Another key property of the simplicial complex of Definition~\ref{def:sc-stabilisation} is that it is isomorphic to the destabilisation complex of Notation~\ref{nota:simplicial-complex-destabilisation} in this context. From now on, we fix an integer $g\ge 1$ for the rest of \S\ref{sss:connectivity-ss}, and we set $x_{0}=x_{1}$ as the image of the point $(1,0)\in\bD^{2}$ under the parametrisation $i_{1}$ of the first marked disc $\cD_{V^{s}_{g,r}}^{1}$ in $\partial V^{s}_{g,r}$. We define a map of simplicial complexes
\begin{equation}\label{eq:iso-simplicial-complexes}
\Phi^{s}_{g,r}\colon S_{g}(V^{s}_{0,r},V_{1,1})_{\bullet}\to Y^{A}(V_{g},\cD_{r}\sqcup\cP_{s},x_{0},x_{0}),
\end{equation}
as follows. First, let us once and for all choose a map $\Phi_{0}\colon I+\bD^{2}\to V_{1,1}$ satisfying the conditions of Definition~\ref{def:sc-stabilisation}. Namely, we may assume that the parametrisation map $i_{1}\colon \bD^{2}\to\partial V_{1,1}$ defining the marked disc $\cD^{1}_{V^{s}_{g,r}}$ to be such that if $x_{0}=i_{1}(1,0)=(t,y)\in \bS^{1}\times\partial\bD^{2}$, then $\bS^{1}\times\{y\}$ does not intersect $\Image(i_{1})$ elsewhere. Then, we define
\begin{align*}
    \Phi_{0}(x)=\begin{cases}
        (x,y)&\text{ if }x\in I\\
        (1/2,x)&\text{ if }x\in\bD^{2}.
    \end{cases}
\end{align*}
We recall from Notation~\ref{nota:simplicial-complex-destabilisation} that a vertex of $S_{g}(V^{s}_{0,r},V_{1,1})_{\bullet}$ is a morphism $[B,\phi]$ in the colimit over $\sH^{\ge 1}_{r,s}$ of $\Hom_{\sH^{\ge 1}_{r,s}}(- \odot V_{1,1},V^{s}_{g,r})$, i.e.~an equivalence class of pairs $(B,\phi)$ where $\phi\colon B\natural V_{1,1}\to V^{s}_{g,r}$ is a morphism in $\sH^{\ge 1}_{r,s}$. We denote by $\Phi'_{0}$ the isotopy class of the map $I+\bD^{2}\to B\natural V_{1,1}$ defined by postcomposing the above map $\Phi_{0}\colon I+\bD^{2}\to V_{1,1}$ with the canonical embedding $V_{1,1}\hookrightarrow B\natural V_{1,1}$ induced by the boundary connected sum $\natural$. 
\begin{lem}\label{lem:Phi-well-defined}
The assignment $[B,\phi] \mapsto \phi\circ \Phi'_{0}$ for the map \eqref{eq:iso-simplicial-complexes} defines a morphism of simplicial complexes.
\end{lem}
\begin{proof}
First, let us verify that the assignment for $\Phi^{s}_{g,r}([B,\phi])$ does not depend on the choice of representative $(B,\phi)$. We note that $(B',\phi')\sim (B,\phi)$ in the colimit over $\sH^{\ge 1}_{r,s}$ of $\Hom_{\sH^{\ge 1}_{r,s}}(- \odot V_{1,1},V^{s}_{g,r})$ if and only if there is an isomorphism $\psi\colon B\to B'$ in $\sH^{\ge 1}_{r,s}$ such that $\phi=\phi'\circ(\psi\natural\id_{V_{1,1}})$. Since $\phi$ and $\phi'$ may differ only on the left summand, we deduce that $\phi'_{\mid V_{1,1}} = \phi_{\mid  V_{1,1}}$ and so $\Phi^{s}_{g,r}([B,\phi])=\Phi^{s}_{g,r}([B',\phi'])$ by definition of $\Phi'_{0}$.

Next, it remains to show that $\Phi^{s}_{g,r}$ is a well-defined map of simplicial complexes, i.e.~that it sends $p$-simplices to $p$-simplices. We recall from Notation~\ref{nota:simplicial-complex-destabilisation} that a $(p+1)$-tuple $([B_{0},\phi_{0}],\ldots,[B_{p},\phi_{p}])$ of vertices of $S_{g}(V^{s}_{0,r},V_{1,1})_{\bullet}$ defines a $p$-simplex if there is $[B,\phi]\in\colim_{\sH^{\ge 1}_{r,s}}\left(\Hom_{\sH^{\ge 1}_{r,s}}(- \odot V_{1,1}^{\natural p+1},V^{s}_{g,r})\right)$ such that for each $0\le i\le p$
\[
[B_{i},\phi_{i}]=[B,\phi]\circ [V^{\natural p}_{1,1},\id_{V_{i,1}}\natural (b_{V_{1,1},V_{p-i,1}})^{-1}],
\]
where $b_{V_{1,1},V_{p-i,1}}$ is the braiding of $(\sH^{\ge 1}_{1},\natural,I_{1})$ and $[V^{\natural p}_{1,1},\id_{V_{i,1}}\natural (b_{V_{1,1},V_{p-i,1}})^{-1}]$ is the canonical map $V_{1,1}\to V_{i,1} \natural V_{1,1}\natural V_{p-i,1}$ induced by equivalence classes of isomorphisms of the target fixing the source summand.
Therefore, the image of $([B_{0},\phi_{0}],\ldots,[B_{p},\phi_{p}])$ by $\Phi^{s}_{g,r}$ induces maps $f_{0},\ldots,f_{p}\colon I+\bD^{2}\to B\natural V_{1,1}^{\natural p+1}$ whose isotopy classes belong to $Y^{A}(V_{g},\cD_{r}\sqcup\cP_{s},x_{0},x_{0})$, with the relation $\Phi^{s}_{g,r}([B_{i},\phi_{i}])=\phi\circ f_{i}$ and such that the image of $f_{i}$ is in the $i$-th $V_{1,1}$-summand (except the endpoints).
In particular, the maps $f_{0},\ldots,f_{p}$ have pairwise disjoint images (outside of the endpoints), and their restrictions to $\partial\bD^{2}$ define a coconnected system of discs $f_{0}(\partial\bD^{2})\cup\cdots\cup f_{p}(\partial\bD^{2})$ in $B\natural V_{1,1}^{\natural p+1}$. Since $\phi$ is the isotopy class of a diffeomorphism, it preserves nonseparating simple closed curves, and so the union $\phi\circ f_{0}(\partial\bD^{2})\cup\cdots\cup \phi\circ f_{p}(\partial\bD^{2})$ is a coconnected system of $p+1$ discs in $V^{s}_{g,r}$. Hence, by Definition~\ref{def:sc-stabilisation}, the system $(\Phi^{s}_{g,r}([B_{0},\phi_{0}]),\ldots,\Phi^{s}_{g,r}([B_{p},\phi_{p}]))$ is a $p$-simplex in $Y^{A}(V_{g},\cD_{r}\sqcup\cP_{s},x_{0},x_{0})$.
\end{proof}

\begin{prop}\label{prop:stab1-S-g-iso-Y}
The morphism of simplicial complexes $\Phi^{s}_{g,r}$ is an isomorphism.
\end{prop}
\begin{proof}
We start by showing that the map $\Phi^{s}_{g,r}$ is surjective. Let $f\colon I+\bD^{2}\to V^{s}_{g,r}$ be a vertex of $Y^{A}(V_{g},\cD_{r}\sqcup\cP_{s},x_{0},x_{0})$. First, we note from Definition~\ref{def:sc-stabilisation} that any regular neighbourhood of $\Image(f)$ is diffeomorphic to $V_{1,1}$. We can pick such a neighbourhood $\cN(f)$, with a disc $\cD_{\cN(f)}$ in its boundary which intersects the first marked disc $\cD^{1}_{V^{s}_{g,r}}$ of $V^{s}_{g,r}$ exactly in its right-half. We set $B=V^{s}_{g,r}\setminus \cN(f)$, with a first marked disc given by the left-half of $\cD^{1}_{V^{s}_{g,r}}$, while its right-half given by the left hand side of $\cD_{\cN(f)}$. We also equip $B$ with $r-1$ more marked discs in the boundary given by the remaining marked discs in $V^{s}_{g,r}$, and with the marked points $\cP_{s}$. This construction thus induces a diffeomorphism $\phi\colon B\natural \cN(f)\to V^{s}_{g,r}$, preserving marked discs $\cD_{r}$ and points $\cP_{s}$. Hence this defines a vertex $[B,\phi]$ in $S_{g}(V^{s}_{0,r},V_{1,1})_{\bullet}$, such that $\phi([B,\phi])=f$.

Now, let us show that $\Phi^{s}_{g,r}$ is injective. For $[B,\phi]$ and $[B',\phi']$ in the colimit over $\sH^{\ge 1}_{r,s}$ of $\Hom_{\sH^{\ge 1}_{r,s}}(- \odot V_{1,1},V^{s}_{g,r})$, suppose that $\Phi^{s}_{g,r}([B,\phi])=\phi\circ \Phi'_{0}=\phi'\circ \Phi'_{0}=\Phi^{s}_{g,r}([B',\phi'])$. Since $V_{1,1}$ is a regular neighbourhood of the image of $\Phi'_{0}$, we deduce from the equality $\phi\circ \Phi'_{0}=\phi'\circ \Phi'_{0}$ that there exist mapping classes $\varphi\in \Aut(\sH^{\ge 1}_{r,s})(B, B')$ and $\psi\in \cH_{1,1}$ such that $\phi^{-1}\circ\phi'=\varphi\natural\psi$.
More precisely, the map $\psi$ is the isotopy class of a diffeomorphism $\tilde{\psi}$ of $V_{1,1}=\bS^{1}\times \bD^{2}$ pointwise fixing the marked disc in the boundary where  $x_{0}$ lies. Moreover, since $(\phi^{-1}\circ\phi')\circ \Phi'_{0}= \Phi'_{0}$, the representative $\tilde{\psi}$ setwise fixes the embedded disc $\{1/2\}\times\bD^{2}$, as well as a simple closed curve $\gamma\subset \partial V_{1,1}$ along the longitude (i.e.~bounding the genus), based at $x_{0}$ and intersecting $\{1/2\}\times\bD^{2}$ transversely. We cut $V_{1,1}$ along the embedded disc $\{1/2\}\times\bD^{2}$ and since setwise fixing a disc is equivalent to fixing a point up to isotopy, this implies that $\psi$ is an element of $\cH_{0,1}^{2}$.
We recall from Lemma~\ref{lem:H-injects-into-Gamma} that $\cH_{0,1}^{2}\hookrightarrow \MCG_{0,1}^{2}$. It is a classical fact (see \cite[\S9.1]{farbmargalit} for instance) that $\MCG_{0,1}^{2}$ is isomorphic to $\bZ$, which generator $T_{m}$ is the Dehn twist along a meridian of $\Sigma_{1,1}$ (i.e.~a simple closed curve in $\partial V_{1,1}$ that bounding a nonseparating disc in $V_{1,1}$). Since the mapping class $\psi$ fixes the longitude simple closed curve $\gamma$, we deduce that $\psi$ cannot be a non-zero power of $T_{m}$. Hence we have $\psi=\id_{V_{1,1}}$, and so $\phi'=(\varphi\natural\id_{V_{1,1}})\circ \phi$. Therefore, by the equivalence relation in the colimit over $\sH^{\ge 1}_{r,s}$ of $\Hom_{\sH^{\ge 1}_{r,s}}(- \odot V_{1,1},V^{s}_{g,r})$, we conclude that $[B,\phi]=[B,\phi']$.
\end{proof}

\subsubsection{Proof of Theorem~\ref{thm:HS_stabilistation-1}}\label{sss:proof-HS-1}

We now prove Theorem~\ref{thm:HS_stabilistation-1} by applying Theorem~\ref{thm:RWW-main-thm} with the setting $(\cG,\odot,0)=(\sH^{\ge 1}_{1},\natural,I_{1})$, $(\cM,\odot)=(\sH^{\ge 1}_{r,s},\natural)$, $A=V^{s}_{0,r}$ and $X=V_{1,1}$.
First of all, we recall that the necessary hypotheses of having no zero-divisors, $\Aut_{\sH^{\ge 1}_{1}}(I_{1})=\{\id_{I_{1}}\}$ and local homogeneity at $(V^{s}_{0,r},V_{1,1})$ are verified in Propositions~\ref{prop:no-0-divisors-no-automorphisms-of-0} and \ref{prop:homogeneity-handlebody}\eqref{item:first-stabilisation} respectively. It thus only remains to check the high-connectivity condition of Theorem~\ref{thm:RWW-main-thm} for the semi-simplicial set of destabilisations $W_{g}(V^{s}_{0,r},V_{1,1})_{\bullet}$ for each $g\ge 1$. This is a consequence of the following result:
\begin{prop}\label{prop:stab1-W-g-S-g-homeomorphic}
For each $g\ge 1$, there is a homeomorphism
\[
|W_{g}(V^{s}_{0,r},V_{1,1})_{\bullet}|\overset{\cong}{\to} |Y^{A}(V_{g},\cD_{r}\sqcup\cP_{s},x_{0},x_{0})|.
\]
\end{prop}
\begin{proof}
We consider the map $\pi\colon |W_{g}(V^{s}_{0,r}, V_{1,1})_{\bullet}|\to |S_{g}(V^{s}_{0,r})_{\bullet}|$ of Notation~\ref{nota:simplicial-complex-destabilisation}.
For a fixed $p\ge 0$, we recall that a $p$-simplex in $S_{g}(V^{s}_{0,r},V_{1,1})_{p}$ has a canonical order on its vertices, induced by the orientation near the first marked disc in the boundary of $V^{s}_{0,r}\natural V_{1,1}^{\natural g}$. (In terms of the complex $Y^{A}((V_{g},\cD_{r}\sqcup\cP_{s},x_{0},x_{0})$, this is the canonical order on the arcs induced by the orientation.) Since the right $(\sH^{\ge 1}_{1},\natural,I_{1})$-module $(\sH^{\ge 1}_{r,s},\natural)$ is locally standard at $(V^{s}_{g,r},V_{1,1})$ by Proposition~\ref{prop:local-standardness}\eqref{item:first-stabilisation-standard}, it follows from Proposition~\ref{prop:LS-iff-Wn-determined-by-vertices} that, for each $\sigma\in S_{g}(V^{s}_{0,r},V_{1,1})_{p}$, each simplex in $\pi^{-1}_{p}(\sigma)$ is determined by its ordered set of vertices. Therefore, the preimage $\pi^{-1}_{p}(\sigma)$ contains exactly one simplex for each $\sigma\in S_{g}(V^{s}_{0,r},V_{1,1})_{p}$, so the map $\pi_{p}$ is a bijection for each $p\ge 0$ and then $\pi$ is thus a bijection. Combining this with Proposition~\ref{prop:stab1-S-g-iso-Y}, we deduce the intended homeomorphism as the composite $\pi \circ |\Phi^{s}_{g,r}|$.
\end{proof}
Combining Proposition~\ref{prop:stab1-W-g-S-g-homeomorphic} with Theorem~\ref{thm:HatcherWahl-connectivity1}, we deduce that the semi-simplicial set $W_{g}(V^{s}_{0,r},V_{1,1})_{\bullet}$ is $\frac{g-3}{2}$-connected for each $g\ge 1$, which thus proves Theorem~\ref{thm:HS_stabilistation-1} as a consequence of Theorem~\ref{thm:RWW-main-thm}.

\subsection{Handle stabilisation with respect to $\#$}\label{ss:handle-stabilisation-sharp}

We now prove a twisted homological stability result for the less standard genus stabilisation induced by attaching the genus-zero handlebody with two marked discs $V_{0,2}$ via the operation $\#$; see Theorem~\ref{thm:HS_stabilistation-2}. For all \S\ref{ss:handle-stabilisation-sharp}, we fix integers $r\ge 2$ and $s\ge 0$ and write for brevity $\sH^{\ge 2}_{r,s}$ (resp.~$\sH^{\ge 2}_{2}$) for $\sH^{\ge 2}_{V^{s}_{0,r},V_{0,2}}$ (resp.~$\sH^{\ge 2}_{V_{0,2}}$). We consider the category $([\sH^{\ge 2}_{r,s} , \sH^{\ge 2}_{2}\rangle,\#,I_{2})$ constructed from the right-module $(\sH^{\ge 2}_{r,s},\#)$ over the braided monoidal category $(\sH^{\ge 2}_{2},\#,I_{2})$ (see \S\ref{sss:groupoids-handlebodies}) via the Quillen bracket construction (see \S\ref{sss:Quillen-bracket-construction}). The objects of $[\sH^{\ge 2}_{r,s} , \sH^{\ge 2}_{2}\rangle$ are the handlebodies $V^{s}_{0,r}\# V_{0,2}^{\# g} \cong V^{s}_{g,r}$ (see \eqref{eq:natural-iso-decomposition}) for all $g\ge 0$, so we again deduce from Corollary~\ref{coro:maximal-subgroupoid-of-UGH} that we have isomorphisms:
\[
\cH^{s}_{g,r}\cong\Aut_{\sH^{\ge 2}_{r,s}}(V^{s}_{g,r})\cong\Aut_{[\sH^{\ge 2}_{r,s} , \sH^{\ge 2}_{2}\rangle}(V^{s}_{g,r}).
\]
Here, the stabilisation map is given by the group injection $\rho^{s}_{g,r}:=(-)_{g}\#\id_{V_{0,2}}\colon \cH^{s}_{g,r}\to\cH^{s}_{g+1,r}$ (see \eqref{eq:canonical_stabilisation-sharp}). We also write $d^{s}_{g,r}$ for the canonical morphism $[\id_{V^{s}_{g+1,r}},V_{0,2}] \colon V^{s}_{g,r}\to V^{s}_{g+1,r}$ of $[\sH^{\ge 2}_{r,s} , \sH^{\ge 2}_{2}\rangle$. We have the following twisted homological stability result with respect to this stabilisation, whose proof is done in \S\ref{sss:proof-HS-2}:
\begin{thm}\label{thm:HS_stabilistation-2}
We fix integers $r\ge 2$ and $s\ge 0$. Let $F\colon \sH \to\Ab$ be a double coefficient system at $V^{s}_{0,r}$ of degree $d$ (see Definition~\ref{def:double-coeff-system}). Then
\[(\rho^{s}_{g,r};F(d^{s}_{g,r}))_{i}\colon H_{i}(\cH^{s}_{g,r}; F(V^{s}_{g,r}))\to H_{i}(\cH^{s}_{g+1,r};F(V^{s}_{g+1,r}))\]
is an isomorphism for $i\le\frac{g-1}{2}-d-1$. If both coefficient systems are split, it is an isomorphism for $i\le\frac{g-1-d}{2}-1$. 
\end{thm}
\begin{rmk}\label{rmk:HS_stabilistation-2}
The assumption for the coefficients to form a double coefficient system is due to leveraging (a mild alternative of) Theorem~\ref{thm:HS_stabilistation-1} in the proof; see Theorem~\ref{thm:HS_stabilistation-1'}. Indeed, a key point to prove Theorem~\ref{thm:HS_stabilistation-2} is to use a ``stabilised version'' of the destabilisation complex introduced in \S\ref{sss:infinite-HS}, and then to utilise stability with respect to the analogue of $\sigma^{s}_{g,r}$ adding a handle on the left; see Theorem~\ref{thm:HS_stabilistation-1'}.
\end{rmk}

\subsubsection{Stable twisted homological stability}\label{sss:infinite-HS}

For the proof of Theorem~\ref{thm:HS_stabilistation-2}, we cannot proceed quite as for that of Theorem~\ref{thm:HS_stabilistation-1}, essentially because the connectivity of the corresponding complex of Definition~\ref{def:sc-stabilisation} is not known. We rather follow here the example of Hatcher and Wahl \cite[\S8.2]{HatcherWahl}: instead of directly considering the stablisation maps $\rho^{s}_{g,r}$, we first prove a twisted homological stability result for infinite handlebodies (in the sense of Definition~\ref{eq:def-infinite-in-barGH}); see Theorem~\ref{thm:stable-stability-for-sharp}.

\paragraph*{Stable simplicial complex of destabilisations.} We begin with studying the simplicial complex of destabilisations $S_{g}(V_{\infty},V_{0,2})_{\bullet}$ for each $g\ge 1$ (see Notation~\ref{nota:simplicial-complex-destabilisation}) associated to the infinite object $V^{s}_{\infty,r}$, showing that it is contractible; see Theorem~\ref{thm:sc-infty-contractible}.
To do this, we again use the suitable simplicial complexes of Definition~\ref{def:sc-stabilisation} for this situation. Namely, for each $g\ge 0$, we consider the simplicial complex $Y^{A}(V_{g},\cD_{r}\sqcup\cP_{s},x_{0},x_{1})$ associated to the object $V^{s}_{g,r}=(V_{g},r,s,i,j)$ with $x_{0}=i_{1}(1,0)$ and $x_{1}=i_{2}(1,0)$, where $i_{j}\colon \bD^{2}\to\partial V_{g}$ for $j \in \{1,2\}$ denotes the parametrisation map defining the marked disc $\cD^{j}_{V^{s}_{g,r}}$.
For brevity, we denote by $R^{s}_{g,r}$ the set $\cD_{r}\sqcup\cP_{s}$ of marked discs and marked points associated to $V_{g}$, and we write $\sH^{\ge 1}_{r,s}$ and $\sH^{\ge 1}_{1}$ for $\sH^{\ge 1}_{V^{s}_{0,r},V_{1,1}}$ and $\sH^{\ge 1}_{V_{1,1}}$ respectively.

For each $i\ge 1$, we consider the morphism $\iota_{i}(V^{s}_{0,r})=[V_{1,1},\id_{V^{s}_{i,r}}]\colon V^{s}_{i-1,r}\to V^{s}_{i,r}$ of $\langle \sH^{\ge 1}_{1}, \sH^{\ge 1}_{r,s} ]$. We note from Definition~\ref{def:monoidal-structure-cG} that the image of $(1,0)\in\bD^{2}$ in the parametrisation of any marked disc is the same in each $V_{i}$, as we are adding $V_{1,1}$ from the left, and so $x_{0}$ and $x_{1}$ are left invariant by $\iota_{i}(V^{s}_{0,r})$. Also, we remark that, although they are isomorphic, $R^{s}_{i,r}$ is obtained from $R^{s}_{i-1,r}$ by gluing the left hand side of the marked disc $V_{1,1}$ to the right hand half-disc $\cD^{1}_{V^{s}_{i,r}}$.
Therefore, the morphism $\iota_{i}(V^{s}_{0,r})$ induces by postcomposition a map of simplicial complexes $\iota^{Y}_{i}\colon Y^{A}(V_{i-1},R^{s}_{i-1,r},x_{0},x_{1})\to Y^{A}(V_{i},R^{s}_{i,r},x_{0},x_{1})$. Hence we obtain the following directed system of simplicial complexes:
\begin{equation}\label{eq:infinite-sc-stabilisation}
\begin{tikzcd}
    Y^{A}(V_{0},R^{s}_{0,r},x_{0},x_{1}) \arrow[r,"\iota^{Y}_{1}"]
    & Y^{A}(V_{1},R^{s}_{1,r},x_{0},x_{1})\arrow[r,"\iota^{Y}_{2}"]
    & \cdots \arrow[r,"\iota^{Y}_{i}"]
    & Y^{A}(V_{i},R^{s}_{i,r},x_{0},x_{1}) \arrow[r,"\iota^{Y}_{i+1}"]
    & \cdots
\end{tikzcd}
\end{equation}
Since the category of simplicial complexes is cocomplete, the following definition makes sense similarly to \cite[\S8.2]{HatcherWahl}:
\begin{defn}\label{def:colim-complexes-Y}
We define $Y^{A}(V^{s}_{\infty,r},R^{s}_{\infty,r},x_{0},x_{1})$ as the directed colimit of the diagram \eqref{eq:infinite-sc-stabilisation}.
\end{defn}
In order to study the connectivity of $S_{g}(V^{s}_{\infty,r},V_{0,2})_{\bullet}$, we start by proving the following property:
\begin{prop}\label{prop:iso-simplicial-complex-2}
For each integer $g\ge 1$, there is an isomorphism of simplicial complexes $S_{g}(V^{s}_{0,r},V_{0,2})_{\bullet}\cong Y^{A}(V^{s}_{g,r},R^{s}_{g,r},x_{0},x_{1})$.
\end{prop}
\begin{proof}
The definition of the morphism simplicial complexes $S_{g}(V^{s}_{0,r},V_{0,2})_{\bullet}\to Y^{A}(V^{s}_{g,r},R^{s}_{g,r},x_{0},x_{1})$ is very similar to that of \eqref{eq:iso-simplicial-complexes}, while the proof that this is well-defined and an isomorphism is almost identical to those of Lemma~\ref{lem:Phi-well-defined} and Proposition~\ref{prop:stab1-S-g-iso-Y}. Let us only sketch most of the proof and only include the details for the points that differ.

Recalling that $I$ denotes the unit interval $[0,1]$, let us make the identification $V_{0,2}=I\times\bD^{2}$, with $\{0\}\times\bD^{2}$ and $\{1\}\times\bD^{2}$ the marked discs parametrised in the obvious way. Let $y\in\partial\bD^{2}$ be the point such that $i_{1}(1,0)=(0,y)$ and $i_{2}(1,0)=(1,y)$, where $i_{1},i_{2}\colon \bD^{2}\to\partial V_{0,2}$ are the defining maps of the marked discs $\cD^{1}_{V_{0,2}}$ and $\cD^{2}_{V_{0,2}}$. We then define $\Phi_{0}\colon I+\bD^{2}\to V_{0,2}$ by 
\begin{align*}
    \Phi_{0}(x) = \begin{cases}
    (x,y) & \text{ if }x\in I,\\
    (1/2,x) &\text{ if }x\in\bD^{2}.
    \end{cases}
\end{align*}
We recall from Notation~\ref{nota:simplicial-complex-destabilisation} that a vertex of $S_{g}(V^{s}_{0,r},V_{0,2})_{\bullet}$ is a morphism $[B,\phi]$ in the colimit over $\sH^{\ge 2}_{r,s}$ of $\Hom_{\sH^{\ge 2}_{r,s}}(- \odot V_{0,2},V^{s}_{g,r})$, i.e.~an equivalence class of pairs $(B,\phi)$ where $\phi\colon B\# V_{0,2}\to V^{s}_{g,r}$ is a morphism in $\sH^{\ge 2}_{r,s}$. We denote by $\Phi'_{0}$ the isotopy class of the map $I+\bD^{2}\to B\# V_{0,2}$ defined by postcomposing the map $\Phi_{0}\colon I+\bD^{2}\to V_{0,2}$ with the canonical embedding $V_{0,2}\hookrightarrow B\# V_{0,2}$ induced by the monoidal product $\#$.
We may then define a map of simplicial complexes $\Phi^{s}_{g,r}\colon S_{g}(V^{s}_{0,r},V_{0,2})_{\bullet}\to Y^{A}(V_{g},\cD_{r}\sqcup\cP_{s},x_{0},x_{1})$ by sending $[B,\phi]$ to the isotopy class $\phi\circ \Phi'_{0}$. That this gives a well-defined, surjective map of simplicial complexes is proved by repeating mutatis mutandi the proof of Lemma~\ref{lem:Phi-well-defined} and the first half of Proposition~\ref{prop:stab1-S-g-iso-Y}. The proof of injectivity also goes through in a similar way to the second half of the proof of Proposition~\ref{prop:stab1-S-g-iso-Y}: by a routine verbatim adaptation, the proof boils down to showing that given $\psi\in\Aut_{\sH^{\ge 2}_{2}}(V_{0,2})=\cH_{0,2}$ which fixes the isotopy class of the embedding $\Phi_{0}\colon I+\bD^{2}\to V_{0,2}$, we have $\psi=\id_{V_{0,2}}$. Recall from Lemma~\ref{lem:H-injects-into-Gamma} that $\cH_{0,2}\hookrightarrow \MCG_{0,2}$ and it is a classical fact (see \cite[\S9.1]{farbmargalit} for instance) that $\MCG_{0,2} \cong \bZ$, generated by a Dehn twist around the meridian $\{t\}\times\bD^{2}$ for any $0<t<1$. However, such a Dehn twist does not fix the isotopy class of the embedding $\Phi_{0}$, so the claim follows.
\end{proof}
A fundamental theorem of \cite{HatcherWahl} is that the simplicial complex $Y^{A}(V^{s}_{\infty,r},R^{s}_{\infty,r},x_{0},x_{1})$ is contractible; see \cite[Thm.~8.6]{HatcherWahl}. By using Proposition~\ref{prop:iso-simplicial-complex-2}, we are now ready to show the following key result:
\begin{thm}\label{thm:sc-infty-contractible}
For each $g\ge 1$, the simplicial complex $S_{g}(V_{\infty},V_{0,2})_{\bullet}$ is contractible.
\end{thm}
\begin{proof}
For each $g\ge 1$, we recall that $V^{s}_{\infty,r}\cong V^{s}_{\infty,r}\# V_{0,2}^{\# g} $ thanks to the isomorphism \eqref{eq:canonical-iso-dash-natural} of Lemma~\ref{lem:stabilising-handlebody-of-inf-rank}, and that we have an isomorphism $V^{s}_{g,r}\cong V^{s}_{0,r}\#V_{0,2}^{\# g}$ (see \eqref{eq:natural-iso-decomposition}). Then, we deduce from the definition of the simplicial complex of destabilisations (see Notation~\ref{nota:simplicial-complex-destabilisation}) the following simplicial complex isomorphisms
\begin{equation}\label{eq:iso-SC-infty}
S_{0}(V^{s}_{\infty,r},V_{0,2})_{\bullet}\cong S_{g}(V^{s}_{\infty,r},V_{0,2})_{\bullet}
\end{equation}
and
\begin{equation}\label{eq:iso-SC-g-0}
S_{g}(V^{s}_{0,r},V_{0,2})_{\bullet}\cong S_{0}(V^{s}_{g,r},V_{0,2})_{\bullet}.
\end{equation}
Now, each morphism $\iota_{g}(V^{s}_{0,r})$ for $g\ge 1$ induces by postcomposition a map of simplicial complexes $\iota^{S}_{i}\colon S_{0}(V^{s}_{g-1,r},V_{0,2})_{\bullet}\to S_{0}(V^{s}_{g,r},V_{0,2})_{\bullet}$, and we denote by $\colim_{g} \left(S_{0}(V^{s}_{g,r},V_{0,2})_{\bullet}\right)$ the directed colimit of the sequential system defined by these maps $\{\iota^{S}_{g}\}_{g\ge 1}$. By the universal property of a colimit, it then follows from Proposition~\ref{prop:iso-simplicial-complex-2} and the isomorphism \eqref{eq:iso-SC-g-0} that we have an isomorphism of simplicial complexes
\begin{equation}\label{eq:iso-SC-0-Y}
\colim_{g} \left(S_{0}(V^{s}_{g,r},V_{0,2})_{\bullet}\right) \cong Y^{A}(V^{s}_{\infty,r},R^{s}_{\infty,r},x_{0},x_{1}).
\end{equation}
Furthermore, since the postcomposition by the structural maps $V^{s}_{g,r}\to V^{s}_{\infty,r}$ for each $g\ge 0$ induces a map of simplicial complex of destabilisations $S_{0}(V^{s}_{g,r},V_{0,2})_{\bullet}\to S_{0}(V^{s}_{\infty,r},V_{0,2})_{\bullet}$, we deduce that there is a canonical map of simplicial complexes induced by the universal property of a colimit
\[
\Xi_{\bullet} \colon \colim_{g} \left(S_{0}(V^{s}_{g,r},V_{0,2})_{\bullet}\right) \to S_{0}(V^{s}_{\infty,r},V_{0,2})_{\bullet}.
\]
Let us describe more precisely what happens for the vertices. By Definition~\ref{def:groupoids-handlebody-infinity}, a morphism $\varphi \in S_{0}(V^{s}_{\infty,r},V_{0,2})_{0}$ is the colimit map induced by a set of morphisms $\{\varphi_{g}=[V^{s}_{g-1,r},\phi_{g}]\}_{g\ge j}$ for some integer $j\ge 0$, with $\phi_{g}\in \Aut_{\sH^{\ge 2}}(V^{s}_{g,r})$ for some integer $j\ge 0$ and making the following diagram commutative
\[
\begin{tikzcd}[column sep=.7em]
    V_{0,2}\arrow{d}{\varphi_{j}}\arrow[rrrrrr,"\id_{V_{0,2}}"]&&&&&&
    V_{0,2}\arrow{d}{\varphi_{j+1}}\arrow[rrrrrr,"\id_{V_{0,2}}"] &&&&&&
    \cdots \arrow[rrrrrr,"\id_{V_{0,2}}"] &&&&&&
    V_{0,2}\arrow{d}{\varphi_{g}} \arrow[rrrrrr,"\id_{V_{0,2}}"] &&&&&&
    \cdots
    \\
    V^{s}_{j,r}\arrow[rrrrrr,"\iota_{j+1}(V_{0,2})"] &&&&&&
    V^{s}_{j+1,r}\arrow[rrrrrr,"\iota_{j+2}(V_{0,2})"] &&&&&& 
    \cdots \arrow[rrrrrr,"\iota_{g}( V_{0,2})"] &&&&&&
    V^{s}_{g,r}\arrow[rrrrrr,"\iota_{g+1}(V_{0,2})"]  &&&&&&
    \cdots.
\end{tikzcd}
\]
In particular, each $\varphi_{g}=[V^{s}_{g-1,r},\phi_{g}]$ belongs to $S_{0}(V^{s}_{g,r},V_{0,2})_{0}$ by definition (see Notation~\ref{nota:simplicial-complex-destabilisation}), and so the morphisms $\{\varphi_{g}\}_{g\ge j}$ induce an element $[\varphi]$ in $\colim_{g} \left(S_{0}(V^{s}_{g,r},V_{0,2})_{0}\right)$ by the universal property of that colimit. Assigning $\varphi \mapsto [\varphi]$ then clearly induces a well-defined map $\Upsilon\colon  S_{0}(V^{s}_{\infty,r},V_{0,2})_{0} \to \colim_{g} \left(S_{0}(V^{s}_{g,r},V_{0,2})_{0}\right)$.
Writing the elementary explicit assignment of the map $\Xi_{0}$ on an element $\colim_{g} \left(S_{0}(V^{s}_{g,r},V_{0,2})_{0}\right)$, it is a clear routine to check that is an $\Upsilon$ is the inverse of $\Xi_{0}$. Therefore, the map $\Xi$ is a bijection on vertices, and so it is an isomorphism of simplicial complexes.
Hence, composing $\Xi$ with the isomorphisms \eqref{eq:iso-SC-infty} and \eqref{eq:iso-SC-0-Y}, we obtain a simplicial complex isomorphism $S_{g}(V^{s}_{\infty,r},V_{0,2})_{\bullet}\cong Y^{A}(V^{s}_{\infty,r},R^{s}_{\infty,r},x_{0},x_{1})$. The contractibility of $S_{g}(V^{s}_{\infty,r},V_{0,2})_{\bullet}$ thus follows from that of $Y^{A}(V^{s}_{\infty,r},R^{s}_{\infty,r},x_{0},x_{1})$ proved in \cite[Thm.~8.6]{HatcherWahl}.
\end{proof}

\paragraph*{Twisted homological stability for infinite handlebodies.}
We are now ready to prove the twisted homological stability result Theorem~\ref{thm:stable-stability-for-sharp} for infinite handlebodies. We follow the method of Theorem~\ref{thm:RWW-main-thm}, with the setting $(\cG,\odot,0)=(\sH^{\ge 2}_{2},\natural,I_{2})$, $(\cM,\odot)=(\overline{\sH}^{\ge 2}_{r,s},\natural)$, $A=V^{s}_{\infty,r}$ and $X=V_{0,2}$, where we write $\overline{\sH}^{\ge 2}_{r,s}$ for $\overline{\sH}^{\ge 2}_{V^{s}_{\infty,r},V_{0,2}}$ for brevity.
Beforehand, let us fix some notation. We recall that the category of groups is cocomplete (see \cite[\S V.1.]{MacLane1} for instance). We denote by $\cH^{s}_{\infty,r}$ the group of automorphisms $\Aut_{\overline{\sH}^{\ge 2}_{r,s}}(V^{s}_{\infty,r})$, i.e.~equivalently the colimits of the groups $\{\cH^{s}_{g,r}\}_{g\ge 0}$ with respect to the maps $\{\iota_{g}(V^{s}_{0,r})\}_{g\ge 0}$; see Definition~\ref{def:groupoids-handlebody-infinity}. We write $\rho^{s}_{\infty,g,r}\colon \Aut_{\overline{\sH}^{\ge 2}_{r,s}}(V^{s}_{\infty,r}\# V^{\# g}_{0,2})\to \Aut_{\overline{\sH}^{\ge 2}_{r,s}}(V^{s}_{\infty,r}\# V^{\# g+1}_{0,2})$ for the colimit of the stabilisation maps $\{\rho^{s}_{i,r}\}_{i\ge 0}$ (see \eqref{eq:canonical_stabilisation-sharp}) with respect to the maps $\{\iota_{j}(V^{s}_{0,r}\# V^{\# g}_{0,2})\}_{j\ge 0}$.
Finally, let $d^{s}_{\infty,g,r}$ denote the canonical morphism $[\id_{V^{s}_{\infty,r}\# V^{\# g}_{0,2}},V_{0,2}]$ of $[\overline{\sH}^{\ge 2}_{r,s},\sH^{\ge 2}_{2} \rangle$.

\begin{thm}\label{thm:stable-stability-for-sharp}
We fix integers $r\ge 2$ and $s\ge 0$. Let $F\colon [ \overline{\sH}^{\ge 2}_{r,s},\sH^{\ge 2}_{2} \rangle \to \Ab$ be a coefficient system of degree $d$. For any $i\ge 0$, the map
\[
(\rho^{s}_{\infty,0,r};F(d^{s}_{\infty,1,r}))_{i}\colon H_{i}(\cH^{s}_{\infty,r}; F(V^{s}_{\infty,r}))\to H_{i}(\Aut_{\overline{\sH}^{\ge 2}_{r,s}}(V^{s}_{\infty,r}\# V_{0,2}); F(V^{s}_{\infty,r}\# V_{0,2})).
\]
is an isomorphism.
\end{thm}

\begin{proof}
First, we note that the same argument as in the proof of Proposition~\ref{prop:stab1-W-g-S-g-homeomorphic} works to show that the map $\pi\colon |W_{g}(V_{\infty},V_{0,2})_{\bullet}|\to |S_{g}(V_{\infty},V_{0,2})_{\bullet}|$ (see Notation~\ref{nota:simplicial-complex-destabilisation})
is a homeomorphism for each $g\ge 1$. Hence it follows from Theorem~\ref{thm:sc-infty-contractible} that the semi-simplicial set $W_{g}(V^{s}_{\infty,r},V_{0,2})_{\bullet}$ is contractible for each $g\ge 1$, and so it is $k$-connected for any $k\ge 0$.
Moreover, the necessary hypotheses of having no zero-divisors, $\Aut_{\sH^{\ge 2}_{2}}(I_{2})=\{\id_{I_{2}}\}$ and local homogeneity at $(V^{s}_{\infty,r},V_{0,2})$ are verified in Propositions~\ref{prop:no-0-divisors-no-automorphisms-of-0} and \ref{prop:homogeneity-handlebody}\eqref{item:infinite-stabilisation} respectively. Therefore, by applying Theorem~\ref{thm:RWW-main-thm}, we deduce that for any $i\ge 0$, there exists $g\gg i+d$ such that the map $(\rho^{s}_{\infty,g,r};F(d^{s}_{\infty,g,r}))_{i}$ of the form
\[
H_{i}(\Aut_{\overline{\sH}^{\ge 2}_{r,s}}(V^{s}_{\infty,r}\# V_{0,2}^{\# g}); F(V^{s}_{\infty,r}\# V_{0,2}^{\# g}))\to H_{i}(\Aut_{\overline{\sH}^{\ge 2}_{r,s}}(V^{s}_{\infty,r}\# V_{0,2}^{\# (g+1)}); F(V^{s}_{\infty,r}\# V_{0,2}^{\# g+1}))
\]
is an isomorphism. Then, we deduce from the isomorphism \eqref{eq:canonical-iso-dash-natural} of Lemma~\ref{lem:stabilising-handlebody-of-inf-rank} that we have an isomorphism $V^{s}_{\infty,r}\# V_{0,2}^{\#g}\cong V^{s}_{\infty,r}$ for each $g\ge 0$, and thus also $V^{s}_{\infty,r}\# V_{0,2}^{\#(g+1)}\cong V^{s}_{\infty,r}\# V_{0,2}$. By the functoriality of $F$, we also have isomorphisms $F(V^{s}_{\infty,r}\# V_{0,2}^{\# g})\cong F(V^{s}_{\infty,r})$ and similarly $F(V^{s}_{\infty,r}\# V_{0,2}^{\# g+1})\cong F(V^{s}_{\infty,r}\# V_{0,2})$, proving the theorem.
\end{proof}

\subsubsection{Proof of Theorem~\ref{thm:HS_stabilistation-2}}\label{sss:proof-HS-2}

We can now prove Theorem~\ref{thm:HS_stabilistation-2} from Theorem~\ref{thm:HS_stabilistation-1} and Theorem~\ref{thm:stable-stability-for-sharp}.

\paragraph*{$V_{1,1}\natural(-)$-stabilisation version of Theorem~\ref{thm:HS_stabilistation-1}.} Actually, we need an alternative version of Theorem~\ref{thm:HS_stabilistation-1}, where we stabilise with $V_{1,1}$ on the left; see Theorem~\ref{thm:HS_stabilistation-1'}. Namely, we now work in the category $\langle \sH^{\ge 1}_{1}, \sH^{\ge 1}_{r,s}]$ defined by the Quillen bracket construction (see \S\ref{sss:Quillen-bracket-construction}) applied to $(\sH^{\ge 1}_{r,s},\natural)$ seen as a \emph{left}-module over the braided monoidal category $(\sH^{\ge 1}_{1},\natural,I_{1})$ (see \S\ref{sss:groupoids-handlebodies}). The stabilisation map is defined by the map $\bar{\sigma}^{s}_{g,r}:=\id_{V_{1,1}}\natural(-)\colon \cH^{s}_{g,r}\hookrightarrow \cH^{s}_{g+1,r}$, i.e.~analogously to \eqref{eq:canonical_stabilisation-natural} by extending by the identity on the left, for which Proposition~\ref{prop:pullback-square-handlebody-MCG} repeats mutatis mutandis.
Also, we write $\bar{c}_{g,1}$ for the canonical morphism $[V_{1,1},\id_{V^{s}_{g+1,r}}]\colon V^{s}_{g,r}\to V^{s}_{g+1,r}$ of $\langle \sH^{\ge 1}_{1}, \sH^{\ge 1}_{r,s}]$. Moreover, the notion of finite degree coefficient system from Definition~\ref{def:finite-deg-coeff-system} as well as the examples of \S\ref{sss:first-homologies-coeff-systems} adapt verbatim for functors of the form $\langle \sH^{\ge 1}_{1}, \sH^{\ge 1}_{r,s}] \to \Ab$.

\begin{thm}\label{thm:HS_stabilistation-1'}
We fix integers $r\ge 1$ and $s\ge 0$. Let $F\colon \langle \sH^{\ge 1}_{r,s},\sH^{\ge 1}_{1} ] \to \Ab$ be a coefficient system of degree $d$ at $0$. Then
\[
(\bar{\sigma}^{s}_{g,r};F(\bar{c}_{g,1}))_{i}\colon H_{i}(\cH^{s}_{g,r};F(V^{s}_{g,r}))\to H_{i}(\cH^{s}_{g+1,r};F(V^{s}_{g+1,r}))
\]
is an isomorphism for $i\le \frac{g-1}{2}-d-1$. If $F$ is split, it is an isomorphism for $i\le\frac{g-1-d}{2}-1$.
\end{thm}
\begin{proof}
First of all, the work of \S\ref{sss:HS-meta-theorem} repeats verbatim to provide an analogous version of Theorem~\ref{thm:RWW-main-thm} for a functor $F\colon \langle \sH^{\ge 1}_{r,s},\sH^{\ge 1}_{1} ] \to \Ab$ and a suitable corresponding semi-simplicial set $\bar{W}_{g}(V^{s}_{0,r},V_{1,1})_{\bullet}$.
Then, the hypotheses of having no zero-divisors, $\Aut_{\sH^{\ge 1}_{1}}(I_{1})=\{\id_{I_{1}}\}$ are done in Proposition~\ref{prop:no-0-divisors-no-automorphisms-of-0}, while the alternative version of local homogeneity at $(V^{s}_{0,r},V_{1,1})$ in this context is checked by adapting mutatis mutandis Proposition~\ref{prop:homogeneity-handlebody}\eqref{item:first-stabilisation}. Also, the work of \S\ref{sss:connectivity-ss} carries over mutatis mutandis as well as the proof of Proposition~\ref{prop:stab1-W-g-S-g-homeomorphic}, which proves that the semi-simplicial set $\bar{W}_{g}(V^{s}_{0,r},V_{1,1})_{\bullet}$ is $\frac{g-3}{2}$-connected for each $g\ge 1$. The result thus follows from the adapted version of Theorem~\ref{thm:RWW-main-thm}.
\end{proof}
Furthermore, we need the following commutativity result to prove Theorem~\ref{thm:HS_stabilistation-2}:
\begin{lem}\label{lem:sigma-bar-rho-commute}
For each $g\ge 0$, there is a commutative diagram of group injections
\begin{equation}\label{eq:commutative-square-injections-sigma-bar-rho}
\centering
\begin{split}
\begin{tikzpicture}
[x=1mm,y=1mm]
\node (tl) at (0,12) {$\cH^{s}_{g,r}$};
\node (tr) at (35,12) {$\cH^{s}_{g+1,r}$};
\node (bl) at (0,0) {$\cH^{s}_{g+1,r}$};
\node (br) at (35,0) {$\cH^{s}_{g+2,r}.$};
\draw[right hook->] (tl) to node[above,font=\small]{$\bar{\sigma}^{s}_{g,r}$} (tr);
\draw[right hook->] (bl) to node[above,font=\small]{$\bar{\sigma}^{s}_{g+1,r}$} (br);
\draw[right hook->] (tl) to node[left,font=\small]{$\rho^{s}_{g,r}$} (bl);
\draw[right hook->] (tr) to node[right,font=\small]{$\rho^{s}_{g+1,r}$} (br);
\end{tikzpicture}
\end{split}
\end{equation}
\end{lem}
\begin{proof}
We consider the canonical inclusions $\mathfrak{i}_{1}(g)\colon V^{s}_{g,r}\hookrightarrow V_{1,1}\natural V^{s}_{g,r}$ and $\mathfrak{i}_{2}(g)\colon V^{s}_{g,r}\hookrightarrow V^{s}_{g,r} \# V_{0,2}$.
For an element $\phi\in \cH^{s}_{g,r}$, recall that a representative diffeomorphism $\tilde{\phi}$ is isotopic to the identity in a neighbourhood of the gluing discs. By definition, the morphisms $\bar{\sigma}^{s}_{g,r}$ and $\rho^{s}_{g,r}$ consists in extending (up to isotopy) $\phi$ by the identity over the complement of the canonical inclusions $\mathfrak{i}_{1}(g)$ and $\mathfrak{i}_{2}(g)$ respectively.
Then we note that the extension along $\mathfrak{i}_{1}(g)$ and then that along $\mathfrak{i}_{2}(g+1)$ is the same as extending along $\mathfrak{i}_{2}(g)$ and then along $\mathfrak{i}_{1}(g+1)$, because in both cases we simply prolongate isotopy classes of diffeomorphisms by the identity in regions which have disjoint support up to isotopy.
Therefore, we have an equality $(\rho^{s}_{g+1,r}\circ \bar{\sigma}^{s}_{g,r})(\phi) =(\id_{V_{1,1}}\natural \phi) \# \id_{V_{0,2}} = \id_{V_{1,1}}\natural (\phi \# \id_{V_{0,2}})=(\bar{\sigma}^{s}_{g+1,r}\circ \rho^{s}_{g,r})(\phi)$ for each $\phi\in \cH^{s}_{g,r}$, and so the diagram \eqref{eq:commutative-square-injections-sigma-bar-rho} is commutative.
\end{proof}

\paragraph*{Proving Theorem~\ref{thm:HS_stabilistation-2}.}
We now use a trick from \cite[\S8.4]{HatcherWahl} to prove Theorem~\ref{thm:HS_stabilistation-2}. We fix an integer $i\ge 0$. 
By the commutativity of the diagrams \eqref{eq:commutative-square-injections-sigma-bar-rho} (for the groups) and \eqref{eq:commutative-diag-extension-functor} (for the coefficients) Definition~\ref{def:double-coeff-system} and the functoriality of group homology (see \cite[Chap.~III,\S8]{Brown} for instance), we obtain the following commutative diagram:
\[
\begin{tikzcd}[column sep=.7em]
H_{i}(\cH^{s}_{0,r};F(V^{s}_{0,r})) \arrow[rr,"\varPsi^{s}_{0,r}"] \arrow{d}{(\rho^{s}_{0,r},d^{s}_{0,r})_{*}} &&
\cdots \arrow[rr,"\varPsi^{s}_{g-1,r}"] &&
H_{i}(\cH^{s}_{g,r};F(V^{s}_{g,r}))\arrow[rrr,"\varPsi^{s}_{g,r}"]\arrow{d}{(\rho^{s}_{g,r},d^{s}_{g,r})_{*}}  &&&
H_{i}(\cH^{s}_{g+1,r};F(V^{s}_{g+1,r}))\arrow[rr,"\varPsi^{s}_{g+1,r}"]\arrow{d}{(\rho^{s}_{g+1,r},d^{s}_{g+1,r})_{*}}  &&
\cdots
\\
H_{i}(\cH^{s}_{1,r};F(V^{s}_{1,r})) \arrow[rr,"\varPsi^{s}_{1,r}"] &&
\cdots \arrow[rr,"\varPsi^{s}_{g,r}"] &&
H_{i}(\cH^{s}_{g+1,r};F(V^{s}_{g+1,r}))\arrow[rrr,"\varPsi^{s}_{g+1,r}"]  &&& H_{i}(\cH^{s}_{g+2,r};F(V^{s}_{g+2,r}))\arrow[rr,"\varPsi^{s}_{g+2,r}"]  && \cdots
\end{tikzcd}
\]
where $\varPsi^{s}_{g,r}=(\bar{\sigma}^{s}_{g,r}{,}\bar{c}_{g,1})_{i}$.
By Theorem~\ref{thm:HS_stabilistation-1'}, the map $(\bar{\sigma}^{s}_{g,r}{,}\bar{c}_{g,1})_{i}$ is an isomorphism for $i\le \frac{g-1}{2}-r-1$.
Also, we recall from \S\ref{sss:double-coeff-systems} and Proposition~\ref{prop:extension-coeff-system-infty} that the double coefficient system $F$ induces by definition a coefficient system $\overline{F}\colon [\overline{\sH}^{\ge 2}_{V^{s}_{\infty,r},V_{0,2}},\sH^{\ge 2}_{V_{0,2}}\rangle \to \Ab$, whose assignment on $[\id_{V^{s}_{\infty,r}\# V_{0,2}},V_{0,2}]$ is defined as the colimit of the vertical abelian group maps defined by diagram \eqref{eq:commutative-diag-extension-functor} for $m=n-1=0$.
Since twisted group homology commutes with filtered colimits (as a consequence of \cite[Th.~2.6.15]{Weibel} for instance), we deduce that the colimit of the vertical arrows in the above diagram is the isomorphism $(\rho^{s}_{\infty,0,r};\overline{F}(d^{s}_{\infty,1,r}))_{i}$ of Theorem~\ref{thm:stable-stability-for-sharp}. Hence $(\rho^{s}_{g,r}{,}d^{s}_{g,r})_{*}$ is an isomorphism once $(\bar{\sigma}^{s}_{g,r}{,}\bar{c}_{g,1})_{i}$ has reached its range of stability, whence the result of Theorem~\ref{thm:HS_stabilistation-2}.

\subsection{Independence of number of marked discs}\label{ss:HS-marked-discs}

Combining Theorems~\ref{thm:HS_stabilistation-1} and \ref{thm:HS_stabilistation-2}, we may now prove the twisted homological result of Theorem~\ref{thm:mainB}. In order to state it, let us start by fixing some notation. For all \S\ref{ss:HS-marked-discs}, we fix integers $r\ge 1$ and $s\ge 0$. Let $\bar{\natural}:=\bar{\natural}_{1}$ and $\bar{\#}:=\bar{\natural}_{2}$ be the operations analogous to $\natural=\natural_{1}$ and $\#=\natural_{2}$ respectively where we glue handlebodies along the full marked discs rather than just along half-discs Definition~\ref{def:monoidal-structure-cG}, i.e.~$V_{1}\bar{\natural}_{r} V_{2}$ with $r\in\{1,2\}$ is the adjunction space $V_{1}\cup_{\bar{f}_{1}\sim \bar{f}_{2}} V_{2}$ where $\bar{f}_{k} \colon \sqcup_{r}\bD^{2}\hookrightarrow \Image(i_{k})$ denotes the canonical inclusion into the first $r$ marked discs of $V_{k}$; see Figures~\ref{fig:Stabilisation_mu} and \ref{fig:Stabilisation}.

\begin{rmk}
The operations $\bar{\natural}$ and $\bar{\#}$ do not define monoidal structures on $\sH^{\ge 1}$ and $\sH^{\ge 2}$ respectively, which is the reason why we do not work with them in \S\ref{ss:handle-stabilisation-natural} and \S\ref{ss:handle-stabilisation-sharp}.
\end{rmk}

For each $g\ge 0$, repeating the framework of Proposition~\ref{prop:pullback-square-handlebody-MCG}, we consider the injective group morphisms $\mu^{s}_{g,r}:=(-)_{g}\bar{\natural} \id_{V_{0,3}}\colon \cH^{s}_{g,r} \hookrightarrow \cH^{s}_{g,r+1}$ and $\nu^{s}_{g,r+1}:=(-)_{g}\bar{\#} \id_{V_{0,3}}\colon \cH^{s}_{g,r+1} \hookrightarrow \cH^{s}_{g+1,r}$ induced by extending the isotopy classes of diffeomorphisms by the identity on $V_{0,3}$, analogous to the injections $\sigma^{s}_{g,r}$ and $\rho^{s}_{g,r}$ (see \eqref{eq:canonical_stabilisation-natural} and \eqref{eq:canonical_stabilisation-sharp}).
Furthermore, following Notation~\ref{nota:standard-notation-HS}, we write $m^{s}_{g,r}$ and $n^{s}_{g,r+1}$ for the equivalence classes of pairs $[\id_{V^{s}_{g,r+1}},V_{0,3}]_{\bar{\natural}}\colon V^{s}_{g,r} \to V^{s}_{g,r+1}$ and $[\id_{V^{s}_{g,r}},V_{0,3}]_{\bar{\#}}\colon V^{s}_{g,r+1} \to V^{s}_{g+1,r}$ respectively, where the pair $(\phi,V_{0,3})$ is equivalent to $(\id_{V^{s}_{g,r+1}},V_{0,3})$ (resp.~$(\id_{V^{s}_{g+1,r}},V_{0,3})$) if there exists an isomorphism $\chi \in \cH_{0,3}$ such that $\phi = \id_{V^{s}_{g,r+1}} \bar{\natural} \chi$ (resp.~$\phi= \id_{V^{s}_{g+1,r}} \bar{\#} \chi$).
Since the classes of the diffeomorphisms in the handlebody groups are isotopic to the identity on a collar neighbourhood of each marked disc, it is a routine to check from these definitions that we have equalities of the following group morphisms for each $g\ge 0$:
For each $g\ge 0$, we have equalities
\begin{equation}\label{eq:sigma-decomposition}
\sigma^{s}_{g,r} = \nu^{s}_{g,r+1}\circ \mu^{s}_{g,r} \,\,\textrm{ and }\,\, c^{s}_{g,r} = n^{s}_{g,r+1}\circ m^{s}_{g,r}
\end{equation}
while
\begin{equation}\label{eq:rho-decomposition}
\rho^{s}_{g,r} = \mu^{s}_{g,r}\circ \nu^{s}_{g,r+1} \,\,\textrm{ and }\,\, d^{s}_{g,r} = m^{s}_{g,r}\circ n^{s}_{g,r+1}.
\end{equation}
We deduce the following more detailed version of Theorem~\ref{thm:mainB}:
\begin{coro}\label{cor:independence-of-marked-discs}
For integers $r\ge 1$ and $s\ge0$, let $F\colon \sH\to\Ab$ be a coefficient bisystem at $V^{s}_{0,r}$ of degree $d \ge 0$ (see Definition~\ref{def:double-coeff-system}). Then both of the maps
\begin{align*}
    (\mu^{s}_{g,r};F(m^{s}_{g,r}))_{i}\colon H_{i}(\cH^{s}_{g,r}; F(V^{s}_{g,r}))\to H_{i}(\cH^{s}_{g,r+1};F(V^{s}_{g,r+1})),\\
    (\nu^{s}_{g,r+1};F(n^{s}_{g,r+1}))_{i}\colon H_{i}(\cH^{s}_{g,r+1}; F(V^{s}_{g,r+1}))\to H_{i}(\cH^{s}_{g,r};F(V^{s}_{g,r})),
\end{align*}
are isomorphisms for $i\le\frac{g-1}{2}-d-1$. If both coefficient systems are split, they are isomorphisms for $i\le\frac{g-1-d}{2}-1$. 
\end{coro}
\begin{proof}
By Theorem~\ref{thm:HS_stabilistation-1}, it follows from the decomposition of an isomorphism into two composite factors and the equalities \eqref{eq:sigma-decomposition} that $(\mu^{s}_{g,r};F(m^{s}_{g,r}))_{i}$ is injective and that $(\nu^{s}_{g,r+1};F(n^{s}_{g,r+1}))_{i}$ is surjective for the ranges stated in the result. Similarly, we deduce from Theorem~\ref{thm:HS_stabilistation-2} combined with the equalities \eqref{eq:rho-decomposition} that $(\mu^{s}_{g,r};F(m^{s}_{g,r}))_{i}$ is surjective and that $(\nu^{s}_{g,r+1};F(n_{g,r+1}))_{i}$ is injective in the same ranges, whence the result.
\end{proof}
\begin{eg}\label{eg:HW-recover}
Assigning $s=0$ and $F=\bZ$, Corollary~\ref{cor:independence-of-marked-discs} corresponds to \cite[Th.~1.8(ii)]{HatcherWahl} and may be seen as a generalisation of that previous result to any $s\ge 0$ and to twisted coefficients given by coefficient bisystems. For instance, as explained in Example~\ref{eg:extending-functors-first-homologies}, the functors $T^{d}H_{\Sigma}$ and $T^{d}H_{\cV}$ define coefficient bisystems at $V^{s}_{0,r}$ of degree $d \ge 0$, and thus Corollary~\ref{cor:independence-of-marked-discs} applies to these functors.
\end{eg}

\begin{rmk}\label{rmk:non-stability-no-marked-discs}
The theorem \cite[Th.~1.8(ii)]{HatcherWahl} also shows that the map $(\nu_{g,1};\bZ)_{i}$ (i.e.~$s=r=0$ and $F=\bZ$) is an isomorphism for $g\ge 2i+4$. However, Corollary~\ref{cor:independence-of-marked-discs} does not hold for $r=0$ with finite degree coefficient bisystem. Indeed, Ishida and Sato \cite[Th.~1.1-1.2]{IshidaSato} provide a counterexample, by showing that $(\nu_{g,1};H_{\Sigma}(n^{s}_{g,r}))_{1}$ is not an isomorphism for $g\ge 4$.
\end{rmk}

\begin{rmk}\label{rmk:Boldsen-comparison}
Analogous results to those of Corollary~\ref{cor:independence-of-marked-discs} above for mapping class groups of surfaces are proven in \cite[Th.~4.13]{Boldsen}, but under the stronger assumption that the involved coefficient systems are \emph{split} (following the terminology of Definition~\ref{def:finite-deg-coeff-system}, see \cite[Def.~4.2]{Boldsen}). However, the framework of \S\ref{s:background} and studies of \S\ref{s:homological_stability} leading to Corollary~\ref{cor:independence-of-marked-discs} may be repeated mutatis mutandis for mapping class groups of surfaces, and so this assumption appears to us to be unnecessary.
\end{rmk}

\section{Moduli spaces of handlebodies with tangential structures}\label{s:HS-moduli-tangential-structures}

In this section, we apply the results of \S\ref{s:homological_stability} to prove Theorem~\ref{thm:mainC} (see \S\ref{ss:proof-of-thmC}) after some preliminary recollections on tangential structures (see \S\ref{def:tangential-structures}).

\subsection{Tangential structures}\label{def:tangential-structures}

We start by recalling the definition and some basics on tangential structures. Throughout the section, $\Orth(d)$ denotes the group of $d\times d$ orthogonal matrices and $\B\Orth(d)$ its classifying space. Furthermore, we let $\gamma_{d}\to\B\Orth(d)$ denote the tautological $d$--dimensional vector bundle over $\B\Orth(d)$.

\begin{defn}\label{def:theta-structure}
For $d\ge 1$, a $d$-tangential structure is a fibration $\theta\colon B \to \B\Orth(d)$. For $M$ an $m$-manifold with $m\le d$, a $\theta$-structure on $M$ is a map of vector bundles (i.e.~fibrepreserving map whose restriction to each fibre is a linear isomorphism; see for example \cite[Chapter 3]{MilnorStasheff-CharClasses})
\[
\ell\colon TM\oplus\epsilon_{M}^{\oplus(d-m)}\to\theta^{*}\gamma_{d},
\] where $TM$ is the tangent bundle on $M$, $\epsilon_{M}=\bR\times M\to M$ denotes the trivial line bundle on $M$, $\theta^{*}\gamma_{d}\to B$ is the pullback vector bundle of $\gamma_{d}$ along $\theta$ and $\oplus$ denotes the Whitney sum of vector bundles.
\end{defn}

We now introduce the moduli space of manifolds equipped with a $\theta$-structure. In order to construct stabilisation maps, we restrict attention to $\theta$-structures on $M$ satisfying a prescribed condition along a fixed submanifold of the boundary $\partial M$.
Let $G$ be a topological group. We denote by $\Top_{G}$ the category of $G$-spaces and equivariant maps. We also fix a functorial model of the total space of the universal principal $G$-bundle $EG \to \B G$. The \emph{homotopy orbit} (or Borel construction) functor $(-)_{hG} \colon \Top_{G} \to \Top$ is defined by $X \mapsto (X \times EG)/G$ where $G$ acts diagonally on $X \times EG$, i.e.~the set of the orbits of $(-)\times EG$ under the action of $G$.
We use this construction to formulate the following definition.
\begin{defn}\label{def:Bun-BDiff}
Let $\theta \colon B \to \B\Orth(d)$ be a tangential structure, let $M$ be a smooth $d$-manifold, and let $R \subset \partial M$ be a $(d-1)$--dimensional submanifold. Suppose we are given a fixed $\theta$-structure $\ell_{R} \colon TR\oplus\epsilon_{R} \to \theta^{*}\gamma_{d}$ on $R$. We define $\Bun^{\theta}(M;\ell_{R})$ to be the space of $\theta$-structures on $M$ (topologised as a subspace of $\Map(TM,\theta^{*}\gamma_d)$) which restrict to $\ell_{R}$ on $R$.

Furthermore, let $\Diff(M;R)$ denote the group of diffeomorphisms of $M$ that restrict to the identity on $R$. This group acts on $\Bun^{\theta}(M;\ell_{R})$ by $\phi \cdot \ell = \ell \circ D\phi$, where $D\phi$ is the differential of $\phi$. We then define
\[
\B\Diff^{\theta}(M;\ell_{R})
\;:=\;
\Bun^{\theta}(M;\ell_{R})_{h\Diff(M;R)}
\]
to be the homotopy orbit space of this action.
\end{defn}

\begin{rmk}
We note that the notation $\B\Diff^{\theta}(M;\ell_{R})$ is potentially misleading: in general, there is no topological group $\Diff^{\theta}(M;\ell_{R})$ whose classifying space is $\B\Diff^{\theta}(M;\ell_{R})$. Rather, this space is defined as a homotopy orbit space, and should be regarded as the moduli space of $\theta$-structured manifolds $(M,\ell)$ extending the fixed boundary condition $\ell_{R}$.
\end{rmk}

\subsection{Proof of Theorem~\ref{thm:mainC}}\label{ss:proof-of-thmC}

Now let us restrict to the case where $d=3$, $M=V$ where $(V,r,s,i,j)\in\sH$ is a compact $3$--dimensional handlebody and $R=\Image(i)\sqcup\Image(j)$ (see \S\ref{ss:handlebody_MCG_recollections}). Recall that $V\cong V^{s}_{g,r}$, for some $g\ge 0$, $r\ge 1$ and $s\ge 0$, and we write $R(V)=R^{s}_{g,r}=\cD_{r}\sqcup\cP_{s}\subset\partial V_{g}$ for brevity.

In order to construct stabilisation maps between spaces of $\theta$-structures, we follow the strategy of \cite[\S8]{Perlmutter}. Henceforth, by a \emph{framed} $m$-manifold we mean a parallelisable $m$-manifold $M$ equipped with a specified trivialisation of its tangent bundle, i.e.~a vector bundle isomorphism $\epsilon_{M}^{\oplus m}\cong TM$. We now introduce the following definition.

\emph{We fix once and for all of \S\ref{ss:proof-of-thmC} a bundle map $\ell_{0}\colon\bR^{3}\to\theta^{*}\gamma_{3}$, where $\bR^{3}$ is regarded as the trivial rank--$3$ vector bundle over a point.}

\begin{defn}\label{def:canonical-theta-structure-framing}
Let $M$ be a framed $m$-manifold with $m \leq 3$ for $m\le 3$. Using the standard inclusion $\bR^{m} \hookrightarrow \bR^{3}$ (for $m \leq 3$), the specified trivialisation $\epsilon_{M}^{\oplus m} \xrightarrow{\cong} TM$ identifies $TM$ with a trivial subbundle of $M \times \bR^{3}$. The map $\ell_{0}$ then induces a canonical $\theta$-structure $\ell_{0}(M)\colon TM \times \bR^{3-m} \cong M\times\bR^{3}\to\theta^{*}\gamma_{3}$ by $(x,y)\mapsto \ell_{0}(y)$.
\end{defn}

Let us choose framings on $V_{1,1}$ and $V_{0,2}$ once and for all, which determine canonical $\theta$-structures $\ell_{0}(V_{1,1})$ and $\ell_{0}(V_{0,2})$ as in Definition~\ref{def:canonical-theta-structure-framing}. We may then define a stabilisation map
\begin{equation}\label{eq:bundle-stabilisation-sigma}
    \sigma^{\theta,s}_{g,r}\colon
    \Bun^{\theta}(V_{g};\ell_{0}(R^{s}_{g,r}))
    \longrightarrow
    \Bun^{\theta}(V_{g+1};\ell_{0}(R^{s}_{g+1,r}))
\end{equation}
by $\ell \mapsto \ell\natural\ell_{0}(V_{1,1})$, where the operation $\natural$ is induced from the monoidal structure $\natural_{1}$ of Definition~\ref{def:monoidal-structure-cG} on the source. Similarly, we define a stabilisation map
\begin{equation}\label{eq:bundle-stabilisation-mu}
    \mu_{g,r}^{\theta,s}\colon 
    \Bun^{\theta}(V_{g};\ell_{0}(R^{s}_{g,r}))
    \longrightarrow 
    \Bun^{\theta}(V_{g};\ell_{0}(R^{s}_{g,r+1}))
\end{equation}
by $\ell\mapsto \ell\bar{\natural}\ell_{0}(V_{0,3})$, where the operation $\bar{\natural}$ is induced from that on the source introduced in \S\ref{ss:HS-marked-discs}. Both maps are $\Diff(V_{g};R^{s}_{g,r})$-equivariant, where the action on the targets is defined via the standard stabilisation maps $\Diff(V_{g};R^{s}_{g,r})\to \Diff(V_{g+1};R^{s}_{g+1,r})$ and $\Diff(V_{g};R^{s}_{g,r})\to \Diff(V_{g};R^{s}_{g,r+1})$ respectively. Hence applying the homotopy orbit functor therefore yields induced stabilisation maps
\begin{equation}
    \breve{\sigma}_{g,r}^{\theta,s}\colon
    \B\Diff^{\theta}(V_{g};\ell_{0}(R^{s}_{g,r}))
    \longrightarrow
    \B\Diff^{\theta}(V_{g+1};\ell_{0}(R^{s}_{g+1,r}))
\end{equation}
and
\begin{equation}
    \breve{\mu}_{g,r}^{\theta,s}\colon
    \B\Diff^{\theta}(V_{g};\ell_{0}(R^{s}_{g,r}))
    \longrightarrow
    \B\Diff^{\theta}(V_{g};\ell_{0}(R^{s}_{g,r+1})).
\end{equation}We now recall the statement of Theorem~\ref{thm:mainC}.

\begin{thm}\label{thmC-detailed}
If $\theta \colon B \to \B\Orth(3)$ is a simply connected tangential structure, then the stabilisation maps $\breve{\sigma}_{g,r}^{\theta,s}$ and $\breve{\mu}_{g,r}^{\theta,s}$ induce isomorphisms on homology with rational coefficients in all degrees $*\le \frac{g-3}{2}$. 
\end{thm}
The proof follows the strategy of \cite[Thm.~4.2(i)]{RWautfreegroups}, with Proposition~\ref{prop:tangential-str-bisystem} below providing the key additional input. We begin by outlining the main idea, in order to place the subsequent lemmas in context. First, by \cite[Lem.~2.1]{AnnexLS1}, applying the homotopy orbit construction to the trivial map $\Bun^{\theta}(V_{g};\ell_{0}(R^{s}_{g,r}))\to *$ yields a fibre bundle
\begin{equation}\label{eq:tang-str-fibre-bundle}
    \B\Diff^{\theta}(V_{g};\ell_{0}(R^{s}_{g,r}))\longrightarrow \B\Diff(V_{g};R^{s}_{g,r}),
\end{equation}
whose fibre is $\Bun^{\theta}(V_{g};\ell_{0}(R^{s}_{g,r}))$. Moreover, there is a commutative diagram
\begin{equation}\label{eq:projection-map-sigma}
\begin{tikzcd}
    \Bun^{\theta}(V_{g};\ell_{0}(R^{s}_{g,r})) \times E\Diff(V_{g};R^{s}_{g,r})
    \arrow[r] \arrow[d]
    &
    \Bun^{\theta}(V_{g+1};\ell_{0}(R^{s}_{g+1,r})) \times E\Diff(V_{g+1};R^{s}_{g+1,r})
    \arrow[d] \\
    E\Diff(V_{g};R^{s}_{g,r})
    \arrow[r]
    &
    E\Diff(V_{g+1};R^{s}_{g+1,r}),
\end{tikzcd}
\end{equation}
where the vertical maps are the projections, while the horizontal maps are defined by the stabilisation map $\sigma^{\theta,s}_{g,r}$ (see \eqref{eq:bundle-stabilisation-sigma}) and by the map of total spaces of universal bundles induced the usual stabilisation map $\Diff(V_{g};R^{s}_{g,r})\to \Diff(V_{g+1};R^{s}_{g,r})$ defined by extending diffeomorphisms by the identity on the complementary summand $V_{1}$. All the maps in the diagram are obviously $\Diff(V_{g};R^{s}_{g,r})$-equivariant, and hence pass to homotopy orbits. Taking the $\Diff(V_{g+1};R^{s}_{g+1,r})$-orbit in the targets of the horizontal maps, the induced horizontal maps are precisely the stabilisation maps $\breve{\sigma}_{g,r}^{\theta,s}$ and $\breve{\sigma}^{s}_{g,r}$. We therefore obtain a map of fibrations
\begin{equation}\label{eq:fibration-map-sigma}
\begin{tikzcd}
    \B\Diff^{\theta}(V_{g};\ell_{0}(R^{s}_{g,r}))
    \arrow[r,"\breve{\sigma}_{g,r}^{\theta,s}"]
    \arrow[d]
    &
    \B\Diff^{\theta}(V_{g+1};\ell_{0}(R^{s}_{g+1,r}))
    \arrow[d] \\
    \B\Diff(V_{g};R^{s}_{g,r})
    \arrow[r,"\breve{\sigma}^{s}_{g,r}"]
    &
    \B\Diff(V_{g+1};R^{s}_{g+1,r}),
\end{tikzcd}
\end{equation}
whose restriction to the fibres is the map is the map $\sigma^{\theta,s}_{g,r}\colon \Bun^{\theta}(V_{g};\ell_{0}(R^{s}_{g,r})) \to \Bun^{\theta}(V_{g+1};\ell_{0}(R^{s}_{g+1,r}))$. Similarly, the stabilisation maps $\mu_{g,r}^{\theta,s}$ and $\breve{\mu}^{s}_{g,r}$ fit into a map of fibrations
\begin{equation}
\begin{tikzcd}\label{eq:fibration-map-mu}
    \B\Diff^{\theta}(V_{g};\ell_{0}(R^{s}_{g,r}))
    \arrow[r,"\breve{\mu}_{g,r}^{\theta,s}"]
    \arrow[d]
    &
    \B\Diff^{\theta}(V_{g};\ell_{0}(R^{s}_{g,r+1}))
    \arrow[d] \\
    \B\Diff(V_{g};R^{s}_{g,r})
    \arrow[r,"\breve{\mu}^{s}_{g,r}"]
    &
    \B\Diff(V_{g};R^{s}_{g,r+1}),
\end{tikzcd}
\end{equation}
where the restriction to the fibres is the map $\mu_{g,r}^{\theta,s}\colon
\Bun^{\theta}(V_{g};\ell_{0}(R^{s}_{g,r})) \to \Bun^{\theta}(V_{g};\ell_{0}(R^{s}_{g,r+1}))$. To prove that the maps on total spaces in \eqref{eq:fibration-map-sigma} and \eqref{eq:fibration-map-mu} induce rational homology isomorphisms in the range of Theorem~\ref{thmC-detailed}, we analyse the associated Serre spectral sequences. The argument reduces to twisted homological stability for the base spaces, with coefficients given by the homology of the fibres.
More precisely, for each $q \ge 0$ we show that the homology groups $H_{q}(\Bun^{\theta}(V_{g};\ell_{0}(R^{s}_{g,r}));\bQ)$ assemble into the components of a functor $F_{q}^{\theta}\colon \sH\to\Ab$, and that this functor is a coefficient bisystem in the sense of Definition~\ref{def:double-coeff-system}; see Proposition~\ref{prop:tangential-str-bisystem}. This allows us to apply Theorems~\ref{thm:HS_stabilistation-1} and \ref{cor:independence-of-marked-discs}, thereby deducing the required twisted homological stability.

\subsubsection{Connectivity of the fibre}
A crucial ingredient in the proof is the following lemma, which is the only point at which the assumption that $\theta$ is simply connected is used.
\begin{lem}\label{lem:sc-tang-str-connected-space}
If the tangential structure $\theta\colon B \to B\Orth(3)$ is simply connected, then the space $\Bun^{\theta}(V_{g};\ell_{0}(R^{s}_{g,r}))$ is path-connected.
\end{lem}
\begin{proof}
Let $\bar{\ell}_{1},\bar{\ell}_{2}\in \Bun^{\theta}(V_{g};\ell_{0}(R^{s}_{g,r}))$ be two $\theta$-structures on $V_{g}$. Write $\ell_{1},\ell_{2}\colon V_{g}\to B$ for the underlying maps on the base, obtained by restricting $\bar{\ell}_{i}$ along the zero section of $TV_{g}\to V_{g}$ and then followed by postcomposition with the projection $\theta^{*}\gamma_{3}\to B$. We construct a bundle homotopy (i.e.\ a homotopy through bundle maps) from $\bar{\ell}_{1}$ to $\bar{\ell}_{2}$, thereby establishing the path-connectivity of $\Bun^{\theta}(V_{g};\ell_{0}(R^{s}_{g,r}))$.

Let $\bar{\theta}\colon \theta^{*}\gamma_{3}\to \gamma_{3}$ be the bundle map covering $\theta$, i.e.~the canonical map induced by the pullback of $\theta$ along the projection $\gamma_{3}\to\B\Orth(3)$ defining $\theta^{*}\gamma_{3}$. By \cite[Thm.~5.7]{MilnorStasheff-CharClasses}, the bundle maps $\bar{\theta}\circ \bar{\ell}_{1}$ and $\bar{\theta}\circ \bar{\ell}_{2}$ are bundle homotopic. Let $H'\colon TV_{g}\times I\to\gamma_{3}$ be such a bundle homotopy and let $H\colon V_{g}\times I\to \B\Orth(3)$ denote the induced homotopy on the bases. Thus $H$ is a homotopy from $\theta\circ \ell_{1}$ to $\theta\circ \ell_{2}$.

We first prove that $H$ admits a relative lift, as indicated by the dashed arrow in the diagram
\[
\begin{tikzcd}
    (\{0,1\}\times V_{g})\cup_{\{0,1\}\times R^{s}_{g,r}} (I\times R^{s}_{g,r})
    \arrow[r] \arrow[d,hook]
    &
    B \arrow[d,"\theta"] \\
    V_{g}\times I
    \arrow[r,"H"]
    \arrow[ur,dashed,"\tilde{H}"]
    &
    B\Orth(3),
\end{tikzcd}
\]
where the top map is given by $\ell_{1}$ on $\{0\}\times V_{g}$, by $\ell_{2}$ on $\{1\}\times V_{g}$, and by $\ell_{0}(R^{s}_{g,r})$ on $I\times R^{s}_{g,r}$.
Let $F$ denote the fibre of $\theta$. By obstruction theory (see e.g.~\cite[Prop.~9.2.3]{Arkowitz-HomotopyTheory}), such a relative lift exists provided that
\[
H^{j+1}\left(V_{g} \times I,
(\{0,1\}\times V_{g})\cup_{\{0,1\}\times R^{s}_{g,r}} (I\times R^{s}_{g,r});
\pi_{j}(F)\right)=0
\]
for all $0\le j\le 4$. Since $\theta$ is simply connected, we have $\pi_{0}(F)=\pi_{1}(F)=0$, so the above relative cohomology groups vanish for $j\le 1$. Furthermore, we note that $V_{g}$ deformation retracts onto a wedge of $g$ circles and $I$ is contractible, hence $V_{g}\times I$ has the homotopy type of a $1$--dimensional CW complex. Likewise, each component of $I\times R^{s}_{g,r}$ deformation retracts onto an interval, and each of $\{0\}\times V_{g}$ and $\{1\}\times V_{g}$ deformation retracts onto a wedge of $g$ circles. Gluing these deformation retractions compatibly along $\{0,1\}\times R^{s}_{g,r}$ shows that
$(\{0,1\}\times V_{g})\cup_{\{0,1\}\times R^{s}_{g,r}} (I\times R^{s}_{g,r})$ has the homotopy type of a graph obtained from two wedges of $g$ circles by attaching finitely many intervals between specified vertices, and thus also has the homotopy type of a $1$--dimensional CW complex. We deduce that $H^{j+1}(V_{g} \times I;\pi_{j}(F))= H^{j+1}((\{0,1\}\times V_{g})\cup_{\{0,1\}\times R^{s}_{g,r}} (I\times R^{s}_{g,r});\pi_{j}(F))= 0$ for all $j\ge 1$. 
The long exact sequence in relative cohomology for the pair then implies that $H^{j+1}(V_{g} \times I,(\{0,1\}\times V_{g})\cup_{\{0,1\}\times R^{s}_{g,r}} (I\times R^{s}_{g,r});\pi_{j}(F))=0$ for all $j\ge 2$. Therefore the desired relative lift exists.

Now, we now consider for $i\in\{1,2\}$ the pullback square defining $\theta^{*}\gamma_{3}$ and the resulting commutative diagram
\[
\begin{tikzcd}
    TV_{g} \arrow[r,"\bar{\ell}_{i}"] \arrow[d,"\pi"]
    &
    \theta^{*}\gamma_{3}
    \arrow[dr, phantom, "\usebox\pullback", very near start, color=black]
    \arrow[d]
    \arrow[r,"\bar{\theta}"]
    &
    \gamma_{3} \arrow[d] \\
    V_{g} \arrow[r,"\ell_{i}"]
    &
    B \arrow[r,"\theta"]
    &
    B\Orth(3).
\end{tikzcd}
\]
By the existence of the relative lift $\tilde{H}$, we obtain a commutative diagram
\[
\begin{tikzcd}
    TV_{g}\times I
    \arrow[rr,bend left,"H'"]
    \arrow[r,dashed,"\tilde{H}'"]
    \arrow[rd,"\tilde{H}\circ(\pi\times\id_{I})"',bend right]
    &
    \theta^{*}\gamma_{3}
    \arrow[dr, phantom, "\usebox\pullback", very near start, color=black]
    \arrow[d]
    \arrow[r,"\bar{\theta}"]
    &
    \gamma_{3} \arrow[d] \\
    &
    B \arrow[r,"\theta"]
    &
    B\Orth(3),
\end{tikzcd}
\]
where $\tilde{H}'\colon TV_{g}\times I\to \theta^{*}\gamma_{3}$ is the unique map provided by the universal property of the pullback.
By construction, the bundle $\theta^{*}\gamma_{3} \to B$ has the same fibre as the tautological bundle $\gamma_{3}\to B\Orth(3)$, and the action of $\tilde{H}'$ on each fibre is determined by $H'$. Then, since $H'$ is a bundle homotopy, it follows by a straightforward verification that $\tilde{H}'$ is also a bundle homotopy. Moreover, since $\tilde{H}$ is a lift relative to $R^{s}_{g,r}$ and $H'$ is a bundle homotopy which restricts to $\theta\circ \ell_{0}(R^{s}_{g,r})$ on $TR^{s}_{g,r}$, we deduce that $\tilde{H}'$ is a homotopy through bundle maps in $\Bun^{\theta}(V_{g};\ell_{0}(R^{s}_{g,r}))$.

Finally, by uniqueness of $\tilde{H}'$ implied by the universal property of pullbacks, we have $\tilde{H}'(0,-)=\bar{\ell}_{1}$ and $\tilde{H}'(1,-)=\bar{\ell}_{2}$. Hence $\tilde{H}'$ is a path from $\bar{\ell}_{1}$ to $\bar{\ell}_{2}$ in $\Bun^{\theta}(V_{g};\ell_{0}(R^{s}_{g,r}))$.
\end{proof}

\subsubsection{A key coefficient bisystem}

Let $\theta \colon B \to \B\Orth(3)$ be an arbitrary tangential structure. For brevity, we continue to write $\ell_{0}$ for the prescribed boundary condition, whenever it is clear from the context. For each $q \ge 0$, we now introduce a functor $F_{q}^{\theta} \colon \sH \to \Ab$, defined as follows.
We begin by defining a functor $\Bun^{\theta}(-;\ell_{0})\colon \sH\to\hTop$, where $\hTop$ denotes the (naive) homotopy category of spaces (i.e.~whose morphisms are homotopy classes of maps). On objects, we set $\Bun^{\theta}(-;\ell_{0})(V) := \Bun^{\theta}(V;\ell_{0})$.
On morphisms, if $\phi \colon V \to W$ is an isomorphism in $\sH$, and $\tilde{\phi}$ is a diffeomorphism representing $\phi$, we define $\Bun^{\theta}(\phi;\ell_{0}) \colon \Bun^{\theta}(V;\ell_{0}) \to \Bun^{\theta}(W;\ell_{0})$ to be the homotopy class of the map induced by precomposition with $D\tilde{\phi}^{-1}$. This is well-defined at the level of homotopy classes, and hence determines a functor to $\hTop$. We then define the functor $F_{q}^{\theta}$ by postcomposing with the homology functor $H_{q}(-;\bQ)\colon \hTop \to \Ab$ (see \cite[\S2.3]{Hatcherbook} for instance):
\[
F_{q}^{\theta} := H_{q}(-;\bQ) \circ \Bun^{\theta}(-;\ell_{0})
\]

\begin{prop}\label{prop:tangential-str-bisystem}
Let $\theta \colon B \to \B\Orth(3)$ be a simply connected tangential structure. 
For all $r \ge 2$ and $s \ge 0$, and for each $q \ge 0$, the functor $F_{q}^{\theta} \colon \sH \to \Ab$ is a coefficient bisystem at $V^{s}_{0,r}$ of degree $q$.
\end{prop}

Following Definition~\ref{def:double-coeff-system}, the key input for Proposition~\ref{prop:tangential-str-bisystem} is a precise understanding of the kernels and cokernels (in homology) of the stabilisation maps. For $V \in \sH^{\ge 1}$ we consider the stabilisation map
\begin{equation}
    \sigma^{\theta}_{V}\colon
    \Bun^{\theta}(V;\ell_{0})
    \longrightarrow
    \Bun^{\theta}(V \natural V_{1,1};\ell_{0}),
\end{equation}
defined by $\ell \mapsto \ell \natural \ell_{0}(V_{1,1})$, and for $W \in \sH^{\ge 2}$ we consider the stabilisation map
\begin{equation}
    \rho^{\theta}_{W}\colon
    \Bun^{\theta}(W;\ell_{0})
    \longrightarrow
    \Bun^{\theta}(W \# V_{0,2};\ell_{0}),
\end{equation}
defined by $\ell \mapsto \ell \# \ell_{0}(V_{0,2})$. Here the operation $\#$ is induced by the monoidal structure $\#=\natural_{2}$ of Definition~\ref{def:monoidal-structure-cG}.
It straightforwardly follows from the definition of the spaces of $\theta$-structures (see Definition~\ref{def:Bun-BDiff}), the monoidal product $\natural$ (Definition~\ref{def:monoidal-structure-cG}) induces a continuous map
\begin{equation}\label{diag:homotopy-eq-natural}
\Bun^{\theta}(V;\ell_{0}) \times \Bun^{\theta}(V_{1,1};\ell_{0})
\;\xrightarrow{-\natural-}\;
\Bun^{\theta}(V \natural V_{1,1};\ell_{0}),
\end{equation}
and similarly the product $\#$ induces a continuous map
\begin{equation}\label{diag:homotopy-eq-sharp}
\Bun^{\theta}(W;\ell_{0}) \times \Bun^{\theta}(V_{0,2};\ell_{0})
\;\xrightarrow{-\#-}\;
\Bun^{\theta}(W \# V_{0,2};\ell_{0}).
\end{equation}
We also generically write $-\times\ell_{0}$ for the canonical inclusion into the left factor of a product, where $\ell_{0}$ denotes the unique map from the initial object $\emptyset$ to any space of $\theta$-structures, whose image is the $\theta$-structure induced by the fixed bundle map $\ell_{0}$ (see Definition~\ref{def:canonical-theta-structure-framing}). With this convention, the above maps are compatible with stabilisation in the sense that the diagrams
\begin{equation}
\begin{tikzcd}
    \Bun^{\theta}(V;\ell_{0})
    \arrow[r,hook,"-\times\ell_{0}"]
    \arrow[rd,"\sigma^{\theta}_{V}"']
    &
    \Bun^{\theta}(V;\ell_{0}) \times \Bun^{\theta}(V_{1,1};\ell_{0})
    \arrow[d,"-\natural-"]
    \\
    &
    \Bun^{\theta}(V \natural V_{1,1};\ell_{0})
\end{tikzcd}
\end{equation}
and
\begin{equation}
\begin{tikzcd}
    \Bun^{\theta}(W;\ell_{0})
    \arrow[r,hook,"-\times\ell_{0}"]
    \arrow[rd,"\rho^{\theta}_{W}"']
    &
    \Bun^{\theta}(W;\ell_{0}) \times \Bun^{\theta}(V_{0,2};\ell_{0})
    \arrow[d,"-\#-"]
    \\
    &
    \Bun^{\theta}(W \# V_{0,2};\ell_{0})
\end{tikzcd}
\end{equation}
commute. To analyse the kernels and cokernels of the stabilisation maps on homology, we use the following lemma.
\begin{lem}\label{lem:monoidal-str-homotopy-eq}
For all $V \in \sH^{\ge 1}$ and $W \in \sH^{\ge 2}$, 
the maps~\eqref{diag:homotopy-eq-natural} and~\eqref{diag:homotopy-eq-sharp} 
are homotopy equivalences.
\end{lem}

\begin{proof}
We begin with treating the case of \eqref{diag:homotopy-eq-natural} in detail.  Let $D$ be the closed half-disc along which $V$ and $V_{1,1}$ are glued to form $V\natural V_{1,1}$. We choose a neighbourhood $\cN_{D}\subset V\natural V_{1,1}$ of $D$, along with a homeomorphism $\phi\colon [-1,1]\times D \overset{\cong}{\longrightarrow} \cN_{D}$, whose restriction to $\{0\}\times D$ identifies $\{0\}\times D$ with $D\subset \cN_{D}$ via the identity. Using $\phi$, we henceforth identify $\cN_{D}$ with $[-1,1]\times D$. We also identify $D$ with the right half of the unit disc in $\bR^{2}$, with coordinates $(x,y)$. Now, for each $t\in I=[0,1]$, we define a self-map $f_{t}\colon \cN_{D}\to \cN_{D}$ by
\[
f_{t}(s,(x,y))
=
(s,((1+(|s|-1)t)\,x,\;y)).
\]
Then $f_{0}=\id_{\cN_{D}}$ and $f_{t}$ restricts to the identity on $\{\pm 1\}\times D$ for all $t\in I$. Extending $f_{t}$ by the identity on $(V\natural V_{1,1})\setminus \cN_{D}$ and taking the cartesian product with $\id_{\bR^{3}}$ yields a self-map $\tilde{f}_{t}\colon (V\natural V_{1,1})\times\bR^{3}\to (V\natural V_{1,1})\times\bR^{3}$, which is a bundle map over $V\natural V_{1,1}$. We also have a bundle homotopy $H\colon I\times (V\natural V_{1,1})\times\bR^3 \to (V\natural V_{1,1})\times\bR^{3}$ between $\id_{(V\natural V_{1,1})\times\bR^{3}}$ and $\tilde{f}_{1}$, defined by $H(t,-)=\tilde{f}_t(-)$. Then, the precomposition with $H$ induces a homotopy
\[
K\colon I\times \Bun^{\theta}(V\natural V_{1,1};\ell_{0}(R(V\natural V_{1,1})))
\longrightarrow
\Bun^{\theta}(V\natural V_{1,1};\ell_{0}(R(V\natural V_{1,1}))),
\]
defined by $(t,\ell)\mapsto \ell\circ H(t,-)$.

Now, we consider the inclusion
\[
\iota\colon
\Bun^{\theta}(V\natural V_{1,1};\ell_{0}(R(V)\cup R(V_{1,1})))
\hookrightarrow
\Bun^{\theta}(V\natural V_{1,1};\ell_{0}(R(V\natural V_{1,1})))
\]
induced by the inclusion $R(V\natural V_{1,1})\hookrightarrow R(V)\cup R(V_{1,1})$. By construction, $K$ is a homotopy between the identity and $\iota\circ(-\circ \tilde{f}_{1})$. Moreover, if $\ell$ lies in the subspace
$\Bun^{\theta}(V\natural V_{1,1};\ell_{0}(R(V)\cup R(V_{1,1})))$, then $\ell\circ \tilde{f}_{t}$ lies in the same subspace for all $t\in I$. Hence, the restriction of $K$ along $\iota$ yields a homotopy from $\id$ to $(-\circ \tilde{f}_{1})\circ \iota$ (i.e.~$(t,\ell')\mapsto \ell'\circ \tilde{H}(t,-)$ for $(t,\ell')\in I\times \Bun^{\theta}(V\natural V_{1,1};\ell_{0}(R(V)\cup R(V_{1,1})))$) yields a homotopy from the identity map on $\Bun^{\theta}(V\natural V_{1,1};\ell_{0}(R(V)\cup R(V_{1,1})))$ to $(-\circ \tilde{f}_{1})\circ \iota$. It follows that $\iota$ and $(-\circ \tilde{f}_{1})$ are homotopy inverses.

On the other hand, the map \eqref{diag:homotopy-eq-natural} lands in the subspace $\Bun^{\theta}(V\natural V_{1,1};\ell_{0}(R(V)\cup R(V_{1,1})))$, and the restrictions along the canonical embeddings $V_{1,1}\hookrightarrow V\natural V_{1,1}$ and $V\hookrightarrow V\natural V_{1,1}$ induce a map
\[
\Bun^{\theta}(V\natural V_{1,1};\ell_{0}(R(V)\cup R(V_{1,1})))
\longrightarrow
\Bun^{\theta}(V;\ell_{0})\times \Bun^{\theta}(V_{1,1};\ell_{0}),
\]
which is inverse to $-\natural-$. This shows that $-\natural -$ is a homeomorphism onto its image. Since $-\natural-$ is, by definition, the composite of this homeomorphism with the homotopy equivalence $\iota$, we conclude that \eqref{diag:homotopy-eq-natural} is a homotopy equivalence.

The argument for~\eqref{diag:homotopy-eq-sharp} is identical: one replaces $\cN_{D}$ by the disjoint union of neighbourhoods of the two half-discs used to form $W\# V_{0,2}$, defines the corresponding isotopy $f_{t}$ on each neighbourhood, and repeats the above proof verbatim.
\end{proof}

We can now combine Lemmas~\ref{lem:sc-tang-str-connected-space} and~\ref{lem:monoidal-str-homotopy-eq} to prove Proposition~\ref{prop:tangential-str-bisystem}.

\begin{proof}[Proof of Proposition~\ref{prop:tangential-str-bisystem}]    
Let us verify that the functor $F_q^{\theta}$ satisfies all of the hypotheses of Definition~\ref{def:double-coeff-system}. For brevity, we write $A=V^{s}_{0,r}$, $X=V_{1,1}$.

We start by checking that the restriction of $F_{q}^{\theta}$ to $\sH_{A,X}^{\ge 1}$ extends to a functor $F_{q}^{\theta} \colon [\sH_{A,X}^{\ge 1},\sH_{X}^{\ge 1}\rangle\to\Ab$ of degree $d$ at $0$. We show this by first extending the functor $\Bun^{\theta}(-;\ell_{0})\colon \sH_{A,X}^{\ge 1}\to \hTop$, which was defined prior to Proposition~\ref{prop:tangential-str-bisystem}. Given a representative $(\phi,V_{h})$ of a morphism $V\to W$ in $[\sH_{A,X}^{\ge 1},\sH_{X}^{\ge 1}\rangle$, and a diffeomorphism $\tilde{\phi}$ representing $\phi$, we obtain a (well-defined homotopy class of) map $\Bun^{\theta}(-;\ell_{0})(\phi):\Bun^{\theta}(V\natural V_{h};\ell_{0})\to \Bun^{\theta}(W;\ell_{0})$. Furthermore, using the stabilisation map
\[
\tilde{\sigma}^{h}_{V}\colon
\Bun^{\theta}(V;\ell_{0})
\longrightarrow
\Bun^{\theta}(V\natural V_{h};\ell_{0})
\]
defined by $\ell\mapsto \ell\natural\ell_{0}(V_{h})$, and we assign 
\[
\Bun^{\theta}(-;\ell_{0})(\phi,V_{h})
:=
\Bun^{\theta}(-;\ell_{0})(\phi)\circ \tilde{\sigma}^{h}_{V}.
\]
For two different representatives $(\phi,V_{h})$ and $(\phi',V_{h})$ of the same class of morphisms, we now need to show that $\Bun^{\theta}(-;\ell_{0})(\phi,V_{h})$ and $\Bun^{\theta}(-;\ell_{0})(\phi',V_{h})$ belong to the same homotopy class of maps. By definition, there exists an isomorphism $f:V_{h}\to V_{h}$ in $\sH_{A,X}^{\ge 1}$ such that $\phi'=\phi\circ(\id_{V}\natural f)$. It therefore suffices to show that
\[
\Bun^{\theta}(-;\ell_{0})(\id_V\natural f)\circ \tilde{\sigma}^{h}_{V} \colon \Bun^{\theta}(V;\ell_{0})\longrightarrow \Bun^{\theta}(V\natural V_{h};\ell_{0})
\]
is homotopic to $\tilde{\sigma}^{h}_{V}$.
belongs to the same homotopy class as $\tilde{\sigma}_{V}^{h}$. We have a commutative diagram
\[
\begin{tikzcd}
    \Bun^{\theta}(V;\ell_{0}) \arrow[rrd,"\tilde{\sigma}_{V}^{h}"'] \arrow[rr,"-\times\ell_{0}(V_{h})"]
    && \Bun^{\theta}(V;\ell_{0})\times\Bun^{\theta}(V_{h};\ell_{0}) \arrow[d,"-\natural-"] \arrow[rrr,"\id\times\Bun^{\theta}(-;\ell_{0})(f)"]
    &&& \Bun^{\theta}(V;\ell_{0})\times\Bun^{\theta}(V_{h};\ell_{0}) \arrow[d,"-\natural-"]
    \\
    && \Bun^{\theta}(V\natural V_{h};\ell_{0})\arrow[rrr,"\Bun^{\theta}(-;\ell_{0})(\id_{V}\natural f)"]
    &&& \Bun^{\theta}(V\natural V_{h};\ell_{0}).
\end{tikzcd}
\]
The vertical maps are homotopy equivalences by Lemma~\ref{lem:monoidal-str-homotopy-eq}. Hence it is enough to check that the constant map $*\to\Bun^{\theta}(V_{h};\ell_{0})$ with value $\ell_{0}(V_h)$ is homotopic to its postcomposition with $\Bun^{\theta}(-;\ell_{0})(f)$, or in other words that the two bundle maps $TV_{h}\to\theta^{*}\gamma_{3}$ are bundle homotopic. This follows from Lemma~\ref{lem:sc-tang-str-connected-space}, which implies that $\Bun^{\theta}(V_h;\ell_{0})$ is connected. Thus the extension is well-defined up to homotopy, and so we have a well-defined functor $\Bun^{\theta}(-;\ell_{0}) \colon [\sH_{A,X}^{\ge 1},\sH_{X}^{\ge 1}\rangle\to \hTop$. Then, composing with the functor $H_{1}(-;\bQ)\colon \hTop\to\Ab$ gives the desired functor $F_{q}^{\theta}$.

Now, we prove by induction on $q\ge 0$ that the functor $F_{q}^{\theta}$ is a finite degree coefficient system $[\sH_{A,X}^{\ge 1},\sH_{X}^{\ge 1}\rangle\to \Ab$ of degree $q$. Since $\theta$ is simply connected, Lemma~\ref{lem:sc-tang-str-connected-space} shows that the space of $\theta$-structures $\Bun^{\theta}(V;\ell_{0})$ is connected. Hence the functor $F_{0}^{\theta}$ is constant with value $\bQ$, and so a coefficient system of degree $0$; see Example~\ref{eg:Z_H_poly_functor}.
Thus the functor $F_{0}^{\theta}$ is constant with value $\bQ$ and so indeed a coefficient system of degree $0$; see Example~\ref{eg:Z_H_poly_functor}.

Let us now assume the claim holds for $F_{k}$ for all $k < q$.
We note that for any $V\in\sH^{\ge 1}_{A,X}$, the stabilisation map $\sigma^{\theta}_{V}\colon \Bun^{\theta}(V;\ell_{0})\to \Bun^{\theta}(V\natural X;\ell_{0})$ and the functor $F_{q}^{\theta}$ are defined precisely so that $H_{q}(\sigma^{\theta}_{V};\bQ)=F_{q}^{\theta}(\iota_{X}(V))$, where we recall that $\iota_{X}(V)$ is the map $[\id,X]\colon V\to V\natural X$. Furthermore, we have a commutative diagram
\begin{equation}\label{diag:natural-stabilisation-via-monoidal-str}
\begin{tikzcd}
    \Bun^{\theta}(V;\ell_{0})\arrow[rd,"\sigma^{\theta}_{V}"']\arrow[r,hook,"-\times\ell_{0}"] &
    \Bun^{\theta}(V;\ell_{0})\times\Bun^{\theta}(V_{1,1};\ell_{0})\arrow[d,"-\natural-"] \\
    & \Bun^{\theta}(V\natural V_{1,1};\ell_{0})
    \end{tikzcd}
\end{equation}
where the vertical map is a homotopy equivalence by Lemma~\ref{lem:monoidal-str-homotopy-eq}. Applying the Künneth theorem for topological spaces (see \cite[Cor.~VI.12.10, Prop.~VII.2.6]{Dold}) yields a natural commutative diagram
\[
\begin{tikzcd}
    F_{q}^{\theta}(V)\arrow[r]\arrow[rd,"F_{q}^{\theta}(\iota_{X}(V))"']&\underset{i+j=q}{\bigoplus} F_{i}(V)\otimes F_{j}(X) \arrow[d,"\cong"]\\
    & F_{q}^{\theta}(V\natural X),
\end{tikzcd}
\]
where the vertical arrow is the Künneth isomorphism and the horizontal arrow is the direct summand inclusion. We deduce that $\kappa_{X}(F_{q}^{\theta})=0$ and we have a natural isomorphism
\begin{equation}\label{iso:Kunneth-delta}
\delta_{X}(F_{q}^{\theta})(V)
\cong
\bigoplus_{i=0}^{q-1}F_{i}(V)\otimes F_{q-i}(X).
\end{equation}
For each $1\le j\le q$, we note that the functor $-\otimes F_{j}(X)\colon \Ab\to\Ab$ is exact, because the functor $F_{j}$ takes values in $\bQ$-vector spaces. By the inductive assumption, it is then immediate that the coefficient system $F_{i}(-)\otimes F_{q-i}(X)$ is of finite degree at most $i$ for each $0\le i\le q-1$. Since finite degree coefficient systems are closed under direct sums (see \cite[Lem.~1.6]{AnnexLS1}), it follows from \eqref{iso:Kunneth-delta} that the coefficient system $\delta_{X}(F_{q}^{\theta})$ is of finite degree at most $q-1$, and thereby finally that $F_{q}^{\theta}$ is a finite degree coefficient system of degree $q$. Moreover, by repeating mutatis mutandis the above proof by stabilising on the left instead of on the right, we also check that the functor $F_{q}^{\theta}$ defines a coefficient system $\langle \sH^{\ge 1}_{V_{1,1}}, \sH^{\ge 1}_{V,V_{1,1}} ]  \to \Ab$ of degree $q$, i.e.~it satisfies Condition~\eqref{item:extension-left} of Definition~\ref{def:double-coeff-system}.

The verification of Condition~\eqref{item:extension-right} in Definition~\ref{def:double-coeff-system} is entirely analogous to the argument for Condition~\eqref{item:extension-left}. We therefore sketch the construction and highlight only the points where the details differ. First of all, let $W\in\sH^{\ge 2}$ and consider the restriction of $F_{q}^{\theta}$ to $\sH_{W,Y}^{\ge 2}$. We extend it to $[\sH_{W,Y}^{\ge 2},\sH_{Y}^{\ge 2}\rangle$ by assigning $F_{q}^{\theta}(\phi,Y^{\# h}):=H_{q}(F(\phi)\circ \rho^{h}_{W};\bZ)$, where 
\[
\tilde{\rho}^{h}_{W}:\Bun^{\theta}(W;\ell_{0})\longrightarrow \Bun^{\theta}(W\# Y^{\# h};\ell_{0})
\]
is defined by assigning $\ell\mapsto \ell\natural\ell_{0}(V_{h})$. As before, we prove by induction on $q\ge 0$ that $F_{q}^{\theta}$ is a coefficient system of finite degree $q$ over $[\sH_{W,Y}^{\ge 2},\sH_{Y}^{\ge 2}\rangle$.
The case $q=0$ again follows from Lemma~\ref{lem:sc-tang-str-connected-space}. For the inductive step, we now use the commutative diagram
\begin{equation*}
\begin{tikzcd}
    \Bun^{\theta}(W;\ell_{0}) \arrow[rd,"\rho^{\theta}_{W}"'] \arrow[r,hook,"-\times\ell_{0}"]
    & \Bun^{\theta}(W;\ell_{0})\times\Bun^{\theta}(V_{0,2};\ell_{0}) \arrow[d,"-\#-"]
    \\
    & \Bun^{\theta}(W\# V_{0,2};\ell_{0})
\end{tikzcd}
\end{equation*}
in which the vertical map is a homotopy equivalence by Lemma~\ref{lem:monoidal-str-homotopy-eq}. The Künneth theorem then implies that $\kappa_{Y}(F_{q}^{\theta})=0$ and that $\delta_{Y}(F_{q}^{\theta})$ identifies with a direct sum of coefficient systems of degrees $\le q-1$. Since finite-degree coefficient systems are closed under direct sums \cite[Lem.~1.6]{AnnexLS1}, it follows that $\delta_{Y}(F_{q}^{\theta})$ has degree at most $q-1$, and hence $F_{q}^{\theta}$ has degree $q$.

Finally, we show that $F_{q}^{\theta}$ satisfies Condition~\eqref{item:extension-mixed} of Definition~\ref{def:double-coeff-system}. For any $V\in \sH^{\ge 2}$, there is a commutative diagram of spaces
\[
\begin{tikzcd}
    \Bun^{\theta}(V;\ell_{0}) \arrow[r,"\tilde{\rho}^{1}_{V}"] \arrow[d,"\tilde{\sigma}^{1}_{V}"]
    & \Bun^{\theta}(V\# Y;\ell_{0}) \arrow[d,"\tilde{\sigma}^{1}_{V}"]
    \\
    \Bun^{\theta}(X\natural V;\ell_{0}) \arrow[r,"\tilde{\rho}^{1}_{V}"]
    & \Bun^{\theta}(X\natural V\# Y;\ell_{0}),
\end{tikzcd}
\]
whose commutativity is immediate from the above definitions of the maps $\tilde{\sigma}^{1}_{V}$ and $\tilde{\rho}^{1}_{V}$. Applying the functor $H_{q}(-;\bQ)$ gives the required commutative diagram in $\Ab$.
\end{proof}

\begin{rmk}
A careful inspection of the proof shows that the coefficient systems appearing in the bisystem $F_{q}^{\theta}$ are in fact \emph{split}. We shall not make use of this stronger statement, since it does not improve the stable range in Theorem~\ref{thmC-detailed}, and we have therefore omitted it from the argument.
\end{rmk}

\begin{rmk}\label{rmk:thmC-over-Q}
If one replaces the rational homology functor $H_{q}(-;\bQ) \colon \hTop \to \Ab$ with integral homology $H_{q}(-;\bZ)$, most of the argument still applies. Let us indicate where rational coefficients are essential.
The first issue arises from the use of exactness of the functor $- \otimes F_{j}(X)$, which need not hold over $\bZ$. This can be remedied by assuming that each $F_{i}$ is \emph{split} (see \cite[Lem.~1.7]{AnnexLS1}). However, a further obstruction appears: over $\bZ$, the Künneth theorem expresses $F_{q}^{\theta}(V \natural X)$ as a (non-naturally split) extension of split coefficient systems. Such extensions are not necessarily split, and this prevents the inductive step from going through.
\end{rmk}

\subsubsection{Proof of Theorem~\ref{thm:mainC}}

We now assemble the intermediate results to complete the proof of Theorem~\ref{thm:mainC}.

\begin{proof}[Proof of Theorem~\ref{thmC-detailed}]
For each $(g,r,s)$, let $E(g,r,s)$ denote the homological Serre spectral sequence associated to the fibre bundle~\eqref{eq:tang-str-fibre-bundle}. Its $E^{2}$-page has the form
\[
E(g,r,s)^{2}_{p,q}
=
H_{p}\left(\B\Diff(V_{g};R^{s}_{g,r});F_{q}^{\theta}(V^{s}_{g,r})\right)
\;\Longrightarrow\;
H_{p+q}\left(\B\Diff^{\theta}(V_{g};\ell_{0}(R^{s}_{g,r}));\bQ\right).
\]
By functoriality of the Serre spectral sequence with respect to maps of fibrations (see for instance \cite[Th.~24.5.1]{MayPonto}), the diagrams~\eqref{eq:fibration-map-sigma} and~\eqref{eq:fibration-map-mu} induce morphisms of spectral sequences $E(g,r,s)\to E(g+1,r,s)$ and $E(g,r,s)\to E(g,r+1,s)$ respectively. Recall that the homology of $\B\Diff(V_{g};R^{s}_{g,r})$ with coefficients in the local system $F_{q}^{\theta}(V^{s}_{g,r})$ agrees with the group homology of $\cH^{s}_{g,r}$ with coefficients in the corresponding representation (see Remark~\ref{rmk:classifying-space-handlebody}). Combining Proposition~\ref{prop:tangential-str-bisystem} with Theorem~\ref{thm:HS_stabilistation-1} and Corollary~\ref{cor:independence-of-marked-discs}, we deduce that the above maps are isomorphisms on the $E^{2}$-page in the range $p\le\frac{g-1}{2}-q$ (i.e.~$p+q\le\frac{g-3}{2}$). It follows that the induced maps on the abutments $(\breve{\sigma}^{\theta,s}_{g,r})_{*}$ and $(\breve{\mu}_{g,r}^{\theta,s})_{*}$ are isomorphisms in degrees $* \le \frac{g-3}{2}$, as claimed.
\end{proof}

\phantomsection
\addcontentsline{toc}{section}{References}
\renewcommand{\bibfont}{\normalfont\small}
\setlength{\bibitemsep}{0pt}
\printbibliography

@article {RWW,
    AUTHOR = {Randal-Williams, Oscar and Wahl, Nathalie},
     TITLE = {Homological stability for automorphism groups},
   JOURNAL = {Adv. Math.},
  FJOURNAL = {Advances in Mathematics},
    VOLUME = {318},
      YEAR = {2017},
     PAGES = {534--626},
      ISSN = {0001-8708},
   MRCLASS = {20J05},
  MRNUMBER = {3689750},
       DOI = {10.1016/j.aim.2017.07.022},
       URL = {https://doi.org/10.1016/j.aim.2017.07.022},
}

@book {farbmargalit,
    AUTHOR = {Farb, Benson and Margalit, Dan},
     TITLE = {A primer on mapping class groups},
    SERIES = {Princeton Mathematical Series},
    VOLUME = {49},
 PUBLISHER = {Princeton University Press, Princeton, NJ},
      YEAR = {2012},
     PAGES = {xiv+472},
      ISBN = {978-0-691-14794-9},
   MRCLASS = {57M50 (20F36 20F65 57M07 57N05)},
  MRNUMBER = {2850125},
MRREVIEWER = {Stephen P. Humphries},
}

@book {Weibel,
    AUTHOR = {Weibel, Charles A.},
     TITLE = {An introduction to homological algebra},
    SERIES = {Cambridge Studies in Advanced Mathematics},
    VOLUME = {38},
 PUBLISHER = {Cambridge University Press, Cambridge},
      YEAR = {1994},
     PAGES = {xiv+450},
   MRCLASS = {18-01 (16-01 17-01 20-01 55Uxx)},
  MRNUMBER = {1269324},
MRREVIEWER = {Kenneth A. Brown},
       DOI = {10.1017/CBO9781139644136},
       URL = {https://doi-org.ezproxy.lib.gla.ac.uk/10.1017/CBO9781139644136},
}

@article {graysonQuillen,
    AUTHOR = {Grayson, Daniel},
     TITLE = {Higher algebraic {$K$}-theory. {II} (after {D}aniel
              {Q}uillen)},
 BOOKTITLE = {Algebraic {$K$}-theory ({P}roc. {C}onf., {N}orthwestern
              {U}niv., {E}vanston, {I}ll., 1976)},
     PAGES = {217--240. Lecture Notes in Math., Vol. 551},
 PUBLISHER = {Springer, Berlin},
      YEAR = {1976},
   MRCLASS = {18F25},
  MRNUMBER = {0574096},
MRREVIEWER = {H. Bass},
}

@book {MacLane1,
    AUTHOR = {Mac Lane, Saunders},
     TITLE = {Categories for the working mathematician},
    SERIES = {Graduate Texts in Mathematics},
    VOLUME = {5},
   EDITION = {Second},
 PUBLISHER = {Springer-Verlag, New York},
      YEAR = {1998},
     PAGES = {xii+314},
      ISBN = {0-387-98403-8},
   MRCLASS = {18-02},
  MRNUMBER = {1712872},
}

@article {Gabriel,
    AUTHOR = {Gabriel, Pierre},
     TITLE = {Des cat\'{e}gories ab\'{e}liennes},
   JOURNAL = {Bull. Soc. Math. France},
  FJOURNAL = {Bulletin de la Soci\'{e}t\'{e} Math\'{e}matique de France},
    VOLUME = {90},
      YEAR = {1962},
     PAGES = {323--448},
      ISSN = {0037-9484},
   MRCLASS = {18.20},
  MRNUMBER = {0232821},
MRREVIEWER = {T.-Y. Lam},
       URL = {http://www.numdam.org/item?id=BSMF_1962__90__323_0},
}

@book {Hatcherbook,
    AUTHOR = {Hatcher, Allen},
     TITLE = {Algebraic topology},
 PUBLISHER = {Cambridge University Press, Cambridge},
      YEAR = {2002},
     PAGES = {xii+544},
   MRCLASS = {55-01 (55-00)},
  MRNUMBER = {1867354},
MRREVIEWER = {Donald W. Kahn},}

@article {HatcherWahl,
    AUTHOR = {Hatcher, Allen and Wahl, Nathalie},
     TITLE = {Stabilization for mapping class groups of 3-manifolds},
   JOURNAL = {Duke Math. J.},
  FJOURNAL = {Duke Mathematical Journal},
    VOLUME = {155},
      YEAR = {2010},
    NUMBER = {2},
     PAGES = {205--269},
      ISSN = {0012-7094,1547-7398},
   MRCLASS = {57M07 (20F28)},
  MRNUMBER = {2736166},
MRREVIEWER = {Mihalis\ A.\ Sykiotis},
       DOI = {10.1215/00127094-2010-055},
       URL = {https://doi.org/10.1215/00127094-2010-055},
}

@book {MilnorStasheff-CharClasses,
    AUTHOR = {Milnor, John W. and Stasheff, James D.},
     TITLE = {Characteristic classes},
    SERIES = {Annals of Mathematics Studies},
    VOLUME = {No. 76},
 PUBLISHER = {Princeton University Press, Princeton, NJ; University of Tokyo
              Press, Tokyo},
      YEAR = {1974},
     PAGES = {vii+331},
   MRCLASS = {57-01 (55-02 55F40 57D20)},
  MRNUMBER = {440554},
MRREVIEWER = {F.\ Hirzebruch},
}

@book {Arkowitz-HomotopyTheory,
    AUTHOR = {Arkowitz, Martin},
     TITLE = {Introduction to homotopy theory},
    SERIES = {Universitext},
 PUBLISHER = {Springer, New York},
      YEAR = {2011},
     PAGES = {xiv+344},
      ISBN = {978-1-4419-7328-3},
   MRCLASS = {55-02 (55Pxx)},
  MRNUMBER = {2814476},
MRREVIEWER = {Samuel B. Smith},
       DOI = {10.1007/978-1-4419-7329-0},
       URL = {https://doi.org/10.1007/978-1-4419-7329-0},
}

@book {Brown,
    AUTHOR = {Brown, Kenneth S.},
     TITLE = {Cohomology of groups},
    SERIES = {Graduate Texts in Mathematics},
    VOLUME = {87},
      NOTE = {Corrected reprint of the 1982 original},
 PUBLISHER = {Springer-Verlag, New York},
      YEAR = {1994},
     PAGES = {x+306},
      ISBN = {0-387-90688-6},
   MRCLASS = {20J05 (20-02)},
  MRNUMBER = {1324339},
}

@article {Richards,
    AUTHOR = {Richards, Ian},
     TITLE = {On the classification of noncompact surfaces},
   JOURNAL = {Trans. Amer. Math. Soc.},
  FJOURNAL = {Transactions of the American Mathematical Society},
    VOLUME = {106},
      YEAR = {1963},
     PAGES = {259--269},
      ISSN = {0002-9947},
   MRCLASS = {54.75},
  MRNUMBER = {143186},
MRREVIEWER = {S. S. Cairns},
       DOI = {10.2307/1993768},
       URL = {https://doi-org.proxy-scd.u-bourgogne.fr/10.2307/1993768},
}

@incollection {Ivanov,
    AUTHOR = {Ivanov, Nikolai V.},
     TITLE = {On the homology stability for {T}eichm\"{u}ller modular groups:
              closed surfaces and twisted coefficients},
 BOOKTITLE = {Mapping class groups and moduli spaces of {R}iemann surfaces
              ({G}\"{o}ttingen, 1991/{S}eattle, {WA}, 1991)},
    SERIES = {Contemp. Math.},
    VOLUME = {150},
     PAGES = {149--194},
 PUBLISHER = {Amer. Math. Soc., Providence, RI},
      YEAR = {1993},
   MRCLASS = {57N05 (20F38 30F60 32G15 57M99)},
  MRNUMBER = {1234264},
MRREVIEWER = {Darryl McCullough},
       DOI = {10.1090/conm/150/01290},
       URL = {https://doi-org.ezproxy.lib.gla.ac.uk/10.1090/conm/150/01290},
}

@article {RWautfreegroups,
    AUTHOR = {Randal-Williams, Oscar},
     TITLE = {Cohomology of automorphism groups of free groups with twisted
              coefficients},
   JOURNAL = {Selecta Math. (N.S.)},
  FJOURNAL = {Selecta Mathematica. New Series},
    VOLUME = {24},
      YEAR = {2018},
    NUMBER = {2},
     PAGES = {1453--1478},
      ISSN = {1022-1824},
   MRCLASS = {20F28 (20J06 57R20)},
  MRNUMBER = {3782426},
MRREVIEWER = {Valeriy G. Bardakov},
       DOI = {10.1007/s00029-017-0311-0},
       URL = {https://doi.org/10.1007/s00029-017-0311-0},
}

@article {Boldsen,
    AUTHOR = {Boldsen, S\o ren K.},
     TITLE = {Improved homological stability for the mapping class group
              with integral or twisted coefficients},
   JOURNAL = {Math. Z.},
  FJOURNAL = {Mathematische Zeitschrift},
    VOLUME = {270},
      YEAR = {2012},
    NUMBER = {1-2},
     PAGES = {297--329},
      ISSN = {0025-5874},
   MRCLASS = {57M50 (55T05 57N05)},
  MRNUMBER = {2875835},
MRREVIEWER = {Thomas Koberda},
       DOI = {10.1007/s00209-010-0798-y},
       URL = {https://doi.org/10.1007/s00209-010-0798-y},
}

@incollection {CohenMadsen,
    AUTHOR = {Cohen, Ralph L. and Madsen, Ib},
     TITLE = {Surfaces in a background space and the homology of mapping
              class groups},
 BOOKTITLE = {Algebraic geometry---{S}eattle 2005. {P}art 1},
    SERIES = {Proc. Sympos. Pure Math.},
    VOLUME = {80},
     PAGES = {43--76},
 PUBLISHER = {Amer. Math. Soc., Providence, RI},
      YEAR = {2009},
   MRCLASS = {57M07 (55P47)},
  MRNUMBER = {2483932},
MRREVIEWER = {Laurence R. Taylor},
       DOI = {10.1090/pspum/080.1/2483932},
       URL = {https://doi.org/10.1090/pspum/080.1/2483932},
}

@unpublished{Barkan-Steinebrunner,
    AUTHOR = {Barkan, Shaul and Steinebrunner, Jan},
     TITLE = {???},
      YEAR = {2024},
      NOTE = {Preprint soon to appear}, 
}

@article {Galatius,
    AUTHOR = {Galatius, S\o ren},
     TITLE = {Mod {$p$} homology of the stable mapping class group},
   JOURNAL = {Topology},
  FJOURNAL = {Topology. An International Journal of Mathematics},
    VOLUME = {43},
      YEAR = {2004},
    NUMBER = {5},
     PAGES = {1105--1132},
      ISSN = {0040-9383},
   MRCLASS = {57M50 (55P47)},
  MRNUMBER = {2079997},
       DOI = {10.1016/j.top.2004.01.011},
       URL = {https://doi-org.ezproxy.lib.gla.ac.uk/10.1016/j.top.2004.01.011},
}

@article {Perlmutter,
    AUTHOR = {Perlmutter, Nathan},
     TITLE = {Homological stability for diffeomorphism groups of
              high-dimensional handlebodies},
   JOURNAL = {Algebr. Geom. Topol.},
  FJOURNAL = {Algebraic \& Geometric Topology},
    VOLUME = {18},
      YEAR = {2018},
    NUMBER = {5},
     PAGES = {2769--2820},
      ISSN = {1472-2747},
   MRCLASS = {57R65 (57R15 57R50 57S05)},
  MRNUMBER = {3848399},
MRREVIEWER = {Sam Nariman},
       DOI = {10.2140/agt.2018.18.2769},
       URL = {https://doi.org/10.2140/agt.2018.18.2769},
}

@article {Botvinnik-Perlmutter,
    AUTHOR = {Botvinnik, Boris and Perlmutter, Nathan},
     TITLE = {Stable moduli spaces of high-dimensional handlebodies},
   JOURNAL = {J. Topol.},
  FJOURNAL = {Journal of Topology},
    VOLUME = {10},
      YEAR = {2017},
    NUMBER = {1},
     PAGES = {101--163},
      ISSN = {1753-8416},
   MRCLASS = {57R65 (57R15 57R50 57R90 57S05)},
  MRNUMBER = {3653064},
MRREVIEWER = {Nikolai N. Saveliev},
       DOI = {10.1112/topo.12003},
       URL = {https://doi.org/10.1112/topo.12003},
}

@article{IshidaSato,
author = {Tomohiko Ishida and Masatoshi Sato},
title = {{A twisted first homology group of the handlebody mapping class group}},
volume = {54},
journal = {Osaka Journal of Mathematics},
number = {3},
publisher = {Osaka University and Osaka Metropolitan University, Departments of Mathematics},
pages = {587 -- 619},
year = {2017},
}

@article {Krannich,
    AUTHOR = {Krannich, Manuel},
     TITLE = {Homological stability of topological moduli spaces},
   JOURNAL = {Geom. Topol.},
  FJOURNAL = {Geometry \& Topology},
    VOLUME = {23},
      YEAR = {2019},
    NUMBER = {5},
     PAGES = {2397--2474},
      ISSN = {1465-3060,1364-0380},
   MRCLASS = {55P48 (55R40 55R80 57R19 57R50)},
  MRNUMBER = {4019896},
MRREVIEWER = {Maria\ Basterra},
       DOI = {10.2140/gt.2019.23.2397},
       URL = {https://doi.org/10.2140/gt.2019.23.2397},
}

@article {Ivanov2,
    AUTHOR = {Ivanov, N. V.},
     TITLE = {Groups of diffeomorphisms of {W}aldhausen manifolds},
      NOTE = {Studies in topology, II},
   JOURNAL = {Zap. Nau\v{c}n. Sem. Leningrad. Otdel. Mat. Inst. Steklov.
              (LOMI)},
  FJOURNAL = {Zapiski Nau\v{c}nyh Seminarov Leningradskogo Otdelenija
              Matemati\v{c}eskogo Instituta im. V. A. Steklova Akademii Nauk
              SSSR (LOMI)},
    VOLUME = {66},
      YEAR = {1976},
     PAGES = {172--176, 209},
   MRCLASS = {57D40 (58D10)},
  MRNUMBER = {448370},
MRREVIEWER = {A.\ Chernavski\u{\i}},
}

@article {Hatcher3manif,
    AUTHOR = {Hatcher, Allen},
     TITLE = {Homeomorphisms of sufficiently large {$P\sp{2}$}-irreducible
              {$3$}-manifolds},
   JOURNAL = {Topology},
  FJOURNAL = {Topology. An International Journal of Mathematics},
    VOLUME = {15},
      YEAR = {1976},
    NUMBER = {4},
     PAGES = {343--347},
      ISSN = {0040-9383},
   MRCLASS = {57A10},
  MRNUMBER = {420620},
MRREVIEWER = {F.\ Laudenbach},
       DOI = {10.1016/0040-9383(76)90027-6},
       URL = {https://doi.org/10.1016/0040-9383(76)90027-6},
}

@article {HatcherSmale,
    AUTHOR = {Hatcher, Allen E.},
     TITLE = {A proof of the {S}male conjecture, {${\rm Diff}(S\sp{3})\simeq
              {\rm O}(4)$}},
   JOURNAL = {Ann. of Math. (2)},
  FJOURNAL = {Annals of Mathematics. Second Series},
    VOLUME = {117},
      YEAR = {1983},
    NUMBER = {3},
     PAGES = {553--607},
      ISSN = {0003-486X,1939-8980},
   MRCLASS = {57M99 (57S05)},
  MRNUMBER = {701256},
MRREVIEWER = {R.\ C.\ Kirby},
       DOI = {10.2307/2007035},
       URL = {https://doi.org/10.2307/2007035},
}

@book {MayPonto,
    AUTHOR = {May, J. P. and Ponto, K.},
     TITLE = {More concise algebraic topology},
    SERIES = {Chicago Lectures in Mathematics},
      NOTE = {Localization, completion, and model categories},
 PUBLISHER = {University of Chicago Press, Chicago, IL},
      YEAR = {2012},
     PAGES = {xxviii+514},
   MRCLASS = {55-02 (16T05 18G55 55P60)},
  MRNUMBER = {2884233},
MRREVIEWER = {Ismar Voli\'{c}},
}

@unpublished{Hatcherincompressible,
    AUTHOR = {Hatcher, Allen},
     TITLE = {Spaces of Incompressible Surfaces},
      YEAR = {1999},
      NOTE = {ArXiv: \href{https://arxiv.org/abs/math/9906074}{9906074}},
}

@misc{AnnexLS1,
    title        = {Annex file to the present article},
    Author = {Lindell, Erik and Souli{\'e}, Arthur},
    howpublished = {Available \href{https://drive.google.com/file/d/13E8QrVAup85L_wgNxIWXYLrWKSAiMFUr/view}{here}},
    year         = {2025},
}

@misc{LS1,
    title        = {Twisted homological stability for handlebody mapping class groups},
    Author = {Lindell, Erik and Souli{\'e}, Arthur},
    howpublished = {URL \href{???}{???}},
    year         = {2025},
}

@unpublished{LS2,
  author = {Lindell, Erik and Souli{\'e}, Arthur},
  date = {2026},
  pubstate = {inpreparation},
  title = {Stable twisted cohomology of handlebody mapping class groups},
}

@article {Bonahon,
    AUTHOR = {Bonahon, Francis},
     TITLE = {Cobordism of automorphisms of surfaces},
   JOURNAL = {Ann. Sci. \'Ecole Norm. Sup. (4)},
  FJOURNAL = {Annales Scientifiques de l'\'Ecole Normale Sup\'erieure.
              Quatri\`eme S\'erie},
    VOLUME = {16},
      YEAR = {1983},
    NUMBER = {2},
     PAGES = {237--270},
      ISSN = {0012-9593},
   MRCLASS = {57N05 (57N10)},
  MRNUMBER = {732345},
MRREVIEWER = {Klaus\ Johannson},
       URL = {http://www.numdam.org/item?id=ASENS_1983_4_16_2_237_0},
}

@article {GalatiusRW-StableCoh,
    AUTHOR = {Galatius, S\o ren and Randal-Williams, Oscar},
     TITLE = {Stable moduli spaces of high-dimensional manifolds},
   JOURNAL = {Acta Math.},
  FJOURNAL = {Acta Mathematica},
    VOLUME = {212},
      YEAR = {2014},
    NUMBER = {2},
     PAGES = {257--377},
      ISSN = {0001-5962},
   MRCLASS = {55R40 (55Pxx 57Mxx)},
  MRNUMBER = {3207759},
MRREVIEWER = {Semen S. Podkorytov},
       DOI = {10.1007/s11511-014-0112-7},
       URL = {https://doi.org/10.1007/s11511-014-0112-7},
}

@article {GalatiusRW-HomStab,
    AUTHOR = {Galatius, S\o ren and Randal-Williams, Oscar},
     TITLE = {Homological stability for moduli spaces of high dimensional
              manifolds. {I}},
   JOURNAL = {J. Amer. Math. Soc.},
  FJOURNAL = {Journal of the American Mathematical Society},
    VOLUME = {31},
      YEAR = {2018},
    NUMBER = {1},
     PAGES = {215--264},
      ISSN = {0894-0347},
   MRCLASS = {57R90 (55P47 57R15 57R56)},
  MRNUMBER = {3718454},
MRREVIEWER = {Sam Nariman},
       DOI = {10.1090/jams/884},
       URL = {https://doi.org/10.1090/jams/884},
}

@article {HarrVistrupWahl-DisorderedArcs,
    AUTHOR = {Harr, Oscar and Vistrup, Max and Wahl, Nathalie},
     TITLE = {Disordered arcs and {H}arer stability},
   JOURNAL = {High. Struct.},
  FJOURNAL = {Higher Structures},
    VOLUME = {8},
      YEAR = {2024},
    NUMBER = {1},
     PAGES = {193--223},
   MRCLASS = {57K20 (57M07 57R50)},
  MRNUMBER = {4752520},
}

@article {PalmerWu,
    AUTHOR = {Palmer, Martin and Wu, Xiaolei},
     TITLE = {On the homology of big mapping class groups},
   JOURNAL = {J. Topol.},
  FJOURNAL = {Journal of Topology},
    VOLUME = {17},
      YEAR = {2024},
    NUMBER = {4},
     PAGES = {Paper No. e12358, 41},
      ISSN = {1753-8416,1753-8424},
   MRCLASS = {57K20 (20J06)},
  MRNUMBER = {4806781},
MRREVIEWER = {Mahender\ Singh},
       DOI = {10.1112/topo.12358},
       URL = {https://doi.org/10.1112/topo.12358},
}

@article {PalmerSoulie,
    AUTHOR = {Palmer, Martin and Souli\'e, Arthur},
     TITLE = {Topological representations of motion groups and mapping class
              groups---a unified functorial construction},
   JOURNAL = {Ann. H. Lebesgue},
  FJOURNAL = {Annales Henri Lebesgue},
    VOLUME = {7},
      YEAR = {2024},
     PAGES = {409--519},
      ISSN = {2644-9463},
   MRCLASS = {20F38 (18B40 20F36 55R80 57K20 57M07)},
  MRNUMBER = {4799903},
}

@article {EarleEells1,
    AUTHOR = {Earle, C. J. and Eells, J.},
     TITLE = {The diffeomorphism group of a compact {R}iemann surface},
   JOURNAL = {Bull. Amer. Math. Soc.},
  FJOURNAL = {Bulletin of the American Mathematical Society},
    VOLUME = {73},
      YEAR = {1967},
     PAGES = {557--559},
      ISSN = {0002-9904},
   MRCLASS = {57.55 (30.00)},
  MRNUMBER = {212840},
MRREVIEWER = {L.\ Keen},
       DOI = {10.1090/S0002-9904-1967-11746-4},
       URL = {https://doi.org/10.1090/S0002-9904-1967-11746-4},
}

@article {EarleEells2,
    AUTHOR = {Earle, Clifford J. and Eells, James},
     TITLE = {A fibre bundle description of {T}eichm\"uller theory},
   JOURNAL = {J. Differential Geometry},
  FJOURNAL = {Journal of Differential Geometry},
    VOLUME = {3},
      YEAR = {1969},
     PAGES = {19--43},
      ISSN = {0022-040X,1945-743X},
   MRCLASS = {57.47},
  MRNUMBER = {276999},
       URL = {http://projecteuclid.org/euclid.jdg/1214428816},
}

@article {EarleSchatz,
    AUTHOR = {Earle, C. J. and Schatz, A.},
     TITLE = {Teichm\"{u}ller theory for surfaces with boundary},
   JOURNAL = {J. Differential Geometry},
  FJOURNAL = {Journal of Differential Geometry},
    VOLUME = {4},
      YEAR = {1970},
     PAGES = {169--185},
      ISSN = {0022-040X},
   MRCLASS = {57.47},
  MRNUMBER = {277000},
MRREVIEWER = {J. N. Mather},
       URL = {http://projecteuclid.org/euclid.jdg/1214429381},
}

@book {Mitchell,
    AUTHOR = {Mitchell, Barry},
     TITLE = {Theory of categories},
    SERIES = {Pure and Applied Mathematics},
    VOLUME = {Vol. XVII},
 PUBLISHER = {Academic Press, New York-London},
      YEAR = {1965},
     PAGES = {xi+273},
   MRCLASS = {18.00},
  MRNUMBER = {202787},
MRREVIEWER = {F.\ E. J. Linton},
}

@article {Brown-Messer,
    AUTHOR = {Brown, Edward M. and Messer, Robert},
     TITLE = {The classification of two-dimensional manifolds},
   JOURNAL = {Trans. Amer. Math. Soc.},
  FJOURNAL = {Transactions of the American Mathematical Society},
    VOLUME = {255},
      YEAR = {1979},
     PAGES = {377--402},
      ISSN = {0002-9947},
   MRCLASS = {57N05 (57M20)},
  MRNUMBER = {542887},
MRREVIEWER = {V. A. Kalinin},
       DOI = {10.2307/1998182},
       URL = {https://doi.org/10.2307/1998182},
}

@book {Borceux1,
    AUTHOR = {Borceux, Francis},
     TITLE = {Handbook of categorical algebra. 1},
    SERIES = {Encyclopedia of Mathematics and its Applications},
    VOLUME = {50},
      NOTE = {Basic category theory},
 PUBLISHER = {Cambridge University Press, Cambridge},
      YEAR = {1994},
     PAGES = {xvi+345},
      ISBN = {0-521-44178-1},
   MRCLASS = {18-02 (18Axx)},
  MRNUMBER = {1291599},
MRREVIEWER = {Martin\ Hyland},
}

@book {Awodey,
    AUTHOR = {Awodey, Steve},
     TITLE = {Category theory},
    SERIES = {Oxford Logic Guides},
    VOLUME = {52},
   EDITION = {Second},
 PUBLISHER = {Oxford University Press, Oxford},
      YEAR = {2010},
     PAGES = {xvi+311},
      ISBN = {978-0-19-923718-0},
   MRCLASS = {18-01 (03G30 18A25)},
  MRNUMBER = {2668552},
}

@book {Richter,
    AUTHOR = {Richter, Birgit},
     TITLE = {From categories to homotopy theory},
    SERIES = {Cambridge Studies in Advanced Mathematics},
    VOLUME = {188},
 PUBLISHER = {Cambridge University Press, Cambridge},
      YEAR = {2020},
     PAGES = {x+390},
      ISBN = {978-1-108-47962-2},
   MRCLASS = {18-01 (18Nxx 55U40)},
  MRNUMBER = {4411367},
MRREVIEWER = {Daniel\ Dugger},
       DOI = {10.1017/9781108855891},
       URL = {https://doi.org/10.1017/9781108855891},
}

@book {Lee,
    AUTHOR = {Lee, John M.},
     TITLE = {Introduction to smooth manifolds},
    SERIES = {Graduate Texts in Mathematics},
    VOLUME = {218},
   EDITION = {Second},
 PUBLISHER = {Springer, New York},
      YEAR = {2013},
     PAGES = {xvi+708},
      ISBN = {978-1-4419-9981-8},
   MRCLASS = {58-01 (53-01 57-01)},
  MRNUMBER = {2954043},
}

@book {Spanier,
    AUTHOR = {Spanier, Edwin H.},
     TITLE = {Algebraic topology},
      NOTE = {Corrected reprint of the 1966 original},
 PUBLISHER = {Springer-Verlag, New York},
      YEAR = {[1995?]},
     PAGES = {xvi+528},
      ISBN = {0-387-94426-5},
   MRCLASS = {55-01},
  MRNUMBER = {1325242},
}

@book {Dold,
    AUTHOR = {Dold, Albrecht},
     TITLE = {Lectures on algebraic topology},
    SERIES = {Classics in Mathematics},
      NOTE = {Reprint of the 1972 edition},
 PUBLISHER = {Springer-Verlag, Berlin},
      YEAR = {1995},
     PAGES = {xii+377},
      ISBN = {3-540-58660-1},
   MRCLASS = {55-02 (01A75)},
  MRNUMBER = {1335915},
       DOI = {10.1007/978-3-642-67821-9},
       URL = {https://doi.org/10.1007/978-3-642-67821-9},
}

\noindent {Erik Lindell, \itshape Institut for Matematiske Fag, Københavns Universitet, Universitetsparken 5, 2100
København Ø, Denmark.} \noindent {Email address: \tt erikjlindell@gmail.com, ejl@math.ku.dk}

\noindent {Arthur Souli{\'e}, \itshape Normandie Univ., UNICAEN, CNRS, LMNO, 14000 Caen, France.} \noindent {Email address: \tt artsou@hotmail.fr, arthur.soulie@unicaen.fr, arthur.soulie@cnrs.fr}

\end{document}